\documentclass{amsart}
\usepackage[all]{xy}
\usepackage{amsmath,amsfonts,amssymb,amsthm,epsfig,amscd,psfrag,framed,comment}

\usepackage[usenames,dvipsnames]{xcolor}

\usepackage{hyperref}
\hypersetup{colorlinks,  citecolor=blue,  linkcolor=BrickRed, urlcolor=black}

\newtheorem{Theorem}{Theorem}[section]
\newtheorem{Proposition}[Theorem]{Proposition} 
\newtheorem{Lemma}[Theorem]{Lemma}

\newtheorem{Corollary}[Theorem]{Corollary}

\theoremstyle{definition}

\newtheorem{Remark}[Theorem]{Remark}

\newcommand{\qbin}[2]{
\left[
 \begin{array}{c}
 #1 \\
 #2 \\
 \end{array}
 \right]
}

\def\Kom{{\rm{Kom}}}
\def\hg{\hat{\mathfrak{g}}}
\def\Gr{{\mathrm{Gr}}}

\def\uck{{\underline{ck}}}
\def\h{\mathfrak{h}}
\def\H{\mathfrak{H}}
\def\A{{\sf{A}}}
\def\B{{\sf{B}}}
\def\X{{\sf{X}}}
\def\Y{{\sf{Y}}}
\def\t{\theta}
\def\({{\rm (}}
\def\){{\rm )}}

\def\eps{\epsilon}

\def\1{{\mathbf{1}}}

\newcommand{\N}{\mathbb{N}}
\newcommand{\Z}{\mathbb{Z}}
\newcommand{\C}{\mathbb{C}}
\def\adj{{{adj}}}
\def\g{\mathfrak{g}}

\def\la{\langle}
\def\ra{\rangle}

\def\l{\lambda}
\def\id{{\mathrm{id}}}
\def\k{{\Bbbk}}

\def\sl{{\mathfrak{sl}}}

\def\E{{\sf{E}}}
\def\F{{\sf{F}}}

\def\sC{{\mathcal{C}}}
\def\sF{{\mathcal{F}}}

\def\tY{{\tilde{Y}}}

\def\O{{\mathcal O}}
\def\sE{{\mathcal{E}}}
\def\H{{\mathcal{H}}}

\def\dim{{\rm{dim}}}

\newcommand{\K}{\mathcal{K}}

\DeclareMathOperator{\Hom}{Hom}

\DeclareMathOperator{\End}{End}

\begin{document}
\setcounter{tocdepth}{1}

\title{Rigidity in higher representation theory}

\author{Sabin Cautis}
\email{cautis@math.ubc.ca}
\address{Department of Mathematics\\ University of British Columbia \\ Vancouver, Canada}

\begin{abstract}
We describe a categorical $\g$ action, called a $(\g,\t)$ action, which is easier to check in practice. Most categorical $\g$ actions can be shown to be of this form. The main result is that a $(\g,\t)$ action carries actions of quiver Hecke algebras (KLR algebras). We discuss applications of this fact to categorical vertex operators, affine Grassmannians (or Nakajima quiver varieties) and to homological knot invariants. 
\end{abstract}

\maketitle
\tableofcontents

\section{Introduction}

The higher representation theory of Kac-Moody Lie algebra $\g$ involves the action of $U_q(\g)$ on categories. This means that to each weight $\l$ of $\g$ one assigns an additive (graded) category $\sC(\l)$ and to generators $E_i$ and $F_i$ of $U_q(\g)$ one assigns functors $\E_i: \sC(\l) \rightarrow \sC(\l+\alpha_i)$ and $\F_i: \sC(\l+\alpha_i) \rightarrow \sC(\l)$. These functors are then required to satisfy certain relations analogous to those in $U_q(\g)$. For example, the relation $[E_i,F_i]= \frac{K_i-K_i^{-1}}{q-q^{-1}}$ becomes 
\begin{align}
\label{EQ:intro1}
\E_i \F_i |_{\sC(\l)} \cong \F_i \E_i |_{\sC(\l)} \bigoplus_{[\la \l, \alpha_i \ra]} \id_{\sC(\l)} \ \  &\text{ if } \la \l, \alpha_i \ra \ge 0 \\
\label{EQ:intro2}
\F_i \E_i |_{\sC(\l)} \cong \E_i \F_i |_{\sC(\l)} \bigoplus_{[-\la \l, \alpha_i \ra]} \id_{\sC(\l)} \ \  &\text{ if } \la \l, \alpha_i \ra \le 0.
\end{align}

Since such categorical actions involve functors one would like to understand their natural transformations. For instance, one might wonder about $\End(\E_i \E_i)$ or $\End(\E_i \E_j)$ or any other composition of $\E$'s and $\F$'s. Moreover, these natural transformations should induce the isomorphisms in (\ref{EQ:intro1}) and (\ref{EQ:intro2}). 

One answer to this question is given by Khovanov-Lauda \cite{KL1,KL2,KL3} and Chuang-Rouquier \cite{CR,R}. They suggest that one should require the endomorphism algebras of composition of $\E$'s to be certain Hecke quiver algebras (or KLR algebras). More precisely, Khovanov-Lauda show that these algebras, together with several other explicit natural transformations relating $\E$'s and $\F$'s, determine all the necessary isomorphisms of functors such as the ones in (\ref{EQ:intro1}) and (\ref{EQ:intro2}). In this sense this gives a categorification of quantum groups by replacing isomorphisms of functors with equalities of natural transformations. 

On the other hand, suppose one is given a ``naive'' categorical action of $\g$. This means that we have categories $\sC(\l)$ and functors $\E_i,\F_i$ together with the knowledge that certain compositions of functors, such as those in (\ref{EQ:intro1}) or (\ref{EQ:intro2}), are isomorphic. There are many examples of such actions (see section \ref{sec:app} for a couple). In such cases there is little reason to expect the quiver Hecke algebras to act. In other words, one would expect the space of natural transformations, between say compositions of $\E$'s, to depend on the choice of categories $\sC(\l)$.

The purpose of this paper is to show that, under some mild conditions, a naive categorical action of $\g$ carries an action of the quiver Hecke algebras by natural transformations. This is a rigidity result for the higher representation theory of Kac-Moody Lie algebras $\g$ because it implies that the spaces of natural transformations are determined by $\g$. 

We do not know if such rigidity phenomena are common. For example, it is not clear the extent to which similar results hold for the higher representation theory of the Heisenberg algebras from \cite{CLi1}. 

\subsection{Main results}

In section \ref{sec:gaction} we introduce the concept of a $(\g,\t)$ action. Roughly this is a naive categorical $\g$ action together with another piece of data $\t$. The main result of this paper, Theorem \ref{THM:main}, states the following. 

\begin{Theorem}
A $(\g,\t)$ action carries an action of a quiver Hecke algebra (modulo transient maps).
\end{Theorem}
\begin{Remark}
Transient maps are certain negligible 2-morphisms as defined in section \ref{sec:Xi}. If $\g=\sl_n
$ then this result holds without having to mod out by transient maps (section \ref{sec:transients}). 
\end{Remark}

The definition of a $(\g,\t)$ action is designed to be as minimal as possible so that it is easier to check in examples. For example, we do not require the existence of divided powers $\E_i^{(r)}$ and $\F_i^{(r)}$. These only appear {\it a posteriori} as a consequence of the affine nilHecke relations. We also do not require the Serre relation nor the commutativity relation $\E_i \E_j \cong \E_j \E_i$ when $\la \alpha_i, \alpha_j \ra = 0$. 

The proof of Theorem \ref{THM:main} is largely based on a sequence of $\Hom$-space calculations. For example, one can show by using adjunction together with isomorphisms (\ref{EQ:intro1}) and (\ref{EQ:intro2}) that $\Hom(\E_i \E_j, \E_j \E_i)$ is one-dimensional if $\la \alpha_i, \alpha_j \ra = 0$. Thus, up to rescaling, this gives us a map $T_{ij}: \E_i \E_j \rightarrow \E_j \E_i$ which subsequently turns out to be the isomorphism $\E_i \E_j \cong \E_j \E_i$. These $\Hom$-space calculations are performed in a series of Lemmas in the appendix. They reflect a certain rigidity of categorical $\g$ actions. 
 
Finally, Theorem \ref{THM:main} together with the main result of \cite{CLa} implies the following. 
\begin{Corollary}\label{COR:main}
A $(\g,\t)$ action induces a 2-representation in the sense of Khovanov-Lauda (modulo transient maps). 
\end{Corollary}
\begin{Remark}
As before, if $\g=\sl_n$ then the condition on transient maps is not necessary. 
\end{Remark}
Recall that a 2-representation in this sense is a 2-functor $\dot{\mathcal{U}}_Q(\g) \rightarrow \K$. Here $\dot{\mathcal{U}}_Q(\g)$ is the 2-category defined by Khovanov-Lauda which categorifies $\g$ while $\K$ is the target category. 

\subsection{Applications}\label{sec:app}

In section \ref{sec:applications} we discuss three applications of Theorem \ref{THM:main}. 

The first application (section \ref{sec:vertex}) is to categorical vertex operators. In \cite{CLi2} we construct a $(\hg,\t)$ action on the homotopy category of a categorification $\sF_\Gamma$ of the Fock space for a particular Heisenberg algebra $\h_\Gamma$. Corollary \ref{COR:main} implies that this action can be lifted to a 2-representation of $\dot{\mathcal{U}}_Q(\hg)$ (in this case there are no transient maps to worry about). 

It is interesting to note that in this case it is very difficult to work with (or even define) divided powers $\E_i^{(k)}$ and $\F_i^{(k)}$. Likewise it is very difficult to check the Serre relation. Subsequently the computations in \cite{CLi2} are made tractable because a $(\g,\t)$ action does not require checking either of these relations. 

A second application is to geometric categorical $\g$ actions (section \ref{sec:quiver}). We show that such a geometric action, defined in \cite{CK3}, induces a $(\g,\t)$ action (essentially, a geometric categorical $\g$ action is a geometric way to define a $(\g,\t)$ action). This implies that the quiver Hecke algebras act on the 2-category $\K^n_{\Gr,m}$ constructed in \cite{CKL1,C} using coherent sheaves on certain varieties $Y(\l)$. These $Y(\l)$ are convolution varieties obtained from the affine Grassmannian of ${\rm PGL}_m$. As in the previous example, this action of the quiver Hecke algebras is difficult to see directly.

For our final application (section \ref{sec:knots}) we discuss how the rigidity of categorical $\g$ actions implies a rigidity for knot homologies. More precisely, in \cite{C} we explained how to use categorical $(\sl_\infty,\t)$ actions to define homological knot invariants. By Corollary \ref{COR:main} we know any such action lifts to a 2-representation of $\dot{\mathcal{U}}_Q(\sl_\infty)$. Subsequently this knot invariant does not depend on the particular 2-representation we choose. In particular, this means that various homological knot invariants defined by very different means (coherent sheaves, category $\O$, matrix factorizations, foams, etc.) must be isomorphic. 

\vspace{.5cm}
\noindent Acknowledgments. Research was supported by NSF grant DMS-1101439, NSERC grant and the Templeton foundation.

\section{$(\g,\t)$ actions}

\subsection{Notation} 

Fix a base field $\k$, which is not assumed to be of characteristic 0, nor algebraically closed. Let $\Gamma$ be a connected graph without multiple edges or loops and with finite vertex set $I$ (i.e. simply laced Dynkin diagram). In addition, fix the following data.
\begin{itemize}
\item a free $\Z$ module $X$  (the weight lattice),
\item for $i \in I$ an element $\alpha_i \in X$ (simple roots),
\item for $i \in I$ an element $\Lambda_i \in X$ (fundamental weights),
\item a symmetric non-degenerate bilinear form $\la \cdot,\cdot \ra$ on $X$.
\end{itemize}
These data should satisfy:
\begin{itemize}
\item the set $\{\alpha_i\}_{i \in I}$ is linearly independent,
\item $C_{i,j} = \la \alpha_i, \alpha_j \ra$ (the Cartan matrix) so that $\la \alpha_i, \alpha_i \ra = 2$ and for $i \neq j$, $\la \alpha_i, \alpha_j \ra \in \{0,-1\}$ depending on whether or not $i,j \in I$ are joined by an edge,
\item $\la \Lambda_i, \alpha_j \ra = \delta_{i,j}$ for all $i, j \in I$.
\end{itemize}
We will often abbreviate $\la \l, \alpha_i \ra = \l_i$ and $\la \alpha_i, \alpha_j \ra = \la i,j \ra$. The root lattice will be denoted $Y$ and $Y_\k := Y \otimes_\Z \k$. Associated to a Cartan datum we fix a choice of scalars $Q$ consisting of $t_{ij} \in \k^{\times}$ for all $i \ne j \in I$ such that $t_{ij} = t_{ji}$ if $\la i,j \ra = 0$. 

\subsection{Definition of $(\g,\t)$ actions}\label{sec:gaction}

Associated to $\Gamma$ we have a Kac-Moody Lie algebra $\g$. A $(\g,\t)$ action consists of a target graded, additive, $\k$-linear idempotent complete 2-category $\K$ where the objects (0-morphisms) are indexed by $\l \in X$ and equipped with 
\begin{enumerate}
\item 1-morphisms: $\E_i \1_\l = \1_{\l+\alpha_i} \E_i$ and $\F_i \1_{\l+\alpha_i} = \1_\l \F_i$ where $\1_\l$ is the identity 1-morphism of $\l$.
\item 2-morphisms: for each $\l \in X$, a linear map $Y_\k \rightarrow \End^2(\1_\l)$. 
\end{enumerate}
\begin{Remark}
We will abuse notation and denote by $\theta \in \End^2(\1_\l)$ the image of $\theta \in Y_\k$ under the linear maps above. 
\end{Remark}
By a graded 2-category we mean a 2-category whose 1-morphisms are equipped with an auto-equivalence $\la 1 \ra$. $\K$ is idempotent complete if for any 2-morphism $f$ with $f^2=f$ the image of $f$ is contained in $\K$. 

On this data we impose the following conditions.
\begin{enumerate}
\item \label{co:hom1} $\Hom(\1_\l, \1_\l \la l \ra)$ is zero if $l < 0$ and one-dimensional if $l=0$ and $\1_\l \ne 0$. Moreover, the space of maps between any two 1-morphisms is finite dimensional.
\item \label{co:adj} $\E_i$ and $\F_i$ are left and right adjoints of each other up to specified shifts. More precisely
\begin{enumerate}
\item $(\E_i \1_\l)_R \cong \1_\l \F_i \la \l_i + 1 \ra$
\item $(\E_i \1_\l)_L \cong \1_\l \F_i \la - \l_i -1 \ra$.
\end{enumerate}

\item \label{co:EF} We have
\begin{align*}
\E_i \F_i \1_\l &\cong \F_i \E_i \1_{\l} \bigoplus_{[\l_i]} \1_\l \ \ \text{ if } \l_i \ge 0 \\
\F_i \E_i \1_{\l} &\cong \E_i \F_i \1_{\l} \bigoplus_{[-\l_i]} \1_\l \ \ \text{ if } \l_i \le 0
\end{align*}

\item \label{co:EiFj} If $i \ne j \in I$ then $\F_j \E_i \1_\l \cong \E_i \F_j \1_\l$.

\item \label{co:theta} If $\l_i \ge 0$ then map $(I \t I) \in \End^2(\E_i \1_\l \F_i)$ induces an isomorphism between $\l_i+1$ (resp. zero) of the $\l_i+2$ summands $\1_{\l+\alpha_i}$ when $\la \t, \alpha_i \ra \ne 0$ (resp. $\la \t, \alpha_i \ra = 0$). If $\l_i \le 0$ then the analogous result holds for $(I \t I) \in \End^2(\F_i \1_\l \E_i)$.  

\item \label{co:vanish1} If $\alpha = \alpha_i$ or $\alpha = \alpha_i + \alpha_j$ for some $i,j \in I$ with $\la i,j \ra = -1$ then $\1_{\l+r \alpha} = 0$ for $r \gg 0$ or $r \ll 0$.

\item \label{co:vanish2} If $\delta = \alpha_i + \alpha_j + \alpha_k$ with $i,j,k \in I$ a triangle in the Dynking diagram or $\delta = \alpha_i + \alpha_j + \alpha_k + \alpha_\l$ with $i,j,k,l \in I$ forming a square then $\1_{\l+r \delta} = 0$ for $r \gg 0$ and $\la \l, \delta \ra > 0$ if $\1_\l \ne 0$. 

\item \label{co:new} Suppose $i \ne j \in I$ and $\l \in X$. If $\1_{\l+\alpha_i}$ and $\1_{\l+\alpha_j}$ are nonzero then $\1_{\l}$ and $\1_{\l+\alpha_i+\alpha_j}$ are also nonzero.
\end{enumerate}

One thing to notice is that there are no divided powers $\E_i^{(r)}$ or $\F_i^{(r)}$ mentioned in the definition above. This is because their existence follows {\it a posteriori}. In certain examples (such as \cite{CLi2}) these divided powers are complicated and it is helpful to not have to deal with them. Here are a few more remarks about the conditions above. 
\begin{itemize} 
\item The condition $\1_\l=0$ (or $\1_\l \ne 0$) means that the object in $\K$ indexed by $\l$ is zero (or nonzero). 
\item Condition (\ref{co:theta}) is necessary in order to avoid ``degenerate'' examples of categorical $\g$ actions (c.f. \cite[Remark 5.19]{CR}). 
\item Condition (\ref{co:vanish2}) is only used in the proofs of Lemmas \ref{LEM:ijab} and \ref{LEM:homijk}. 
\item Condition (\ref{co:new}) is only used in a few places to shorten the argument (mostly in section \ref{sec:2}). It is not used, for instance, to prove the Serre relation. It is possible to remove this condition but it would make several arguments more cumbersome and lengthy. 
\end{itemize}

\subsection{The quiver Hecke algebra action}

Given a $(\g,\t)$ action on $\K$, an action of the quiver Hecke algebra (a.k.a. KLR algebra) $R_Q$ on $\K$ consists of a choice of 2-morphisms:
\begin{enumerate}
\item $X_i: \E_i \1_\l \rightarrow \E_i \1_\l \la 2 \ra$ for each $i \in I$, $\l \in X$,
\item $T_{ij}: \E_i \E_j \1_\l \rightarrow \E_j \E_i \1_\l \la - \la i,j \ra \ra$ for each $i,j \in I$, $\l \in X$. 
\end{enumerate}
These 2-morphisms must sastisfy the following relations:
\begin{enumerate}
\item $T_{ii}$ and $X_i$ satisfy the affine nilHecke relations
\begin{align*}
T_{ii}(X_iI) = (IX_i)T_{ii} + II & \text{ and } (X_iI)T_{ii} = T_{ii}(IX_i) + II \in \End(\E_i \E_i), \\ 
T_{ii}^2 = 0 \in \End^{-4}(\E_i \E_i) &\text{ and } (T_{ii}I)(IT_{ii})(T_{ii}I) = (IT_{ii})(T_{ii}I)(IT_{ii}) \in \End^{-6}(\E_i \E_i \E_i).
\end{align*}
\item If $i \ne j \in I$ then 
$$(IX_i)T_{ij} = T_{ij}(X_iI) \text{ and } (X_jI)T_{ij} = T_{ij}(IX_j) \in \Hom(\E_i \E_j, \E_j \E_i \la -\la i,j \ra + 2\ra).$$
\item If $i,j \in I$ with $\la i,j \ra = -1$ then 
\begin{align*}
(T_{ji}I)(IT_{ii})(T_{ij}I) &= (IT_{ij})(T_{ii}I)(IT_{ji}) + t_{ij} (III) \in \End(\E_i \E_j \E_i) \\
(T_{ji})(T_{ij}) &= t_{ij} (X_i I) + t_{ji}(IX_j) \in \End^2(\E_i \E_j) \\
(T_{ij})(T_{ji}) &= t_{ij} (IX_i) + t_{ji}(X_j I) \in \End^2(\E_j \E_i).
\end{align*}
\item If $i,j \in I$ with $\la i,j \ra = 0$ then $(T_{ji})(T_{ij}) = t_{ij}(II) \in \End(\E_i \E_j)$.
\item If $i,j,k \in I$ with $i \ne k$ then 
$$(T_{jk}I)(IT_{ik})(T_{ij}I) = (IT_{ij})(T_{ik}I)(IT_{jk}) : \E_i \E_j \E_k \rightarrow \E_k \E_j \E_i \la - \ell_{ijk} \ra$$
where $\ell_{ijk} = \la i,j \ra + \la i,k \ra + \la j,k \ra$
\item Far apart maps that do not interact with each other commute. For instance, 
$$(X_iI)(IX_j) = (IX_j)(X_iI) \text{ and } (T_{ij}I)(IIX_k)=(IIX_k)(T_{ij}I).$$
\end{enumerate}

The following is the main result in this paper. 

\begin{Theorem}\label{THM:main}
Modulo transients, a $(\g,\t)$ action carries an action of a quiver Hecke algebra $R_Q$.  
\end{Theorem}
\begin{Remark}
Transients are certain negligible 2-morphisms defined in section \ref{sec:Xi}. However, if $\g=\sl_n$, then we show in section \ref{sec:transients} that Theorem \ref{THM:main} holds without having to mod out by transients. 
\end{Remark}

Unfortunately, in Theorem \ref{THM:main} we cannot say for what choice of $Q$ the algebra $R_Q$ acts. One exception is when $\Gamma$ is a tree in which case any two choices of $Q$ are equivalent (so there is no ambiguity).

\subsection{Outline of proof}

Many of the arguments are based on a series of $\Hom$-space calculations. These calculations only involve properties in the definition of a $(\g,\t)$ action. In particular, they do not use results proved in the main body of the paper. For this reason we have separated them and placed them in the appendix. The common argument used is to prove these lemmas involves repeatedly applying adjunction (condition \ref{co:adj}) and then using the commutator relation between $\E$'s and $\F$'s (conditions \ref{co:EF} and \ref{co:EiFj}) to simplify. 

The proof of Theorem \ref{THM:main} follows a sequence of 10 steps where we gradually prove the quiver Hecke algebra relations. The order of these steps is important as we often use earlier results in later proofs. We now briefly describe these 10 steps. 
\begin{itemize}
\item Step \#0. Define 2-morphisms $T_{ij}, u_{ij}, v_{ij}$ and adjunctions $\adj_i,\adj^i$. The maps $T_{ij}$ will be subsequently rescaled but nonetheless they can be used to define a pairing 
$$(\cdot,\cdot)_\l: \End^2(\1_\l) \otimes_\k Y_\k \rightarrow \k.$$
\item Step \#1. Study the structure of $\End(\E_i \E_i)$. In the process show that $T_{iii}$ and $T'_{iii}$ are nonzero which allows one to prove a preliminary version of the affine nilHecke relations (Proposition \ref{PROP:new1}). Proposition \ref{PROP:new1} also implies that $\E_i \E_i \cong \oplus_{[2]} \E_i^{(2)}$ which is used later. 
\item Step \#2. Show that $T_{iij},T'_{iij},T_{jii},T'_{jii}$ are all nonzero. Show that $\E_i \E_j \cong \E_j \E_i$ if $\la i,j \ra = 0$. 
\item Step \#3. Rescale maps $T_{ii}$ so that $T_{iii}=T'_{iii}$, $T_{iij}=T'_{iij}$ and $T_{jii}=T'_{jii}$. 
\item Step \#4. Define transient maps (from here until step \#9 in section \ref{sec:transients} we will work modulo transients). Rescale $\alpha_i \in \End^2(\1_\l)$ so that $(\alpha_i,\alpha_i)_\l=2$. Use this to define $X_i \in \End^2(\E_i)$. 
\item Step \#5. Prove the affine nilHecke relations. 
\item Step \#6. Prove the Serre relation $\E_i \E_j \E_i \cong \E_i^{(2)} \E_j \oplus \E_j \E_i^{(2)}$ if $\la i,j \ra = -1$. 
\item Step \#7. If $\la i,j \ra = -1$ rescale $T_{ij}$ so that, for some $t_{ij},t_{ji} \in \k^\times$, we have
\begin{align*}
T_{iji} = T'_{iji} + t_{ij}(III) \in \End(\E_i \E_j \E_i) \ \ &\text{ and } \ \ T_{jij} = T'_{jij} + t_{ji}(III) \in \End(\E_j \E_i \E_j) \\ 
T_{ji}T_{ij} = t_{ij}(X_iI) + t_{ji}(IX_j) \in \End^2(\E_i\E_j) \ \ &\text{ and } \ \ T_{ij}T_{ji} = t_{ij}(IX_i)+t_{ji}(X_jI) \in \End^2(\E_j\E_i)
\end{align*}
\item Step \#8. Rescale maps $T_{ij}$ again so that $T_{ijk} = T'_{ijk}$ when $i,j,k$ are distinct. 
\item Step \#9. Show that if $\g=\sl_n$ one can redefine the $X$'s so that the quiver Hecke algebra relations hold on the nose and not just modulo transients. 
\end{itemize}
In the sequence of steps above we tried to prove as many of the relations involving $T$'s as possible before moving onto relations involving $X$'s. The former are usually easier to prove because the spaces of maps involved are smaller. For example, $\dim \End^{-6}(\E_i \E_i \E_i) \le 1$ so the relation $T_{iii} = T'_{iii}$ is almost immediate once we know that $T_{iii}$ and $T'_{iii}$ are both nonzero. On the other hand, the relation $T_{ji}T_{ij} = t_{ij}(X_iI) + t_{ji}(IX_j)$ is much more difficult because $\dim \End^2(\E_i \E_j)$ is relatively large. 

\section{Step \#0 -- Preliminary definitions and properties}\label{sec:defs}

\subsection{Notation and assumptions}

From now on we will assume that if $\l,\mu \in X$ correspond to nonzero weights spaces inside $\K$ then $\l-\mu$ belongs to the root lattice $Y$ (in other words, $\l$ and $\mu$ belong to the same coset in $X/Y$). This assumption is only for notational convenience and not an additional condition because $\E$'s and $\F$'s only go between weights in the same coset of $X/Y$. 

For a weight $\l$ we will denote by $\1_\l$ the identity 1-morphism of $\l$ in $\K$. The identity 2-morphism of $\1_\l$ will be denoted $I_\l$. Sometimes, to shorten notation, we will omit $I_\l$. 

Given two 2-morphisms $f,g$ in $\K$ we write $f \sim g$ to mean that $f$ equals some nonzero multiple of $g$. If $\A$ and $\B$ are 1-morphisms in $\K$ then $\End^d(\A)$ will be short hand for $\Hom(\A,\A \la d \ra)$ and likewise $\Hom^d(\A,\B)$ for $\Hom(\A,\B \la d \ra)$.

The fact that the space of maps between any two 1-morphisms in $\K$ is finite dimensional means that the Krull-Schmidt property holds. This means that any 1-morphism has a unique direct sum decomposition (see section 2.2 of \cite{Ri}). In particular, this means that if $\A,\B,{\sf C}$ are morphisms and $V$ is a $\Z$-graded vector space then we have the following cancellation laws (see section 4 of \cite{CK3}):
\begin{align*}
\A \oplus \B \cong \A \oplus {\sf C} & \Rightarrow \B \cong {\sf C} \\
\A \otimes_\k V \cong \B \otimes_\k V & \Rightarrow \A \cong \B.
\end{align*}

Suppose that $\A$ is a 1-morphism in $\K$ with $\End^0(\A) \cong \k$ and that $X,Y$ are arbitrary 1-morphisms. Then a 2-morphism $f : \X \rightarrow \Y $ gives rise to a bilinear pairing $ \Hom(\A, \X) \times \Hom(\Y, \A) \rightarrow \Hom(\A,\A) \cong \k $. We define the $\A$-rank of $f$ to be the rank of this bilinear pairing.

We can also define $\A$-rank as follows. Choose (non-canonical) direct sum decompositions $\X = \A \otimes_\k V \oplus \B$ and $\Y = \A \otimes_\k V' \oplus \B'$ where $V, V' $ are $\k$ vector spaces and $\B, \B'$ do not contain $\A$ as a direct summand. Then one of the matrix coefficients of $ f $ is a map $ \A \otimes_\k V \rightarrow \A \otimes_\k V' $, which (since $\End^0(\A) \cong \k$) is equivalent to a linear map $ V \rightarrow V' $.  The $\A$-rank of $f$ equals the rank of this linear map. We define the total $\A$-rank of $f$ as the sum of all the $\A \la d \ra$-ranks as $d$ varies over $\Z$. In this paper this will always turn out to be finite.

For $n \ge 1$ we denote by $[n]$ the quantum integer $q^{n-1} + q^{n-3} + \dots + q^{-n+3} + q^{-n+1}$. By convention $[-n] = -[n]$ and $[0]=0$. More generally,
$$\qbin{n}{k} := \frac{[n]\dots[1]}{([n-k] \dots [1])([k]\dots[1])}.$$
If $f = f_aq^a \in {\mathbb N}[q,q^{-1}]$ and $\A$ is a 1-morphism in $\K$ we write $\oplus_f \A$ for the direct sum $\oplus_{a \in \Z} \A^{\oplus f_a} \la a \ra$. For example, $\oplus_{[n]} \A = \oplus_{k=0}^{n-1} \A \la n-1-2k \ra$.

\subsection{Definition of $T_{ij}$}

Using Lemmas \ref{LEM:homEEs}, \ref{LEM:Ts1} and \ref{LEM:Ts2} we can fix nonzero maps 
$$T_{ij}: \E_i \E_j \1_\l \rightarrow \E_j \E_i \1_\l \la - \la i,j \ra \ra$$ 
for any $\l \in X$, $i,j \in I$. This can be done uniquely up to a nonzero multiple. For the moment we choose this multiple arbitrarily. 

To shorten notation, for $i,j,k \in I$ (not necessarily distinct) we will denote 
$$T_{ijk} := (T_{jk}I)(IT_{ik})(T_{ij}I), T'_{ijk} := (IT_{ij})(T_{ik}I)(IT_{jk}) \in \Hom(\E_i \E_j \E_k, \E_k \E_j \E_i \la - \ell_{ijk} \ra)$$
where $\ell_{ijk} := \la i,j \ra + \la i,k \ra + \la j,k \ra$. 

\subsection{Definition of $u_{ij}$ and $v_{ij}$}

Using Lemma \ref{LEM:EF-FEhom} we fix maps the following morphisms which span the corresponding one-dimensional spaces:
\begin{align*}
u_{ji}: \F_j \E_i \1_\l \rightarrow \E_i \F_j \1_\l \ \ & \text{ and } \ \ v_{ij}: \E_i \F_j \1_\l \rightarrow \F_j \E_i \1_\l \\
\adj_i: \F_i \E_i \1_\l \rightarrow \1_\l \la \l_i+1 \ra \ \ & \text{ and } \ \ \adj^i: \1_\l \rightarrow  \F_i \E_i \1_\l \la \l_i+1 \ra \\
\adj_i: \E_i \F_i \1_\l \rightarrow \1_\l \la -\l_i+1 \ra \ \ & \text{ and } \ \ \adj^i: \1_\l \rightarrow \E_i \F_i \1_\l \la -\l_i+1 \ra
\end{align*}
These are uniquely defined up to a nonzero multiple. Note that if $i \ne j$ then $u_{ji}$ and $v_{ij}$ are isomorphisms since $\E_i \F_j \cong \F_j \E_i$. 

\begin{Lemma}\label{lem:uii}
The compositions $(v_{ii})(u_{ii}) \1_\l \in \End(\F_i \E_i \1_\l)$ and $(u_{ii})(v_{ii}) \1_\l \in \End(\E_i \F_i \1_\l)$ are equal to some nonzero multiple of the identity if $\l_i \ge 0$ and $\l_i \le 0$ respectively.
\end{Lemma}
\begin{proof}
If $\l_i \ge 0$ then $\E_i \F_i \1_\l \cong \F_i \E_i \1_\l \oplus_{[\l_i]} \1_\l$ where by adjunction 
\begin{align*}
\Hom(\F_i \E_i \1_\l, \oplus_{[\l_i]} \1_\l) 
&\cong \Hom(\E_i \1_\l, \oplus_{[\l_i]} \E_i \1_\l \la -\l_i - 1\ra) \\
&\cong \bigoplus_{r=0}^{\l_i-1} \Hom(\E_i \1_\l, \E_i \1_\l \la -2-2r \ra). 
\end{align*}
This is zero by Lemma \ref{LEM:homEs}. Thus the map $u_{ii} I_\l$ must be the inclusion $\F_i \E_i \1_\l \hookrightarrow \E_i \F_i \1_\l$. Similarly, the map $v_{ii} I_\l$ must be the projection $\E_i \F_i \1_\l \twoheadrightarrow \F_i \E_i \1_\l$. Subsequently, the composition $(v_{ii})(u_{ii}) I_\l$ is some nonzero multiple of the identity map on $\F_i \E_i \1_\l$. The case of $(u_{ii})(v_{ii}) I_\l$ when $\l_i \le 0$ is proved similarly. 
\end{proof}

\subsection{The ``pitchfork'' relations}

\begin{Lemma}\label{LEM:fork}
For $i,j \in I$ the following pairs of compositions are equal up to a nonzero multiple
\begin{align}
\label{eq:c1} & \begin{cases} \E_i \E_j \1_\l \F_i \xrightarrow{T_{ij}I} \E_j \E_i \1_\l \F_i \la -d_{ij} \ra \xrightarrow{I \adj_i} \E_j \1_{\l+\alpha_i} \la - \l_i - 1 - d_{ij} \ra  \text{ and } \\
\E_i \E_j \1_\l \F_i \xrightarrow{Iv_{ji}} \E_i \F_i \1_{\l+\alpha_i+\alpha_j} \E_j \xrightarrow{\adj_i I} \E_j \1_{\l+\alpha_i} \la -\l_i - 1 - d_{ij} \ra \end{cases} \\
\label{eq:c2} & \begin{cases} \E_j \1_{\l} \xrightarrow{I \adj^i} \E_j \F_i \E_i \1_\l \la \l_i+1 \ra \xrightarrow{v_{ji} I} \F_i \E_j \E_i \1_\l \la \l_i+1 \ra \text{ and } \\
\E_j \1_\l \xrightarrow{\adj^i I} \F_i \E_i \E_j \1_\l \la \l_i+d_{ij}+1 \ra \xrightarrow{IT{ij}} \F_i \E_j \E_i \1_\l \la \l_i+1 \ra \end{cases} \\
\label{eq:c3} & \begin{cases} \F_i \E_j \E_i \1_\l \xrightarrow{u_{ij} I} \E_j \F_i \E_i \1_\l \xrightarrow{I \adj_i} \E_j \1_\l \la \l_i+1 \ra \text{ and } \\
\F_i \E_j \E_i \1_\l \xrightarrow{I T_{ji}} \F_i \E_i \1_{\l+\alpha_j} \E_j \la -d_{ij} \ra \xrightarrow{\adj_i I} \E_j \1_\l \la \l_i+1 \ra \end{cases} \\
\label{eq:c4} & \begin{cases} \E_j \1_{\l+\alpha_i} \xrightarrow{I \adj^i} \E_j \E_i \F_i \1_{\l+\alpha_i} \la -\l_i-1 \ra \xrightarrow{T_{ji}I} \E_i \E_j \F_i \1_{\l+\alpha_i} \la -\l_i-1-d_{ij} \ra \text{ and } \\
\E_j \1_{\l+\alpha_i} \xrightarrow{\adj^i I} \E_i \F_i \E_j \1_{\l+\alpha_i} \la -\l_i-1-d_{ij} \ra \xrightarrow{Iu_{ij}} \E_i \E_j \F_i \1_{\l+\alpha_i} \la -\l_i-1-d_{ij} \ra \end{cases}
\end{align}
where $d_{ij} := \la i,j \ra$. Moreover, each one these compositions is nonzero if and only if $\1_\l, \1_{\l+\alpha_i}, \1_{\l+\alpha_j}$ and $\1_{\l+\alpha_i+\alpha_j}$ are all nonzero. 
\end{Lemma}
\begin{proof}
We prove that the two compositions in (\ref{eq:c1}) are equal (the proof of the others is the same). First, note that 
\begin{align*}
\Hom(\E_i \E_j \1_\l \F_i, \E_j \1_{\l+\alpha_j} \la -\l_i-1- d_{ij} \ra) 
& \cong \Hom(\E_i \E_j \1_\l, \E_j (\1_\l \F_i)_L \la -\l_i-1-d_{ij} \ra) \\
& \cong \Hom(\E_i \E_j \1_\l, \E_j \E_i \1_\l \la -d_{ij} \ra)
\end{align*}
which, by Lemmas \ref{LEM:homEEs}, \ref{LEM:Ts1} and \ref{LEM:Ts2}, is at most one-dimensional. Moreover, it is nonzero if and only if $\1_\l,\1_{\l+\alpha_i},\1_{\l+\alpha_j},\1_{\l+\alpha_i+\alpha_j}$ are all nonzero. So it remains to show that our two compositions in (\ref{eq:c1}) are nonzero. But, by adjunction, these two compositions are equivalent to $T_{ij} \in \Hom(\E_i \E_j \1_\l, \E_j \E_i \1_\l \la -d_{ij} \ra)$ and $v_{ji} \in \Hom(\E_j \F_i \1_{\l+\alpha_i}, \F_i \E_j \1_{\l+\alpha_i})$. Both of these are nonzero if $\1_\l, \1_{\l+\alpha_i}, \1_{\l+\alpha_j}, \1_{\l+\alpha_i+\alpha_j}$ are nonzero. 
\end{proof}

\begin{Corollary}\label{COR:nonzero}
For $i,j,k \in I$ not necessarily distinct and $\gamma \in \End^d(\E_j \E_k \1_{\l+\alpha_i})$ the compositions 
\begin{align}
\label{EQ:nonzero1}
(\gamma I)(IT_{ik})(T_{ij}I) & \in \Hom^{d-d_{ij}-d_{ik}}(\E_i \E_j \E_k \1_\l, \E_k \E_j \E_i \1_\l) \text{ and } \\ 
\label{EQ:nonzero2} 
(I \gamma)(v_{ji}I)(Iv_{ki}) & \in \Hom^d(\E_j \E_k \F_i \1_{\l+\alpha_i}, \F_i \E_j \E_k \1_{\l+\alpha_i})
\end{align}
are either both zero or both nonzero. Similarly for 
\begin{align*}
(I \gamma)(T_{ji}I)(IT_{ki}) & \in \Hom^{d-d_{ij}-d_{ik}}(\E_j \E_k \E_i \1_{\l+\alpha_i}, \E_i \E_j \E_k \1_{\l+\alpha_i}) \text{ and } \\
(\gamma I)(Iu_{ik})(u_{ij}I) & \in \Hom^d(\F_i \E_j \E_k \1_\l, \E_j \E_k \F_i \1_\l).
\end{align*}
\end{Corollary}
\begin{proof}
The composition in (\ref{EQ:nonzero1} is zero if and only if the composition 
$$\E_j \E_k \F_i \1_{\l+\alpha_i} \xrightarrow{\adj^i III} \F_i \E_i \E_j \E_k \F_i \1_{\l+\alpha_i} \xrightarrow{I[(\gamma I)(IT_{ik})(T_{ij}I)]I} \F_i \E_k \E_j \E_i \F_i \1_{\l+\alpha_i} \xrightarrow{III \adj_i} \F_i \E_j \E_k \1_{\l+\alpha_i}$$
is zero (we omit the shifts to simplify the notation). Using Lemma \ref{LEM:fork} twice, we find that this composition is (up to a nonzero multiple) equal to 
$$\E_j \E_k \F_i \1_{\l+\alpha_i} \xrightarrow{II \adj^i I} \E_j \E_k \F_i \E_i \F_i \1_{\l+\alpha_i} \xrightarrow{III \adj_i} \E_j \E_k \F_i \1_{\l+\alpha_i} \xrightarrow{(I \gamma)(v_{ji}I)(Iv_{ki})} \F_i \E_j \E_k \1_{\l+\alpha_i}$$ 
which (up to a multiple) is the same as (\ref{EQ:nonzero2}). Thus (\ref{EQ:nonzero1}) and (\ref{EQ:nonzero2}) are either both zero or both nonzero. The equivalence of the second pair of compositions follows similarly. 
\end{proof}

\subsection{Some properties of $\t$'s}

\begin{Lemma}\label{LEM:bubbles1}
Suppose $i \in I$ and $\t \in Y_\k$ so that $\la \t, \alpha_i \ra \ne 0$. Then the compositions
\begin{align*}
\1_\l \xrightarrow{\adj^i} \E_i \1_{\l-\alpha_i} \F_i \la -\l_i+1 \ra \xrightarrow{I \t^{\l_i-1} I} \E_i \1_{\l-\alpha_i} \F_i \la \l_i-1 \ra \xrightarrow{\adj_i} \1_\l \ \ & \text{ if } \l_i \ge 1 \\
\1_\l \xrightarrow{\adj^i} \F_i \1_{\l+\alpha_i} \E_i \la \l_i+1 \ra \xrightarrow{I \t^{-\l_i-1} I} \F_i \1_{\l+\alpha_i} \E_i \la -\l_i-1 \ra \xrightarrow{\adj_i} \1_\l \ \ & \text{ if } \l_i \le -1 
\end{align*}
are both to a nonzero multiple of the identity map in $\End(\1_\l)$. 
\end{Lemma}
\begin{proof}
If $\l_i \ge 0$ then $\1_\l \la \l_i-1 \ra \xrightarrow{\adj^i} \E_i \F_i \1_\l \cong \F_i \E_i \1_\l \bigoplus_{[\l_i]} \1_\l$ is the inclusion of $\1_\l \la \l_i-1 \ra$ into the lowest degree summand $\1_\l$ on the right side. By assumption (\ref{co:theta}) on $\t$, the map $(I \t^{\l_i-1} I) \in \End^{2(\l_i-1)}(\E_i \1_{\l-\alpha_i} \F_i)$ induces an isomorphism between the lowest degree and highest degree summands of $\1_\l$ inside $\E_i \F_i \1_\l$. 

Finally, $\F_i \E_i \1_\l \bigoplus_{[\l_i]} \1_\l \cong \E_i \F_i \1_\l \xrightarrow{\adj_i} \1_\l \la -\l_i-1 \ra$ is the projection from the highest degree summand of $\1_\l$ in $\E_i \F_i \1_\l$ onto $\1_\l \la -\l_i-1 \ra$. Thus composing the three maps gives an isomorphism $\1_\l \xrightarrow{\sim} \1_\l$ which, by condition (\ref{co:hom1}), must be a multiple of the identity. The case $\l_i \le 0$ is similar. 
\end{proof}

\begin{Lemma} \label{LEM:bubbles2} 
Suppose $i \in I$ and $\t \in Y_\k$ so that $\la \t, \alpha_i \ra \ne 0$. Then the maps 
\begin{align*}
u_{ii} \bigoplus_{r=0}^{\l_i-1} [(I \t^r I) \circ \adj^i]: \F_i \E_i \1_\l \bigoplus_{[\l_i]} \1_\l \xrightarrow{\sim} \E_i \1_{\l-\alpha_i} \F_i & \ \ \text{ if } \l_i \ge 0 \ \ \text{ and }  \\
v_{ii} \bigoplus_{r=0}^{-\l_i-1} [(I \t^r I) \circ \adj^i]: \E_i \F_i \1_\l \bigoplus_{[-\l_i]} \1_\l \xrightarrow{\sim} \F_i \1_{\l+\alpha_i} \E_i & \ \ \text{ if } \l_i \le 0
\end{align*}
are isomorphisms. 
\end{Lemma}
\begin{proof}
Suppose $\l_i \ge 0$ (the case $\l_i \le 0$ is similar). By Lemma \ref{lem:uii} the map $u_{ii}$ is the inclusion of the summand $\F_i \E_i \1_\l$ into $\E_i \F_i \1_\l$. 

On the other hand, $\adj^i: \1_\l \la \l_i-1 \ra \rightarrow \E_i \F_i \1_\l$ is the inclusion of $\1_\l \la \l_i-1 \ra$ as the lowest degree summand $\1_\l$ inside $\E_i \F_i \1_\l$. This means that, by condition (\ref{co:theta}), the composition 
$$\1_\l \la \l_i-1-2r \ra \xrightarrow{\adj^i} \E_i \1_{\l-\alpha_i} \F_i \la -2r \ra \xrightarrow{I \t^r I} \E_i \1_{\l-\alpha_i} \F_i$$
is an isomorphism between $\1_\l$ and the degree $-\l_i+1+2r$ summand of $\1_\l$ inside $\E_i \1_{\l-\alpha_i} \F_i$ (for $r=0, \dots, \l_i-1$). Since $\End^l(\1_\l) = 0$ for $l < 0$ we get that the composition
$$\bigoplus_{r=0}^{\l_i-1} [(I \t^r I) \circ \adj^i]: \bigoplus_{[\l_i]} \1_\l \rightarrow \E_i \1_{\l-\alpha_i} \F_i \cong \F_i \E_i \1_\l \bigoplus_{[\l_i]} \1_\l \rightarrow \bigoplus_{[\l_i]} \1_\l$$ 
(where the rightmost map is a projection) must be an upper triangular matrix with isomorphisms on the diagonal. The result follows. 
\end{proof}

Next we have the following description of $\End^2(\E_i)$. 

\begin{Lemma}\label{LEM:EndE} 
Choose $\t \in Y_\k$ so that $\la \t, \alpha_i \ra \ne 0$. Then the space $\End^2(\1_{\l+\alpha_i} \E_i \1_{\l})$ is spanned by 
\begin{enumerate}
\item $(II \t)$ and elements $(\gamma II)$ for some $\gamma \in \End^2(\1_{\l + \alpha_i})$, if $\l_i \ge -1$, 
\item $(\t II)$ and elements $(II \gamma)$ for some $\gamma \in \End^2(\1_\l)$, if $\l_i \le -1$.
\end{enumerate}
\end{Lemma}
\begin{proof}
Suppose $\l_i \ge 0$. First we have
\begin{align}
\nonumber \Hom(\E_i \1_\l, \E_i \1_\l \la 2 \ra)
\nonumber \cong& \Hom(\1_{\l + \alpha_i}, \E_i (\E_i \1_\l)_L \la 2 \ra) \\
\nonumber \cong& \Hom(\1_{\l + \alpha_i}, \E_i \F_i \1_{\l + \alpha_i} \la -\l_i+1 \ra) \\
\nonumber \cong& \Hom(\1_{\l + \alpha_i}, \bigoplus_{r=0}^{\l_i+1} \1_{\l+\alpha_i} \la 2-2r \ra \bigoplus \F_i \E_i \1_{\l+\alpha_i} \la -\l_i+1 \ra) \\
\nonumber \cong& \bigoplus_{r=0}^{\l_i+1} \Hom(\1_{\l + \alpha_i}, \1_{\l+\alpha_i} \la 2-2r \ra) \bigoplus \Hom((\1_{\l+\alpha_i} \F_i)_L, \E_i \1_{\l+\alpha_i} \la -\l_i+1 \ra) \\
\label{EQ:X} \cong& \End(\1_{\l+\alpha_i}) \bigoplus \End^2(\1_{\l+\alpha_i}) \bigoplus \Hom(\E_i \1_{\l+\alpha_i}, \E_i \1_{\l+\alpha_i} \la -2\l_i-2 \ra).
\end{align}
The right hand term above is zero by Lemma \ref{LEM:homEs}. If $f \in \End^2(\E_i \1_\l)$ then we denote the map induced by adjunction $f' \in \Hom(\1_{\l+\alpha_i}, \E_i \F_i \1_{\l+\alpha_i} \la -\l_i+1 \ra)$ and the induced maps (via the isomorphisms above) in $\End(\1_{\l+\alpha_i})$ and $\End^2(\1_{\l+\alpha_i})$ by $f_0$ and $f_1$ respectively.

Using the isomorphism
$$\E_i \F_i \1_{\l+\alpha_i} \xrightarrow{\sim} \F_i \E_i \1_{\l+\alpha_i} \bigoplus_{[\l_i+2]} \1_{\l+\alpha_i}$$
from Lemma \ref{LEM:bubbles2} we see that $f'$ is a sum of compositions 
$$\1_{\l+\alpha_i} \xrightarrow{f_{1-r}} \1_{\l+\alpha_i} \la 2-2r \ra \xrightarrow{\adj^i} \E_i \1_\l \F_i  \la -\l_i+1-2r \ra \xrightarrow{I \t^r I} \E_i \1_\l \F_i \la -\l_i+1 \ra.
$$
where $r = 0,1$.  Consequently, by adjunction, $f$ is spanned by compositions of the form
$$\1_{\l+\alpha_i} \E_i \1_\l \xrightarrow{f_{1-r} II} \1_{\l+\alpha_i} \E_i \1_\l \la 2-2r \ra \xrightarrow{II \t^r} \1_{\l+\alpha_i} \E_i \1_\l \la 2 \ra.$$
Finally, $f_0$ must be some multiple of the identity by Lemma \ref{LEM:homEs} and the result follows. 

The case $\l_i=-1$ follows as above except that (\ref{EQ:X}) is now $\End(\1_{\l+\alpha_i}) \oplus \End(\E_i \1_{\l+\alpha_i})$. The case $\l_i \le -1$ is similar except that the first step is $\Hom(\E_i \1_\l, \E_i \1_\l \la 2 \ra) \cong \Hom((\E_i \1_\l)_L \E_i \1_\l, \1_\l \la 2 \ra)$.
\end{proof}

\subsection{A natural pairing}

For any $i \in I$ and $\t \in \End^2(\1_\l)$ we have $T_{ii}(I \t I) T_{ii} = c T_{ii} \in \End^{-2}(\E_i \1_\l \E_i)$ for some $c \in \k$ (this is simply because $\dim \End^{-2}(\E_i \1_\l \E_i) \le 1$ by Lemma \ref{LEM:homEs}). 

\begin{framed}
\noindent Definition: For each $\l \in X$ we define the pairing $(\cdot, \cdot)_\l: \End^2(\1_\l) \otimes_\k Y_\k \rightarrow \k$ by $(\t, \alpha_i)_\l := c$ which we extend linearly. 
\end{framed}

\begin{Remark}
The definition of $(\cdot,\cdot)_\l$ depends on $T_{ii}$ but this dependence is mild. For example, rescaling $T_{ii}$ (as we do in section \ref{sec:2}) does not changes whether or not $(\t, \alpha_i)_\l$ is zero. 
\end{Remark}

\section{Step \#1 -- The structure of $\End(\E_i \E_i)$}

We begin by studying the structure of $\End(\E_i \E_i)$ and ultimately show that there exist $\E_i^{(2)}$ such that $\E_i \E_i \cong \oplus_{[2]} \E_i^{(2)}$.

\subsection{Some technical Lemmas}

For the remainder of this section we fix $i \in I$ and $\t \in Y_\k$ such that $\la \t, \alpha_i \ra \ne 0$. We also use the convention that a claim such as ``$f \in \End(\A)$ is nonzero'' assumes the obviously necessary condition that $\A$ is nonzero. We begin with a few technical results.

\begin{Lemma}\label{LEM:new1}
If $\t \in Y_\k$ with $\la \t, \alpha_i \ra \ne 0$ then $T_{ii}(I \t^{|\l_i|+1} I)T_{ii} \in \End^{2|\l_i|-2}(\E_i \1_\l \E_i)$ is nonzero. 
\end{Lemma}
\begin{proof}
Let us suppose $\l_i \ge 0$ (the case $\l_i \le 0$ is the same). Consider the following composition 
$$\E_i \1_\l \xrightarrow{II \adj^i} \E_i \1_\l \E_i \F_i \xrightarrow{(T_{ii}I)(I \t^{\ell+1} II)(T_{ii}I)} \E_i \1_\l \E_i \F_i \xrightarrow{II \adj_i} \E_i \1_\l$$
where we omit the shifts for convenience. We can use Lemma \ref{LEM:fork} to rewrite this composition as 
$$\E_i \1_\l \xrightarrow{adj^i II} \E_i \1_\l \F_i \E_i \xrightarrow{(IIv_{ii})(I \t^{\l_i+1} II)(IIu_{ii})} \E_i \1_\l \F_i \E_i \xrightarrow{\adj_i II} \E_i \1_\l.$$
Now, since $\l_i \ge 0$ we have $v_{ii}u_{ii} \sim id$ and so this simplifies to give 
$$\E_i \1_\l \xrightarrow{(\adj^i II)} \E_i \1_\l \F_i \E_i \xrightarrow{(I \t^{\l_i+1} II)} \E_i \1_\l \F_i \E_i \xrightarrow{(\adj_i I)} \E_i \1_\l.$$
By Lemma \ref{LEM:bubbles1} this is (up to rescaling) equal to the identity map on $\E_i \1_\l$. 
\end{proof}

\begin{Lemma}\label{LEM:new2}
If $\l_i \le -2$ then $\End(\E_i \E_i \1_\l)$ is spanned by $T_{ii}(II \rho)$ where $\rho \in \End^2(\1_\l)$ together with another $3$-dimensional space. Likewise, if $\l_i \ge 2$ then $\End(\1_\l \E_i \E_i)$ is spanned by $T_{ii}(\rho II)$ where $\rho \in \End^2(\1_\l)$ together with another $3$-dimensional space. 
\end{Lemma}
\begin{proof}
Let us suppose we are in the second case where $\l_i \ge 2$ (the first case is the same). The same adjunction argument as in the proof of Lemma \ref{LEM:homEEs} shows that 
$$\End(\1_\l \E_i \E_i) \cong \Hom(\1_{\l+\alpha_i} \E_i \E_i, \1_{\l+\alpha_i} \E_i \E_i \la -2\l_i + 2 \ra) \oplus \Hom(\1_{\l} \E_i, \bigoplus_{[2][\l_i-1]} \1_{\l} \E_i \la -\l_i+3 \ra).$$
If $\l_i \ge 3$ then the middle term above is zero while all but three terms in the direct sum on the right are zero. The surviving three terms are $\End^2(\1_\l \E_i) \oplus \End(\1_\l \E_i)^{\oplus 2} \cong \End^2(\1_\l \E_i) \oplus \k^{\oplus 2}$. The term $\End^2(\1_\l \E_i \1_{\l-\alpha_i})$ is spanned by $II \t$ for a choice of $\t \in Y_\k$ with $\la \t, \alpha_i \ra \ne 0$ and by $\rho II$ where $\rho \in \End^2(\1_\l)$ is arbitrary. If one traces back through the isomorphisms the maps $\rho II$ correspond to maps $T_{ii}(\rho II) \in \End(\1_\l \E_i \E_i)$. The result follows. 

If $\l_i = 2$ then the middle term above is one-dimensional and the direct sum on the right is just $\Hom(\1_\l \E_i, \oplus_{[2]} \1_\l \E_i \la 1 \ra) \cong \End^2(\1_\l \E_i) \oplus \End(\1_\l \E_i)$. The same argument as above now applies. 
\end{proof}

\subsection{Some non-vanishing results}

\begin{Lemma}\label{LEM:2}
If $\l_i \ge -1$ then $T_{iii} I_{\l-\alpha_i}$ is nonzero while if $\l_i \le -1$ then $T'_{iii} I_{\l-\alpha_i}$ is nonzero.
\end{Lemma}
\begin{Remark} 
The adjunction argument below does not tell us whether $T_{iii} I_{\l-\alpha_i}$ and $T'_{iii} I_{\l-\alpha_i}$ are nonzero if $\l_i < -1$ and $\l_i > -1$ respectively. 
\end{Remark}
\begin{proof}
To show that $T'_{iii} I_{\l-\alpha_i} \ne 0$ it suffices to show that 
$$\F_i \E_i \E_i \1_\l \xrightarrow{u_{ii} I} \E_i \F_i \E_i \1_\l \xrightarrow{Iu_{ii}} \E_i \E_i \F_i \1_\l \xrightarrow{T_{ii}I} \E_i \E_i \F_i \1_\l \la -2 \ra$$
is nonzero. This is a consequence of Corollary \ref{COR:nonzero} where we take $i=j=k$ and $\gamma = T_{ii}$. 

Suppose $\l_i \le -2$ (the case $\l_i=-1$ is a little different). Precomposing with $(v_{ii}I)(Iv_{ii})$ and using Lemma \ref{lem:uii} we obtain the map $T_{ii}I \in \End^{-2}(\E_i \E_i \F_i \1_\l)$ (up to a nonzero multiple). This map is clearly nonzero since we can precompose it with $\E_i$ and simplify to obtain several copies of $T_{ii} I_{\l-\alpha_i}$ which are nonzero.

If $\l_i=-1$ we have
$$\E_i \1_\l \oplus \F_i \E_i \E_i \1_\l \cong \E_i \F_i \E_i \1_\l \cong \E_i \1_\l \oplus \E_i \E_i \F_i \1_\l$$
which means that $\F_i \E_i \E_i \1_\l \cong \E_i \E_i \F_i \1_\l$. Then the same argument as in Lemma \ref{lem:uii} shows that the composition $(Iu_{ii}) \circ (u_{ii}I)$ is an isomorphism. The argument used above when $\l_i \le -2$ now applies. 

This concludes the proof that $T'_{iii} I_{\l-\alpha_i} \ne 0$ if $\l_i \le -1$. The nonvanishing of $T_{iii} I_{\l-\alpha_i}$ if $\l_i \ge -1$ is proved similarly. 
\end{proof}

\begin{Lemma}\label{LEM:new3}
Suppose $\l_i \ge -1$. If
\begin{equation}\label{EQ:rel1}
id =  a_1 (IT_{ii}I)(II \t II) + a_2 (II \t II)(IT_{ii}I) + (IT_{ii}I)(\gamma) \in \End(\1_{\l+\alpha_i} \E_i \1_\l \E_i \1_{\l-\alpha_i})
\end{equation}
where $\gamma = a_3 (IIII \t) + (\rho IIII)$ for some $\rho \in \End^2(\1_{\l+\alpha_i})$ and $a_1,a_2,a_3 \in \k$ then $T_{ii}(I \t I)T_{ii} \in \End^{-2}(\E_i \1_\l \E_i)$ and $T'_{iii} I_{\l-\alpha_i}$ are both nonzero. Conversely, if $T_{ii}(I \t I)T_{ii} \in \End^{-2}(\E_i \1_{\l-\alpha_i} \E_i)$ and $T'_{iii} I_{\l-2\alpha_i}$ are both nonzero then (\ref{EQ:rel1}) holds. 

Similarly, suppose $\l_i \le 1$. If 
\begin{equation}\label{EQ:rel2}
id = a_1 (IT_{ii}I)(II \t II) + a_2 (II \t II)(IT_{ii}I) + (IT_{ii}I)(\gamma) \in \End(\1_{\l+\alpha_i} \E_i \1_\l \E_i \1_{\l-\alpha_i})
\end{equation}
where $\gamma = a_3 (\t IIII) + (IIII \rho)$ for some $\rho \in \End^2(\1_{\l-\alpha_i})$ then $T_{ii}(I \t I)T_{ii} \in \End^{-2}(\E_i \1_\l \E_i)$ and $I_{\l+\alpha_i} T_{iii}$ are both nonzero. Conversely, if $T_{ii}(I \t I)T_{ii} \in \End^{-2}(\E_i \1_{\l+\alpha_i} \E_i)$ and $I_{\l+2\alpha_i} T_{iii}$ are both nonzero then (\ref{EQ:rel2}) holds. 
\end{Lemma}
\begin{proof}
Suppose $\l_i \ge 0$ (the case $\l_i \le 0$ is the same). Composing (\ref{EQ:rel1}) on the right with $(IT_{ii}I)$ we get $(IT_{ii}I) = a_1 (IT_{ii}I)(II \t II)(IT_{ii}I)$ which implies that $a_1 \ne 0$ and $T_{ii}(I \t I)T_{ii} \in \End^{-2}(\E_i \1_\l \E_i)$ is nonzero. Note that we also get $a_2 \ne 0$ by multiplying on the left with $(IT_{ii}I)$. 

Now suppose $T'_{iii} I_{\l-\alpha_i} = 0$ and consider 
$$T'_{iii}(II \t I) = (IT_{ii})(T_{ii}I)(IT_{ii})(II \t I) \in \End^{-4}(\E_i \E_i \1_\l \E_i).$$
On the one hand this is zero while on the other we can use (\ref{EQ:rel1}) to rewrite it as
\begin{align*}
a_1^{-1}(IT_{ii})(T_{ii}I) - a_2a_1^{-1} (IT_{ii})(T_{ii}I)(II \t I)(IT_{ii}) 
&= a_1^{-1}(IT_{ii})(T_{ii}I) - a_2 a_1^{-1}(IT_{ii})(II \t I)(T_{ii}I)(IT_{ii}) \\
&= a_1^{-1}(IT_{ii})(T_{ii}I) - a_2^2 a_1^{-2}(T_{ii}I)(IT_{ii}) 
\end{align*}
where we use that $T'_{iii} I_{\l-\alpha_i}=0$ on several occasions to simplify above. Thus we get that 
$$a_1 (IT_{ii})(T_{ii}I) = a_2^2 (T_{ii}I)(IT_{ii}) \in \End^{-4}(\E_i \E_i \E_i \1_{\l-\alpha_i}).$$ 
Multiplying by $(T_{ii}I)$ on the left gives $a_1 T_{iii} I_{\l-\alpha_i} = 0$ which is a contradiction by Lemma \ref{LEM:2}. 

To prove the converse suppose we have 
\begin{equation}\label{EQ:linind}
c_1 (IT_{ii}I)(II \t II) + c_2 (II \t II)(IT_{ii}I) + c_3 (IT_{ii}I)(IIII \t) + (IT_{ii}I)(\rho IIII) = 0
\end{equation}
inside $\End(\1_{\l+\alpha_i} \E_i \1_\l \E_i \1_{\l-\alpha_i})$ for some $\rho \in \End^2(\1_{\l+\alpha_i})$ and $c_1,c_2,c_3 \in \k$. Let $e \ge 1$ be the smallest integer such that $T_{ii}(I \t^e I)T_{ii} \in \End^{2e-4}(\E_i \1_\l \E_i)$ is nonzero (such an $e$ exists by Lemma \ref{LEM:new1}). Then composing (\ref{EQ:linind}) on the right with $(II \t^{e-1} II)(IT_{ii}I)$ all but the first terms vanish and we get $c_1=0$. Similarly, composing on the left gives $c_2=0$. 

Next, consider the composition 
$$\E_i \E_i \1_{\l-\alpha_i} \E_i \xrightarrow{(II \t I)(IT_{ii})} \E_i \E_i \1_{\l-\alpha_i} \E_i \xrightarrow{T_{iii}} \E_i \E_i \1_{\l-\alpha_i} \E_i.$$
On the one hand, since $T'_{iii} I_{\l-2\alpha_i} \ne 0$ it follows that $T'_{iii} I_{\l-\alpha_i} \sim T_{iii} I_{\l-\alpha_i}$ and hence the composition above equals (up to rescaling)
$$\E_i \E_i \1_{\l-\alpha_i} \E_i \xrightarrow{(IT_{ii})(II \t I)(IT_{ii})} \E_i \E_i \1_{\l-\alpha_i} \E_i \la -2 \ra \xrightarrow{(IT_{ii})(T_{ii}I)} \E_i \E_i \1_{\l-\alpha_i} \E_i \la -6 \ra.$$
Now, since $T_{ii}(I \t I)T_{ii} \in \End^{-2}(\E_i \1_{\l-\alpha_i} \E_i)$ is nonzero, it must be equal to $T_{ii}$ (up to rescaling). So the composition above becomes $T'_{iii} I_{\l-2\alpha_i}$ which is nonzero. 

On the other hand, if $c_3 \ne 0$ then we can rewrite the composition as 
$$\1_{\l+\alpha_i} \E_i \E_i \E_i \xrightarrow{c_3^{-1}(IT_{ii})(\rho III)} \1_{\l+\alpha_i} \E_i \E_i \E_i \xrightarrow{T_{iii}} \1_{\l+\alpha_i} \E_i \E_i \E_i \la -6 \ra$$ 
which equals zero. We conclude that $c_3$ must be zero. 

This means that modulo maps of the form $(IT_{ii}I)(\rho IIII) \in \End(\1_{\l+\alpha_i} \E_i \1_\l \E_i \1_{\l-\alpha_i})$ the first three maps in (\ref{EQ:linind}) are linearly independent. It follows by Lemma \ref{LEM:new2} that $id \in \End(\E_i \1_\l \E_i)$ must be a linear combination as in (\ref{EQ:rel1}). This concludes the proof. 
\end{proof}

\begin{Corollary}\label{COR:new1}
The maps $T_{iii} I_\l$ and $T'_{iii} I_\l$ as well as $T_{ii}(I \t I)T_{ii} \in \End^{-2}(\E_i \1_\l \E_i)$ are all nonzero.
\end{Corollary}
\begin{proof}
Suppose $\l_i \ge 0$ (the case $\l_i \le 0$ is the same). By Lemma \ref{LEM:new3} it suffices to show that $T'_{iii} I_{\l-\alpha_i}$ and $T_{ii}(I \t I)T_{ii} \in \End^{-2}(\E_i \1_\l \E_i)$ are both nonzero when $\l_i=-1$ and $\l_i=0$.

Suppose $\l_i = -1$. We know $T'_{iii} I_{\l-\alpha_i} \ne 0$ by Lemma \ref{LEM:2}. Now suppose $T_{ii}(I \t I)T_{ii} \in \End^{-2}(\E_i \1_\l \E_i)$ is zero. By Lemma \ref{LEM:new2} we know that there exists a nontrivial linear relation
\begin{equation}\label{EQ:lindep}
b_0 \cdot id + b_1 (IT_{ii}I)(II \t II) + b_2 (II \t II)(IT_{ii}I) + b_3 (IT_{ii}I)(\t IIII) + (IT_{ii}I)(IIII \rho) = 0
\end{equation}
inside $\End(\1_{\l+\alpha_i} \E_i \1_\l \E_i \1_{\l-\alpha_i})$ for some $b_0,b_1,b_2,b_3 \in \k$. Composing on the left by $(IT_{ii}I)(II \t II)$ and on the right by $(II \t II)(IT_{ii}I)$ all but the first terms vanish and we get $b_0 (I T_{ii} I)(II \t^2 II)(I T_{ii} I) = 0$. By Lemma \ref{LEM:new1} this means $b_0$. 

Next, composing on the right with $(II \t II)(IT_{ii}I)$ we get $b_1 (I T_{ii} I)(II \t^2 II)(I T_{ii} I) = 0$ which means $b_1=0$. Likewise, composing on the left gives $b_2=0$. Subsequently, we are left with 
$$b_3 (IT_{ii}I)(\t III) + (IT_{ii}I)(III \rho) = 0 \in \End(\1_{\l+\alpha_i} \E_i \E_i \1_{\l-\alpha_i}).$$
To see this cannot be the case consider the composition 
\begin{equation}\label{EQ:new8}
\E_i \1_{\l+\alpha_i} \E_i \E_i \xrightarrow{(T_{ii}I)} \E_i \1_{\l+\alpha_i} \E_i \E_i \la -2 \ra \xrightarrow{(I \t^2 II)} \E_i \1_{\l+\alpha_i} \E_i \E_i \la 2 \ra \xrightarrow{T_{iii}} \E_i \1_{\l+\alpha_i} \E_i \E_i \la -4 \ra.
\end{equation}
On the one hand this is equal to 
$$\E_i \1_{\l+\alpha_i} \E_i \E_i \xrightarrow{(T_{ii}I)(I \t^2 II)(T_{ii}I)} \E_i \1_{\l+\alpha_i} \E_i \E_i \xrightarrow{(T_{ii}I)(IT_{ii})} \E_i \1_{\l+\alpha_i} \E_i \E_i \la 4 \ra.$$
By the same argument as in the proof of Lemma \ref{LEM:2} this composition is equivalent to (after composing with the appropriate adjunction maps and applying Lemma \ref{LEM:fork}) 
$$(IT_{ii})(II \t^2 I)(IT_{ii}) \in \End(\F_i \E_i \1_{\l+\alpha_i} \E_i).$$
Composing on the left with $\E_i$ and simplifying gives 3 copies of $(T_{ii})(I \t^2 I)(T_{ii}) \in \End(\E_i \1_{\l+\alpha_i} \E_i)$ which is nonzero by Lemma \ref{LEM:new1}. Thus the composition in (\ref{EQ:new8}) is nonzero. 

On the other hand, (\ref{EQ:new8}) is equal to (up to rescaling)
$$(T'_{iii}I)(I \t^2 III)(T_{ii}II) = b_3^{-2} (T'_{iii}I)(T_{ii}II)(IIII \rho^2) = 0 \in \End^{-4}(\E_i \1_{\l+\alpha_i} \E_i \E_i \1_{\l-\alpha_i}).$$
Thus we arrive at a contradiction which means that $T_{ii}(I \t I)T_{ii} \in \End^{-2}(\E_i \1_\l \E_i)$ must be nonzero. 

Next, suppose $\l_i=0$. A more careful look at the argument used to prove Lemma \ref{LEM:new2} reveals that if you follow through the adjunctions when $\l_i > 0$ then the 3-dimensional subspace of $\End(\E_i \1_\l \E_i)$ when $\l_i > 0$ is spanned by $T_{ii}(I \t I), (I \t I)T_{ii}$ and $id$. More precisely, the first step in the adjunction gives 
$$\Hom(\E_i \1_\l, \E_i \1_\l \E_i \F_i \la -\l_i+1 \ra) \cong \Hom(\E_i \1_\l, \E_i \1_\l \F_i \E_i \la -\l_i+1 \ra \oplus_{[\l_i]} \E_i \1_\l \la -\l_i+1 \ra)$$
and the identity map $id$ corresponds to the copy of $\E_i$ in the highest degree ({\it i.e.} in degree zero) in the right hand summation. However, if $\l_i=0$ the direct sum vanishes. In this case the 3-dimensional space is spanned by $T_{ii}(I \t I), (I \t I)T_{ii}$ as before but instead of $id$ one finds a map $T_{ii}(\beta)T_{ii}$ for some $\beta \in \End^4(\E_i \1_\l \E_i)$. We could make $\beta$ explicit but we do not need to. The only thing to notice now is that we must have a relation as in (\ref{EQ:rel1}) except that 
$$\gamma = (IT_{ii}I)(I \beta I)(IT_{ii}I) + (\rho IIII) \in \End(\1_{\l+\alpha_i} \E_i \1_\l \E_i \1_{\l-\alpha_i}).$$ 
However, it is easy to check that the rest of the argument in Lemma \ref{LEM:new3} goes through exactly as before to show that $T_{ii}(I \t I)T_{ii} \in \End^{-2}(\E_i \1_\l \E_i)$ and $T'_{iii} I_{\l-\alpha_i}$ are still nonzero. 
\end{proof}

\subsection{Divided powers}

\begin{Proposition}\label{PROP:new1}
We have
\begin{equation}\label{EQ:5}
id = a \cdot (I \t I)T_{ii} + a \cdot T_{ii}(I \t I) + T_{ii}(\tau) \in \End(\E_i \1_\l \E_i)
\end{equation}
where $a \ne 0$ and 
\begin{equation}\label{EQ:new2}
\tau = \begin{cases}
(III \rho) + b \cdot (\t III) \in \End^2(\1_{\l+\alpha_i} \E_i \E_i \1_{\l-\alpha_i}) & \text{ if } \l_i \le 0 \\
(\rho III) + b \cdot (III \t) \in \End^2(\1_{\l+\alpha_i} \E_i \E_i \1_{\l-\alpha_i}) & \text{ if } \l_i \ge 0
\end{cases}
\end{equation}
for some $b \in \k$ and $\rho \in \End^2(\1_{\l-\alpha_i})$ or $\rho \in \End^2(\1_{\l+\alpha_i})$ respectively. 
\end{Proposition}
\begin{proof}
By combining Lemma \ref{LEM:new3} and Corollary \ref{COR:new1} we know that 
$$id = a_1 \cdot (I \t I)T_{ii} + a_2 \cdot T_{ii}(I \t I) + T_{ii}(\tau)$$
for some nonzero $a_1, a_2 \in \k$. Composing with $T_{ii}$ on the left one gets $T_{ii} = a_1 \cdot T_{ii}(I \t I)T_{ii}$ while composing on the right gives $T_{ii} = a_2 \cdot T_{ii}(I \t I)T_{ii}$. Thus $a_1=a_2$ and the result follows. 
\end{proof}

\begin{Corollary}\label{COR:new}
There exists $\E_i^{(2)} \1_{\l-\alpha_i}$ so that $\E_i \E_i \1_{\l-\alpha_i} \cong \oplus_{[2]} \E_i^{(2)} \1_{\l-\alpha_i}$ and likewise for $\F_i$'s. 
\end{Corollary}
\begin{proof}
Suppose $\l_i \ge 0$ (the case $\l_i \le 0$ is the same). In the notation of Proposition \ref{PROP:new1} define
$$X_i I_\l := (\rho II) + a(II \t) \in \End(\1_{\l+\alpha_i} \E_i \1_{\l}) \ \ \text{ and } \ \ I_\l X_i := a (\t II) + b(II \t) \in \End(\1_\l \E_i \1_{\l-\alpha_i}).$$
Then it is easy to check that the affine nilHecke relations
$$id = T_{ii}(X_iI) - (IX_i)T_{ii} = (X_iI)T_{ii} - T_{ii}(IX_i) \in \End(\E_i \E_i)$$ 
holds. It is now a standard fact that $\E_i \1_\l \E_i \cong \E_i^{(2)} \1_{\l-\alpha_i} \la -1 \ra \oplus \E_i^{(2)} \1_{\l-\alpha_i} \la 1 \ra$. More precisely, the two summands are given by orthogonal idempotents $T_{ii}(X_iI)$ and $-(IX_i)T_{ii}$. The isomorphism between them is given by $T_{ii}$ and $(X_iI)$ (which explains why they lie in different degrees). The same result for $\F_i$'s follows by adjunction. 
\end{proof}

\begin{Remark} 
The definitions of $X_i$ above are not the ones we will use later. They are only used here in order to prove the decomposition of $\E_i \E_i$. 
\end{Remark}

\section{Step \#2 -- Nonvanishing of $T_{iij}$}

The next step is to show that $T_{iij},T_{jii}$ as well as $T'_{iij},T'_{jii}$ are all nonzero. 

\begin{Proposition}\label{PROP:homiij} 
Suppose $i,j \in I$ are distinct. Then $T_{jii} I_\l, T'_{jii} I_\l, T_{iij} I_\l$ and $T'_{iij} I_\l$ are all nonzero if and only if $\1_{\l+\eps_i \alpha_i + \eps_j \alpha_j} \ne 0$ for $\eps_i \in \{0,1,2\}$ and $\eps_j \in \{0,1\}$. 
\end{Proposition}
\begin{proof}
One direction is immediate since one can easily check that the weights $\l+\eps_i \alpha_i + \eps_j \alpha_j$ all appear in every one of the compositions $T_{jii} I_\l, T'_{jii} I_\l, T_{iij} I_\l$ or $T'_{iij} I_\l$.

For the converse we prove the case of $T_{jii} I_\l$ (the other cases are the same). If $\la i,j \ra = 0$ then $T_{ij}$ and $T_{ji}$ are invertible and the claim is obvious. If $\la i,j \ra = -1$ then using Corollary \ref{COR:nonzero} it suffices to prove that the composition
$$\E_i \E_i \F_j \1_{\l+\alpha_j} \xrightarrow{Iv_{ij}} \E_i \F_j \E_i \1_{\l+\alpha_j} \xrightarrow{v_{ij}I} \F_j \E_i \E_i \1_{\l+\alpha_j}$$
is nonzero. Since $v_{ij}$ is an isomorphism it remains to show that $\E_i \E_i \F_j \1_{\l+\alpha_j} \ne 0$. If $\l_i \le -2$ then we compose on the left with $\F_i \F_i$ and simplify to get several copies of $\F_j \1_{\l+\alpha_j}$. This is nonzero since $\E_j \1_\l \ne 0$. Similarly, if $\l_i \ge -1$ then we compose $\E_i \E_i \F_j \1_{\l+\alpha_j} \cong \F_j \1_{\l+\alpha_j+2\alpha_i} \E_i \E_i$ on the right with $\F_i \F_i$ to obtain several copies of $\1_{\l+2\alpha_i} \F_j$ which are again nonzero. 
\end{proof}

\begin{Corollary}\label{COR:4}
For $i,j \in I$ we have $T_{iij} \sim T'_{iij}$ and $T_{jii} \sim T'_{jii}$. 
\end{Corollary}
\begin{proof}
To prove that $T_{jii} \sim T'_{jii}$ suppose $\la i,j \ra = -1$ (the case $\la i,j \ra = 0$ is clear). By Lemma \ref{LEM:temp} we know $\dim \Hom(\E_j \E_i^{(2)} \1_\l, \E_i^{(2)} \E_j \la d \ra)$ is zero if $d < 2$ and at most one-dimensional if $d=2$. From this it follows that $\dim \Hom(\E_j \E_i \E_i \1_\l, \E_i \E_i \E_j \1_\l) \le 1$. On the other hand, we know that 
$$T_{jii} I_\l, T'_{jii} I_\l \in \Hom(\E_j \E_i \E_i \1_\l, \E_i \E_i \E_j \1_\l)$$ 
are nonzero by Lemma \ref{PROP:homiij}. So we must have $T_{jii} I_\l \sim T'_{jii} I_\l$. The proof that $T_{jii} I_\l \sim T'_{jii} I_\l$ is the same. 
\end{proof}

\begin{Corollary}\label{COR:5}
Suppose $i,j \in I$ and $\t \in Y_\k$ so that $\la \t, \alpha_j \ra = 0$ and $(\t, \alpha_i)_\l \ne 0$. 
\begin{enumerate}
\item If $\l_j \ge -1$ then $(II \t) = (\gamma II) \in \End(\1_{\l+\alpha_j} \E_j \1_\l)$ for some $\gamma \in \End^2(\1_{\l+\alpha_j})$. Moreover, if $\E_i \E_i \E_j \1_{\l-\alpha_i} \ne 0$ then $(\gamma, \alpha_i)_{\l+\alpha_j} \ne 0$. 
\item If $\l_j \le 1$ then $(\t II) = (II \gamma) \in \End(\1_\l \E_j \1_{\l-\alpha_i})$ for some $\gamma \in \End^2(\1_{\l-\alpha_j})$. Morevover, if $\1_{\l+\alpha_i} \E_j \E_i \E_i \ne 0$ then $(\gamma, \alpha_i)_{\l-\alpha_j} \ne 0$.
\end{enumerate}
\end{Corollary}
\begin{proof}
Suppose $\l_i \ge -1$. The fact that $(II \t) = (\gamma II)$ for some $\gamma \in \End^2(\1_{\l+\alpha_j})$ follows from Lemma \ref{LEM:EndE}. It remains to show that $(\gamma, \alpha_i)_{\l+\alpha_j} \ne 0$. To do this consider the following composition 
\begin{equation}\label{EQ:21}
\E_i \E_i \E_j \xrightarrow{(IT_{ij})(T_{ii}I)} \E_i \E_j \1_\l \E_i \xrightarrow{II \t I} \E_i \E_j \1_\l \E_i \xrightarrow{(IT_{ii})(T_{ji}I)} \E_j \E_i \E_i
\end{equation}
where we omit the grading shifts for convenience. On the one hand this is equal to 
\begin{align*}
(IT_{ii})(II \t I)(T_{iij}) 
&\sim (IT_{ii})(II \t I)(T'_{iij}) \\
&= [(IT_{ii})(II \t I)(IT_{ii})](T_{ij}I)(IT_{ij}) \\
&= (\t, \alpha_i)_\l T'_{iij}.
\end{align*}
Notice that $T_{iij} I_{\l-\alpha_i} \sim T'_{iij} I_{\l-\alpha_i}$ are nonzero since $\E_i \E_i \E_j \1_{\l-\alpha_i} \ne 0$ by assumption and $\E_i \1_\l \E_i\ne 0$ because $(\t, \alpha_i)_\l \ne 0$. 

On the other hand, the composition in (\ref{EQ:21}) is equal to 
$$\E_i \E_i \E_j \xrightarrow{(IT_{ij})(T_{ii}I)} \E_i \1_{\l+\alpha_j} \E_j \E_i \xrightarrow{I \gamma II} \E_i \1_{\l+\alpha_j} \E_j \E_i \xrightarrow{(IT_{ii})(T_{ji}I)} \E_j \E_i \E_i$$
which we can simplify as above to get $(\gamma, \alpha_i)_{\l+\alpha_j} T_{iij}$. Since $(\t, \alpha_i)_\l \ne 0$ we get $(\gamma, \alpha_i)_{\l+\alpha_j} \ne 0$. 

The case $\l_i \le 1$ is proved in the same way.
\end{proof}

As a final consequence of Corollary \ref{COR:5} we have the following result which will be used on several occasions later (particularly in the proof of the Serre relation). 

\begin{Lemma}\label{LEM:rank}
Fix $i,j \in I$ with $\la i,j \ra = -1$. If $\l_i \le -1$ then $\E_j$-ranks of 
$$\F_i \E_i \E_j \1_\l \xrightarrow{IT_{ij}I} \F_i \E_j \E_i \1_\l \la 1 \ra \ \ \text{ and } \ \ \F_i \E_j \E_i \1_\l \xrightarrow{IT_{ji}I} \F_i \E_i \E_j \1_\l \la 1 \ra$$ 
are both $-\l_i$. Similarly, if $\l_i \ge 0$ then the $\E_j$-ranks of 
$$\E_i \E_j \1_\l \F_i \xrightarrow{T_{ij}II} \E_j \E_i \1_\l \F_i \la 1 \ra \ \ \text{ and } \ \ \E_j \E_i \1_\l \F_i \xrightarrow{T_{ji}II} \E_i \E_j \1_\l \F_i \la 1 \ra$$ 
are both $\l_i+1$. 
\end{Lemma}
\begin{proof}
Let us consider the first map 
\begin{equation}\label{EQ:map}
IT_{ij}: \F_i \E_i \E_j \1_\l \rightarrow \F_i \E_j \E_i \1_\l
\end{equation}
(we will omit the grading shifts for convenience). Fix $\t \in Y_\k$ so that $\la \t, \alpha_i \ra \ne 0$ but $\la \t, \alpha_j \ra = 0$. The left side of (\ref{EQ:map}) is isomorphic to $\E_i \E_j \F_i \1_\l \oplus_{[-\l_i+1]} \E_j \1_\l$ where, by Corollary \ref{LEM:bubbles2}, the isomorphism involving the $\E_j \1_\l$ summands is given by 
\begin{equation}\label{EQ:30A}
\E_j \1_\l \xrightarrow{\adj^i I} \F_i \1_\mu \E_i \E_j \xrightarrow{I \t^{\ell_1} II} \F_i \1_\mu \E_i \E_j
\end{equation}
where $\mu = \l+\alpha_i+\alpha_j$ and $\ell_1 \in \{0,1,\dots,-\l_i\}$. 

Now, suppose $\mu_j \le 1$ (the case $\mu_j \ge -1$ is similar). Then by Corollary \ref{COR:5} we have 
$$(\t II) = (II \gamma) \in \End^2(\1_\mu \E_j \1_{\mu-\alpha_j})$$
for some $\gamma \in \End^2(\1_{\mu-\alpha_j})$ with $(\gamma, \alpha_i)_{\mu-\alpha_j} \ne 0$. 

Now, the right side of (\ref{EQ:map}) is isomorphic to $\E_j \E_i \F_i \1_\l \oplus_{[-\l_i-1]} \E_j \1_\l$ where the isomorphism involving $\E_j \1_\l$ summands is given by 
\begin{equation}\label{EQ:30B}
\F_i \E_j \E_i \1_\l \xrightarrow{u_{ij}I} \E_j \F_i \1_{\l+\alpha_i} \E_i \xrightarrow{II \gamma^{\ell_2} I} \E_j \F_i \1_{\l+\alpha_i} \E_i \xrightarrow{I \adj_i} \E_j \1_\l
\end{equation}
where $\ell_2 \in \{0,1,\dots,-\l_i-1\}$. The composition of (\ref{EQ:30A}) and (\ref{EQ:30B}) is given by 
\begin{align*}
(I \adj_i)(II \gamma^{\ell_2} I)(u_{ij}I)(IT_{ij})(I\t^{\ell_1}II)(\adj^i I)
&= (I \adj_i)(II  \gamma^{\ell_1+\ell_2} I)(u_{ij}I)(IT_{ij})(\adj^i I) \\
&\sim (I \adj_i)(II \gamma^{\ell_1+\ell_2} I)(u_{ij}I)(v_{ji}I)(I \adj^i) \\
&\sim (I \adj_i)(II \gamma^{\ell_1+\ell_2} I)(I \adj^i)
\end{align*}
where we used the second relation in Lemma \ref{LEM:fork} to get the second line. This last map is the composition 
$$\E_j \1_\l \xrightarrow{I \adj^i} \E_j \F_i \1_{\l+\alpha_i} \E_i \xrightarrow{II \gamma^{\ell_1+\ell_2} I} \E_j \F_i \1_{\l+\alpha_i} \E_i \xrightarrow{I \adj_i} \E_j \1_\l.$$
By degree reasons this composition is zero if $\ell := \ell_1+\ell_2 < -\l_i-1$ and, by Lemma \ref{LEM:bubbles1}, a nonzero multiple of the identity if $\ell = -\l_i-1$. 

Now, $IT_{ij}$ from (\ref{EQ:map}) induces a map between summands $\E_j \1_\l$ on both sides which is represented by a matrix of size $(-\l_i+1) \times (-\l_i)$ whose $(\ell_1,\ell_2)$-entry is the composition of (\ref{EQ:30A}) and (\ref{EQ:30B}). By the calculation above this matrix is upper triangular and carries invertible maps on the (almost) diagonal. Thus it induces a map whose $\E_j$-rank is $-\l_i$. This proves that the map
$$\F_i \E_i \E_j \1_\l \xrightarrow{IT_{ij}I} \F_i \E_j \E_i \1_\l \la 1 \ra$$
has $\E_j$-rank $-\l_i$. The $\E_j$-ranks of the other three maps are calculated similarly. 
\end{proof}

\begin{Lemma}\label{LEM:rank2}
Fix $i,j \in I$ with $\la i,j \ra = 0$. If $\l_i \le -1$ then $\E_j$-ranks of
$$\F_i \E_i \E_j \1_\l \xrightarrow{IT_{ij}I} \F_i \E_j \E_i \1_\l \ \ \text{ and } \ \ \F_i \E_j \E_i \1_\l \xrightarrow{IT_{ji}I} \F_i \E_i \E_j \1_\l$$
are both $-\l_i$. Similarly, if $\l_i \ge -1$ then the $\E_j$-ranks of
$$\E_i \E_j \1_\l \F_i \xrightarrow{T_{ij}II} \E_j \E_i \1_\l \F_i \ \ \text{ and } \ \ \E_j \E_i \1_\l \F_i \xrightarrow{T_{ji}II} \E_i \E_j \1_\l \F_i $$
are both $\l_i+2$.
\end{Lemma}
\begin{proof}
The proof is similar to that of Lemma \ref{LEM:rank} (the only difference is keeping track of the gradings correctly given that $\la i,j \ra = 0$). 
\end{proof}

\begin{Corollary}
If $\la i,j \ra = 0$ then $T_{ij}: \E_i \E_j \1_\l \rightarrow \E_j \E_i \1_\l$ is an isomorphism. 
\end{Corollary}
\begin{proof}
Suppose $\l_i \ge -1$ (the case $\l_i \le 1$ is the same). By Lemma \ref{LEM:rank2} the map $T_{ij}II: \E_i \E_j \1_\l \F_i \rightarrow \E_j \E_i \1_\l \F_i$ induces an isomorphism between all $\l_i+2 > 0$ summands $\E_i$ on either side. Likewise for $T_{ji}II: \E_j \E_i \1_\l \F_i \rightarrow \E_i \E_j \1_\l \F_i$. Thus the composition $T_{ji}T_{ij} \in \End(\E_i \E_j \1_\l)$ is nonzero. 

Since $\dim \End(\E_i \E_j \1_\l) \le 1$ by Lemma \ref{LEM:Ts2} it follows $T_{ji}T_{ij}$ is some multiple of the identity. Similarly we can argue that $T_{ij}T_{ji}$ is also some multiple of the identity. It follows that $T_{ij}$ must be an isomorphism. 
\end{proof}

\section{Step \#3 -- Rescaling $T_{ii}$}\label{sec:2}

Fix $i \in I$. We will now explain how to rescale each $T_{ii} I_\l$ so that for any $j \ne i$ we have
\begin{framed}
\begin{align}
\label{EQ:1} (T_{ii}I)(IT_{ii})(T_{ii}I) I_\l = (IT_{ii})(T_{ii}I)(IT_{ii}) I_\l &\in \End^{-6}(\E_i \E_i \E_i \1_\l) \\
\label{EQ:2} (IT_{ji})(T_{ji}I)(IT_{ii}) I_\l = (T_{ii}I)(IT_{ji})(T_{ji}I) I_\l &\in \Hom^{- 2 - 2 \la i,j \ra}(\E_j \E_i \E_i \1_\l, \E_i \E_i \E_j \1_\l) \\
\label{EQ:3} (T_{ij}I)(IT_{ij})(T_{ii}I) I_\l = (IT_{ii})(T_{ij}I)(IT_{ij}) I_\l &\in \Hom^{- 2 - 2 \la i,j \ra}(\E_i \E_i \E_j \1_\l, \E_j \E_i \E_i \1_\l).
\end{align}
\end{framed}

\noindent The key observation that allows us to do this is that given $T_{ii} I_\l$ there is a unique way of rescaling $T_{ii} I_{\l+\alpha_j}$ so that (\ref{EQ:1}) holds if $i=j$ and (\ref{EQ:2}), (\ref{EQ:3}) hold if $i \ne j$. We denote this rescaling $T_{ii} I_\l \leadsto T_{ii} I_{\l+\alpha_j}$. Similarly one also has the rescaling $T_{ii} I_{\l+\alpha_j} \leadsto T_{ii} I_\l$. 

\begin{Proposition}\label{PROP:rescale}
One can rescale all $T_{ii} I_\l$ so that relations (\ref{EQ:1}), (\ref{EQ:2}) and (\ref{EQ:3}) hold. 
\end{Proposition}
\begin{proof}
Consider a sequence of rescalings
\begin{equation}\label{EQ:24}
T_{ii} I_\l \leadsto T_{ii} I_{\l+c_1 \alpha_{k_1}} \leadsto T_{ii} I_{\l + c_1 \alpha_{k_1} + c_2 \alpha_{k_2}} \leadsto \dots \leadsto T_{ii} I_{\l+ \sum_\ell c_\ell \alpha_{k_\ell}} = T_{ii} I_\l
\end{equation}
where $c_\ell = \pm 1$ and $\sum_\ell c_\ell \alpha_{k_\ell} = 0$. We need to show that this sequence ends up trivially rescaling $T_{ii} I_\l$ (multiplication by one). Let us encode the sequence of rescalings from (\ref{EQ:24}) as $\uck = (c_1 k_1, \dots, c_m k_m)$. There are two operations that one can perform on such a sequence.
\begin{itemize}
\item The switch operation $S_a$ given by 
$$(c_1 k_1, \dots, c_a k_a, c_{a+1}k_{a+1}, \dots,c_m k_m) \mapsto (c_1 k_1, \dots, c_{a+1} k_{a+1}, c_a k_a, \dots, c_m k_m)$$ 
if $c_a+c_{a+1}=0$ and $k_a \ne k_{a+1}$. 
\item The drop operation $D_a$ given by 
$$(c_1 k_1, \dots, c_a k_a, c_{a+1}k_{a+1}, \dots,c_m k_m) \mapsto (c_1 k_1, \dots, c_{a-1}k_{a-1}, c_{a+2} k_{a+2}, \dots, c_m k_m)$$ 
if $c_a k_a + c_{a+1} k_{a+1} = 0$. 
\end{itemize}
The way $\uck$ and $S_a \cdot \uck$ ends up rescaling $T_{ii} I_\l$ are the same due to Corollary \ref{COR:rescale}. The way $\uck$ and $D_a \cdot \uck$ rescales $T_{ii} I_\l$ are also the same (essentially by definition). Finally, it is easy to see that the actions of $S_a$ and $D_a$ can be used to transform any sequence $\underline{c \alpha}$ into the trivial sequence of length zero. This completes the proof.
\end{proof}

\begin{Proposition}\label{PROP:homijk}
Suppose $i,j,k \in I$ are distinct. Then $T_{ijk} I_\l$ and $T'_{ijk} I_\l$ are nonzero if and only if $\1_{\l+\eps_i\alpha_i+\eps_j\alpha_j+\eps_k\alpha_k} \ne 0$ for $\eps_i,\eps_j,\eps_k \in \{0,1\}$. 
\end{Proposition}
\begin{proof}
One direction is immediate since one can easily check that the weights $\l+\eps_i\alpha_i+\eps_j\alpha_j+\eps_k\alpha_k$ all appear in both $T_{ijk} I_\l$ and $T'_{ijk} I_\l$. 

For the converse there are various cases to consider depending on whether $i,j,k \in I$ are joined by an edge. We will consider the most difficult case when $\la i,j \ra = \la i,k \ra = \la j,k \ra = -1$ as the other cases are similar but easier. Note that $\l_i+\l_j+\l_k > 0$ so at least one of $\l_i,\l_j,\l_k$ must be positive. 

To show $T_{ijk} I_\l \ne 0$ it suffices to show, using Corollary \ref{COR:nonzero}, that the following composition is nonzero 
$$\E_j \E_k \F_i \1_{\l+\alpha_i} \xrightarrow{Iv_{ki}} \E_j \F_i \E_k \1_{\l+\alpha_i} \xrightarrow{v_{ji}I} \F_i \E_j \E_k \1_{\l+\alpha_i} \xrightarrow{IT_{jk}} \F_i \E_k \E_j \1_{\l+\alpha_i}.$$
Since $v_{ki}$ and $v_{ji}$ are isomorphisms we conclude that 
\begin{equation}\label{eq:ijk1}
T_{ijk} I_\l \ne 0 \Leftrightarrow IT_{jk}: \F_i \E_j \E_k \1_{\l+\alpha_i} \rightarrow \F_i \E_k \E_j \1_{\l+\alpha_i} \la 1 \ra \text{ is nonzero.}
\end{equation}
If $\la \l+\alpha_i+\alpha_j+\alpha_k, \alpha_i \ra \ge 1 \Leftrightarrow \l_i \ge 1$ then we can compose with $\E_i$ and simplify to get that $IT_{jk}$ is nonzero if $T_{jk}$ is nonzero. This is nonzero since $\1_{\l+\alpha_i}, \1_{\l+\alpha_i+\alpha_j}, \1_{\alpha_i+\alpha_k}, \1_{\alpha_i+\alpha_j+\alpha_k}$ are all nonzero. 

On the other hand, if $\la \l+\alpha_i+\alpha_j,\alpha_j \ra \ge 2 \Leftrightarrow \l_j \ge 1$ we can compose with $\F_j$ on the right to get 
$$IT_{jk}I: \F_i \E_j \E_k \F_j \1_{\l+\alpha_i+\alpha_j} \rightarrow  \F_i \E_k \E_j \F_j \1_{\l+\alpha_i+\alpha_j}.$$
By Lemma \ref{LEM:rank} the $\F_i \E_k$-rank of this map is $\l_j > 0$ which proves it is nonzero. 

Finally, if $\l_k \ge 1$ we compose with $\F_k$ instead of $\F_j$ on the right and find that the $\F_i \E_j$-rank is positive. The result follows. 
\end{proof}

\begin{Corollary}\label{COR:6}
For distinct $i,j,k \in I$ we have $T_{ijk} \sim T'_{ijk}$. 
\end{Corollary}
\begin{proof}
This follows from Proposition \ref{PROP:homijk} together with Lemma \ref{LEM:homijk} which states that 
$$\dim \Hom(\E_i \E_j \E_k \1_\l, \E_k \E_j \E_i \1_\l \la \ell_{ijk} \ra) \le 1.$$
\end{proof}

\begin{Lemma}\label{LEM:iiab}
Suppose $i,a,b \in I$ with $a \ne b$ and $b \ne i$ and let $\ell_{iiab} = 2 \la i,a+b \ra + 2$. Then 
\begin{equation}\label{EQ:iiab}
(IIT_{ii})(IT_{ia}I)(IIT_{ia})(T_{ib}II)(IT_{ib}I)(IIT_{ab}) : \E_i \E_i \E_a \E_b \1_\l \rightarrow \E_b \E_a \E_i \E_i \1_\l \la - \ell_{iiab} \ra
\end{equation}
is nonzero if and only if $\1_{\l+\eps_i\alpha_i+\eps_a\alpha_a+\eps_b\alpha_b}$ for $\eps_i \in \{0,1,2\}$ and $\eps_a,\eps_b \in \{0,1\}$. 
\end{Lemma}
\begin{proof}
One direction is immediate since one can easily check that the weights $\l+\eps_i\alpha_i+\eps_a\alpha_a+\eps_b\alpha_b$ all appear in the composition (\ref{EQ:iiab}) (after possibly having to use the relations $T_{iia} \sim T'_{iia}, T_{iib} \sim T'_{iib}$ and $T_{iab} \sim T'_{iab}$). 

Let us suppose $a \ne i$ (the case $a=i$ is similar but easier). For the converse consider the composition 
$$\E_i \E_a \E_b \F_i \xrightarrow{\adj^i IIII} \F_i \E_i \E_i \E_a \E_b \F_i \xrightarrow{I {\tiny{(\ref{EQ:iiab})}}I} \F_i \E_a \E_b \E_i \E_i \F_i \xrightarrow{IIII \adj_i} \F_i \E_a \E_b \E_i$$
where we omit the grading shift for convenience. It suffices to show this is nonzero. By using Lemma \ref{LEM:fork} we can rewrite it as 
$$\E_i \E_a \E_b \F_i \xrightarrow{T'_{iab} I} \E_b \E_a \E_i \F_i \xrightarrow{IIv_{ii}} \E_b \E_a \F_i \E_i \xrightarrow{Iv_{ai}I} \E_b \F_i \E_a \E_i \xrightarrow{v_{bi}II} \F_i \E_b \E_a \E_i.$$
This is just an analogue of Corollary \ref{COR:nonzero}. 

Since $v_{ai}$ and $v_{bi}$ are invertible we are left with showing that 
\begin{equation}\label{EQ:26}
\E_i \E_a \E_b \F_i \1_{\l+\alpha_i} \xrightarrow{T'_{iab} II} \E_b \E_a \E_i \F_i \1_{\l+\alpha_i} \xrightarrow{IIv_{ii}I} \E_b \E_a \F_i \E_i \1_{\l+\alpha_i}
\end{equation}
is nonzero. To do this we apply the same argument again, namely we consider the composition
$$\E_a \E_b \F_i \F_i \xrightarrow{\adj^i IIII} \F_i \E_i \E_a \E_b \F_i \F_i \xrightarrow{I {\tiny{(\ref{EQ:26})}} I} \F_i \E_b \E_a \F_i \E_i \F_i \xrightarrow{IIII \adj_i} \F_i \E_b \E_a \F_i$$
and show it is nonzero. Using the same methods as before this composition is equal to (up to rescaling)
\begin{equation}\label{EQ:27}
\E_a \E_b \F_i \F_i \xrightarrow{IIT_{ii}} \E_a \E_b \F_i \F_i \xrightarrow{T_{ab}II} \E_b \E_a \F_i \F_i \xrightarrow{(v_{bi} II)(Iv_{ai}I)} \F_i \E_b \E_a \F_i
\end{equation}
where the first map $T_{ii} \in \End^{-2}(\F_i \F_i)$ is the adjoint of $T_{ii} \in \End^{-2}(\E_i \E_i)$ (we abuse notation a bit). Note that to show this we need one more relation not included in Lemma \ref{LEM:fork}, namely that
$$\E_i \F_i \F_i \xrightarrow{v_{ii}I} \F_i \E_i \F_i \xrightarrow{I \adj_i} \F_i  \ \ \text{ and } \ \ \E_i \F_i \F_i \xrightarrow{IT_{ii}} \E_i \F_i \F_i \xrightarrow{\adj_i I} \F_i$$
are equal to each other (up to rescaling). However, this is easy to prove as in Lemma \ref{LEM:fork} since the space of maps is one-dimensional and by adjunction neither composition is zero. Finally, since $v_{bi}$ and $v_{ai}$ are invertible, to show that (\ref{EQ:26}) is nonzero it suffices to show that $T_{ab}IT_{ii}: \E_a \E_b \1_\l \F_i \F_i \rightarrow \E_b \E_a \1_\l \F_i \F_i$ is nonzero or equivalently that 
\begin{equation}\label{EQ:28}
T_{ab}II: \E_a \E_b \1_\l \F_i^{(2)} \rightarrow \E_b \E_a \1_\l \F_i^{(2)} 
\end{equation}
is nonzero. If $\l_i \le -2$ then we can compose (\ref{EQ:28}) on the right with $\E_i^{(2)}$ to get several copies of $T_{ab}: \E_b \E_a \1_\l \rightarrow \E_a \E_b \1_\l$ which is nonzero. In fact, this argument can be extended to the case $\l_i=-1$ by composing with $\E_i^{(3)}$ on the right to obtain the summand $T_{ab}I: \E_a \E_b \1_\l \E_i \rightarrow \E_b \E_a \1_\l \E_i$ which is nonzero using the same arguments used to show $T_{abi}$ is nonzero. 

Finally, if $\l_i \ge 0$ then an analogous argument as above reduces the nonvanishing of (\ref{EQ:iiab}) to showing that 
\begin{equation}\label{EQ:29}
IIT_{ab}: \F_i^{(2)} \1_{\mu} \E_a \E_b \rightarrow \F_i^{(2)} \1_\mu \E_b \E_a 
\end{equation} 
is nonzero, where $\mu = \l+2\alpha_i+\alpha_a+\alpha_b$. Composing on the left this time with $\E_i^{(2)}$ we get several copies of $IT_{ab}: \1_\mu \E_a \E_b \rightarrow \1_\mu \E_b \E_a$ (which is straight-forward to check is nonzero) as long as 
$$\la \mu, \alpha_i \ra \ge 2 \Leftrightarrow \l_i \ge -2 - \la a+b,i \ra.$$
This holds if $\l_i \ge 0$ which concludes the proof that (\ref{EQ:iiab}) is nonzero. 
\end{proof}

\begin{Lemma}\label{LEM:rescale}
Suppose $i,a,b \in I$ so that $T_{ii} I_{\l+\eps_a \alpha_a + \eps_b \alpha_b}$ are all nonzero for $\eps_a,\eps_b \in \{0,1\}$. Then the two sequences 
\begin{equation}\label{EQ:rescale}
T_{ii} I_\l \leadsto T_{ii} I_{\l+\alpha_a} \leadsto T_{ii} I_{\l+\alpha_a+\alpha_b} \ \ \text{ and } \ \ T_{ii} I_\l \leadsto T_{ii} I_{\l+\alpha_b} \leadsto T_{ii} I_{\l+\alpha_a+\alpha_b}
\end{equation}
lead to the same rescaling of $T_{ii} I_{\l+\alpha_a+\alpha_b}$. 
\end{Lemma}
\begin{proof}
Assume $a \ne b$ (otherwise the claim is tautological). Suppose $a,b \ne i$ and consider 
\begin{equation}\label{EQ:quad}
(IIT_{ii})(IT_{ia}I)(IIT_{ia})(T_{ib}II)(IT_{ib}I)(IIT_{ab}) : \E_i \E_i \E_a \E_b \1_\l \rightarrow \E_b \E_a \E_j \E_i \1_\l
\end{equation}
where we omit the grading shifts for convenience. On the one hand, we can write this as 
\begin{align}
\nonumber (IT'_{iia})(T_{ib}II)(IT_{ib}I)(IIT_{ab}) =& (IT_{iia})(T_{ib}II)(IT_{ib}I)(IIT_{ab}) \\
\nonumber =& (IT_{ia}I)(IIT_{ia})[(IT_{ii}I)(T_{ib}II)(IT_{ib}I)](IIT_{ab}) \\
\nonumber =& (IT_{ia}I)(IIT_{ia})(T'_{iib}I)(IIT_{ab}) \\
\nonumber =& (IT_{ia}I)(IIT_{ia})(T_{iib}I)(IIT_{ab}) \\
\label{EQ:quad4} =& (IT_{ia}I)(IIT_{ia})(T_{ib}II)(IT_{ib}I)(T_{ii}II)(IIT_{ab})
\end{align}
where we rescaled $T_{ii} I_\l \leadsto T_{ii} I_{\l+\alpha_a} \leadsto T_{ii} I_{\l + \alpha_a+\alpha_b}$ in order to get the first and fourth equalities.

On the other hand, by Corollary \ref{COR:6}, we know that $T'_{iab} I_\l = c_1 T_{iab} I_\l$ and $T'_{iab} I_{\l+\alpha_i} = c_2 T_{jab} I_{\l+\alpha_i}$ for some $c_1, c_2 \in \k^\times$. Thus, one can rewrite (\ref{EQ:quad}) as
$$c_1 c_2 (T_{ab}II)(IIT_{ii})(IT_{ib}I)(IIT_{ib})(T_{ia}II)(IT_{ia}I)$$
where we slide the $T_{ab}$ from the far right to the far left. Then the same sequence of equalities as above shows that this is equal to
\begin{equation}\label{EQ:quad3}
c_1 c_2 (T_{ab}II)(IT_{ib}I)(IIT_{ib})(T_{ia}II)(IT_{ia}I)(T_{ii}II)
\end{equation}
where this time we had to rescale $T_{ii} I_\l \leadsto T_{ii} I_{\l+\alpha_b} \leadsto T_{ii} I_{\l + \alpha_a+\alpha_b}$. Finally, sliding back the $(T_{ab}II)$ we find that (\ref{EQ:quad3}) is equal to (\ref{EQ:quad4}). Since by Lemma \ref{LEM:iiab} the composition in (\ref{EQ:quad}) is nonzero it follows that $c_1c_2=1$ and so the two sequences of rescalings must be the same.

Finally, suppose one of $a,b$ is equal to $i$ (we can assume $a=i$). The same argument as above works since, by Corollary \ref{COR:4}, we have $T'_{iib} I_\l = c_1 T_{iib} I_\l$ and $T'_{iib} I_{\l+\alpha_i} = c_2 T_{iib} I_{\l+\alpha_i}$ for some $c_1,c_2 \in \k^\times$. 
\end{proof}

\begin{Corollary}\label{COR:rescale}
Suppose $i,a,b \in I$ so that $T_{ii} I_\l$ and $T_{ii} I_{\l+\alpha_a-\alpha_b}$ are nonzero. Then the two sequences 
\begin{equation}\label{EQ:rescale2}
T_{ii} I_\l \leadsto T_{ii} I_{\l+\alpha_a} \leadsto T_{ii} I_{\l+\alpha_a-\alpha_b} \ \ \text{ and } \ \ T_{ii} I_\l \leadsto T_{ii} I_{\l-\alpha_b} \leadsto T_{ii} I_{\l+\alpha_a-\alpha_b}
\end{equation}
lead to the same rescaling of $T_{ii} I_{\l+\alpha_a-\alpha_b}$. 
\end{Corollary}
\begin{proof}
This is an immediate consequence of Lemma \ref{LEM:rescale} as long as we can show that $T_{ii} I_{\l+\alpha_a}$ and $T_{ii} I_{\l-\alpha_b}$ are nonzero. By assumption we know $\E_i \E_i \1_\l$ and $\E_i \E_i \1_{\l+\alpha_a-\alpha_b}$ are nonzero. This implies that $\1_{\l+\ell \alpha_i}$ and $\1_{\l+\alpha_a-\alpha_b+\ell \alpha_i}$ are all nonzero for $\ell \in \{0,1,2\}$. By condition \ref{co:new} this means that 
$$\1_{\l+\alpha_a+\ell \alpha_i} \ne 0 \text{ for } \ell \in \{0,1,2\} \Rightarrow \E_i \E_i \1_{\l+\alpha_a} \ne 0 \Rightarrow T_{ii} I_{\l+\alpha_a} \ne 0$$
and likewise $T_{ii} I_{\l-\alpha_b} \ne 0$. 
\end{proof}

\section{Step \#4 -- The definition of $X_i$}\label{sec:Xi}

We will call $\gamma \in \End^2(\1_\l)$ a {\it transient map} if $(\gamma, \alpha_i)_\l = 0$ for all $i \in I$. The reason for this terminology is that by Lemma \ref{LEM:EndE} such maps can be moved past any $\E_i$ or $\F_i$. From hereon until section \ref{sec:transients} we will work modulo transient maps. In particular, results will hold modulo transient maps unless explicitly indicated otherwise. 

\begin{Lemma}\label{LEM:pairing1}
For $\t \in Y_\k$, if $\la \t, \alpha_i \ra = 0$ then $(\t, \alpha_i)_\l = 0$. 
\end{Lemma}
\begin{proof}
Suppose $\l_i \ge 0$ (the case $\l_i \le 0$ is the same). We need to show that $T_{ii}(I \t I) T_{ii} = 0$. This is equivalent to showing that the map $I \t I \in \End^2(\E_i \1_\l \E_i)$ induces zero between the summands $\E_i^{(2)} \1_\l \la 1 \ra$ on either side. To see this suppose it is nonzero and consider the map $I \t II \in \End^2(\E_i \1_\l \E_i \F_i)$. Then the $\E_i$-rank of this map is at least the number of copies of $\E_i$ in $\E_i^{(2)} \F_i \1_\l$ which is $\l_i+1 > 0$. 

On the other hand, the $\E_i$-rank of this map is the same as that of $I \t II \in \End^2(\E_i \1_\l \F_i \E_i)$ which is zero since $\la \t, \alpha_i \ra = 0$ (condition (\ref{co:theta})). This is a contradiction. 
\end{proof}

\noindent {\bf Definition.} For each $i \in I, \l \in X$ rescale the definition of $\alpha_i \in \End^2(\1_\l)$ so that $(\alpha_i, \alpha_i)_\l = 2$. 

\begin{Lemma}\label{LEM:pairing2}
For $\t \in Y_\k$, $i \in I$ and $\l \in X$ we have $(\t, \alpha_i)_\l = \la \t, \alpha_i \ra$. 
\end{Lemma}
\begin{proof}
It suffices to show this for $\t = \alpha_j$ where $j \in I$. First, if $\la j,i \ra = 0$ then by Lemma \ref{LEM:pairing1} we have $(\alpha_j, \alpha_i)_\l=0$. On the other hand, if $\la j,i \ra = -1$ then $\la i+2j, i \ra = 0$ which, by Lemma \ref{LEM:pairing1} again, means that $(\alpha_i+2\alpha_j,\alpha_i)_\l = 0 \Rightarrow (\alpha_j, \alpha_i)_\l = -1$. 
\end{proof}

\begin{Remark}
If $\t \in Y_\k$ belongs to the radical $R_\k \subset Y_\k$ of the bilinear form $\la \cdot, \cdot \ra$ then, by Lemma \ref{LEM:pairing2}, $\t \in \End^2(\1_\l)$ is transient.
\end{Remark}

\begin{Proposition}\label{PROP:1}
If $\la \t, \alpha_i \ra = 0$ then $II\t = \t II \in \End^2(\1_{\l+\alpha_i} \E_i \1_\l)$. 
\end{Proposition}
\begin{proof}
Suppose $\l_i \ge -1$ (the case $\l_i \le -1$ is similar). Then by Corollary \ref{COR:5} we know $(II \t) = (\gamma II)$ for some $\gamma \in \End^2(\1_{\l+\alpha_i})$. Since we are working modulo transient maps it suffices to show that $(\gamma-\t,\alpha_j)_{\l+\alpha_i} = 0$ for all $j \in I$. 

To see this one follows the same argument as in Corollary \ref{COR:5} (the main difference being that now we know $T_{jji} = T'_{jji}$ rather than $T_{jji} \sim T'_{jji}$). More precisely, we consider the composition 
$$\E_j \E_j \E_i \xrightarrow{(IT_{ji})(T_{jj}I)} \E_j \E_i \1_\l \E_j \xrightarrow{II \t I} \E_j \E_i \1_\l \E_j \xrightarrow{(IT_{jj})(T_{ij}I)} \E_i \E_j \E_j$$
where we omit the shifts for convenience. Then, using the same argument as in Corollary \ref{COR:5}, this composition equals $(\t,\alpha_j)_\l T_{jji}$ while on the other hand it equals $(\gamma, \alpha_j)_{\l+\alpha_i} T_{jji}$. Since $(\t,\alpha_j)_\l = \la \t, \alpha_j \ra = (\t,\alpha_j)_{\l+\alpha_i}$ we get $(\gamma-\t,\alpha_j)_{\l+\alpha_i} = 0$ as long as $T_{jji} I_{\l-\alpha_j} \ne 0$.

If $T_{jji} I_{\l-\alpha_j} = 0$ then either $\E_j \E_j \1_{\l+\alpha_i-\alpha_j} = 0$ or $\E_j \E_j \1_{\l-\alpha_j} = 0$. In the first case we have $(\phi, \alpha_j)_{\l+\alpha_i}$ for any $\phi \in \End^2(\l+\alpha_i)$ and so, in particular, $(\t-\gamma, \alpha_j)_{\l+\alpha_i} = 0$. Otherwise we have $\1_{\l+\alpha_i+\eps_j \alpha_j} \ne 0$ for $\eps_j \in \{-1,0,1\}$ and $\1_\l \ne 0$ which implies $\1_{\l-\alpha_j} \ne 0$ by condition \ref{co:new}. Using that $\E_j \E_j \1_{\l-\alpha_j} = 0$ we get $\1_{\l+\alpha_j} = 0$. However, $\1_{\l+\alpha_i+\alpha_j} \ne 0$ which leads to a contradiction if $\la \l+\alpha_i+\alpha_j,i \ra \ge 1 \Leftrightarrow \l_i \ge 0$. 

Thus we are done unless $\l_i=-1$ and $\E_j \E_j \1_{\l-\alpha_j} = 0$. This last case is taken care by the analogous argument for $\l_i \le 1$ which argues in the same way as above starting with the fact that $(\t II) = (II \gamma) \in \End^2(\1_\l \E_i \1_{\l-\alpha_i})$ for some $\gamma \in \End^2(\1_{\l-\alpha_i})$. 
\end{proof}

\begin{framed}
\noindent Definition: Choose $\t \in Y_\k$ so that $\la \t, \alpha_i \ra = 1$ and let $X_i := - (\t II) + (II \t) \in \End^2(\1_{\l+\alpha_i} \E_i \1_\l)$. 
\end{framed}

\begin{Corollary}\label{COR:2}
If $i \ne j$ then we have 
\begin{align*}
(IX_i)T_{ij} &= T_{ij}(X_iI) \in \Hom(\E_i \E_j \1_\l, \E_j \E_i \1_\l \la - \la i,j \ra + 2 \ra), \\
(X_iI)T_{ji} &= T_{ji}(IX_i) \in \Hom(\E_j \E_i \1_\l, \E_i \E_j \1_\l \la - \la i,j \ra + 2 \ra).
\end{align*}
\end{Corollary}
\begin{proof}
Choose $\gamma \in Y_\k$ so that $\la \gamma, \alpha_i \ra = 1$ and $\la \gamma, \alpha_j \ra = 0$. Then $\la \t-\gamma, \alpha_i \ra = 0$ and so by Proposition \ref{PROP:1} we have $(\t II) - (\gamma II) = (II \t) - (II \gamma) \in \End(\1_{\l+\alpha_i} \E_i \1_\l)$. This means that $X_i = -(\gamma II) + (II \gamma)$. But then, since $\la \gamma, \alpha_j \ra = 0$, we have 
\begin{align*}
(I \gamma II) &= (III \gamma ) \in \End^2(\E_i \1_{\l+\alpha_j} \E_j \1_\l) \\ 
(\gamma III) &= (II \gamma I) \in \End^2(\1_{\l+\alpha_i+\alpha_j} \E_j \1_{\l+\alpha_i} \E_i)
\end{align*}
from which we get $T_{ij}(X_i I) = (I X_i)T_{ij}$. The second relation follows similarly. 
\end{proof}

\begin{Remark}
The argument in Corollary \ref{COR:2} shows that the definition of $X_i$ does not depend on our choice of $\t$. 
\end{Remark}

\section{Step \#5 -- Action of the affine nilHecke algebra}\label{sec:1}

Fix $i \in I$. We will now prove that the affine nilHecke algebra on products of $\E_i$'s. This argument is essentially a simplification and strengthening of the one in \cite{CKL2}.

\begin{Proposition}\label{PROP:2}
Inside $\End(\E_i \E_i)$ we have (modulo transient maps)
$$T_{ii}(X_iI) = (IX_i)T_{ii} + II \ \ \text{ and } \ \ (X_iI)T_{ii} = T_{ii}(IX_i) + II.$$
\end{Proposition}
\begin{proof}
Let us suppose we are trying to prove this for $\End(\E_i \1_\l \E_i)$ where $\l_i \ge 0$ (the case $\l_i \le 0$ is the same). By Lemma \ref{PROP:new1} we have a
\begin{equation}\label{EQ:6}
id = a \cdot (I \t I)T_{ii} + a \cdot T_{ii} (I \t I) + T_{ii}(\tau) \in \End(\E_i \1_\l \E_i)
\end{equation}
where $\tau = (\rho III) + b \cdot (III \t) \in \End^2(\1_{\l+\alpha_i} \E_i \E_i \1_{\l-\alpha_i})$. Composing on the left with $T_{ii}$ and using that $T_{ii}(I \t I)T_{ii} = (\t, \alpha_i)_\l T_{ii} = T_{ii}$ we get $T_{ii} = a T_{ii}$ which means $a=1$. Thus, we can rewrite (\ref{EQ:6}) as 
\begin{equation}\label{EQ:7}
(X_iI)T_{ii} = T_{ii}(IX_i) + II + T_{ii}(\tau') \in \End(\E_i \E_i)
\end{equation}
where $\tau' = (\gamma' III) + (III \gamma) \in \End(\1_{\l+\alpha_i} \E_i \E_i \1_{\l-\alpha_i})$. It remains to show that $(\gamma, \alpha_j)_{\l-\alpha_i} = 0 = (\gamma', \alpha_j)_{\l+\alpha_i}$ for all $j \in I$. 

First we show that $(\gamma, \alpha_i)_{\l-\alpha_i} = 0$. To do this consider 
$$(X_iII)(T_{ii}I) = (T_{ii}I)(IX_iI) + III + (T_{ii}I)(\tau' I) \in \End(\1_{\l+\alpha_i} \E_i \E_i \E_i)$$
and compose on the left with $(T_{ii}I)(IT_{ii})$ and on the right with $(IT_{ii})$. This gives us 
\begin{align*}
(T_{ii}I)(IT_{ii})(X_iII)(T_{ii}I)(IT_{ii}) 
=& T_{iii}(IX_iI)(IT_{ii}) + (T_{ii}I)(IT_{ii})(T_{ii}I)(\tau' I)(IT_{ii}) \\
(T_{ii}I)(X_iII)T'_{iii} =& T_{iii}(IX_iI)(IT_{ii}) + T_{iii}(\tau' I)(IT_{ii}) \\
(T_{ii}I)(X_iII)T_{iii} =& T'_{iii}(IX_iI)(IT_{ii}) + T'_{iii}(\tau' I)(IT_{ii}) \\ 
(T_{ii}I)(X_iII)(T_{ii}I)(IT_{ii})(T_{ii}I) =& T'_{iii} + (IT_{ii})(T_{ii}I)(IT_{ii}) \cdot (\gamma, \alpha_i)_{\l-\alpha_i} \\ 
T_{iii} =& T'_{iii} + (\gamma, \alpha_i)_{\l-\alpha_i} T'_{iii} 
\end{align*}
where we use $T_{ii}(X_i I)T_{ii} = T_{ii}$ and that
$$(IT_{ii})(\tau' I)(IT_{ii}) = (IT_{ii})[(\gamma' IIII) + (III \gamma I)](IT_{ii}) = 0 + (\gamma, \alpha_i)_{\l-\alpha_i} T_{ii}$$
to obtain the last two equalities. From this it follows that $(\gamma, \alpha_i)_{\l-\alpha_i} = 0$. The argument that $(\gamma', \alpha_i)_{\l+\alpha_i} = 0$ is similar. 

Next we show that $(\gamma, \alpha_j)_{\l-\alpha_i} = 0$ for $j \ne i$. To do this we first compose (\ref{EQ:7}) with $\F_j$ on either sides. The term involving $\gamma$ then becomes 
$$h := (IT_{ii}II)(III \gamma I) \in \End(\F_j \E_i \E_i \1_{\l-\alpha_i} \F_j).$$
Now consider the composition
\begin{equation}\label{EQ:8}
\E_i \E_i \F_j \1_{\l-\alpha_i} \F_j \xrightarrow{(v_{ij}III)(Iv_{ij}II)} \F_j \E_i \E_i \1_{\l-\alpha_i} \F_j \xrightarrow{h} \F_j \E_i \E_i \1_{\l-\alpha_i} \F_j \xrightarrow{(Iu_{ji}II)(u_{ji}III)} \E_i \E_i \F_j \1_{\l-\alpha_i} \F_j.
\end{equation}
We have
$$(Iu_{ji}I)(u_{ji}II)(IT_{ii}I) \sim (T_{ii}II)(Iu_{ji}I)(u_{ji}II): \F_j \E_i \E_i \1_{\l-\alpha_i} \rightarrow \E_i \E_i \F_j \1_{\l-\alpha_i}$$
which means that the composition in (\ref{EQ:8}) is (up to a multiple) equal to 
$$(T_{ii}III)(Iu_{ji}II)(u_{ji}III)(v_{ij}III)(Iv_{ij}II)(III \gamma I) \sim (T_{ii}III)(III \gamma I) \in \End(\E_i \E_i \F_j \1_{\l-\alpha_i} \F_j).$$
Finally, composing on the left and right with $IIT'_{jj}$ where $T'_{jj} \in \End^{-2}(\F_j\F_j)$ is the unique map we get (up to a multiple) 
$$(\gamma, \alpha_j)_{\l-\alpha_i} (T_{ii}II)(IIT'_{jj}) \in \End(\E_i\E_i\F_j\F_j).$$
On the other hand, performing the same manipulations with the others terms in (\ref{EQ:7}) ends up giving us zero. It follows that $(\gamma, \alpha_j)_{\l-\alpha_i} = 0$. A similar argument also shows $(\gamma', \alpha_j)_{\l+\alpha_i}=0$. 
\end{proof}

Thus, Proposition \ref{PROP:2} and Corollary \ref{COR:2} gives us the following relations (modulo transient maps):
\begin{framed}
\begin{align}
\label{EQ:KLR1}
T_{ii}(X_iI) &= (IX_i)T_{ii} + II \text{ and } (X_iI)T_{ii} = T_{ii}(IX_i) + II \in \End(\E_i \E_i), \\
\label{EQ:KLR2}
(IX_i)T_{ij} &= T_{ij}(X_iI) \in \Hom(\E_i \E_j, \E_j \E_i \la - \la i,j \ra + 2 \ra) \text{ if } i \ne j, \\
\label{EQ:KLR3}
(X_iI)T_{ji} &= T_{ji}(IX_i) \in \Hom(\E_j \E_i, \E_i \E_j \la - \la i,j \ra + 2 \ra) \text{ if } i \ne j.
\end{align}
\end{framed}

\begin{Corollary}\label{COR:bubbles3}
For any $i \in I$ the compositions
\begin{align*}
\1_\l \xrightarrow{\adj^i} \E_i \F_i \1_\l \la -\l_i+1 \ra \xrightarrow{X_i^{\l_i-1} I} \E_i \F_i \1_\l \la \l_i-1 \ra \xrightarrow{\adj_i} \1_\l \ \ & \text{ if } \l_i \ge 0 \ \ \text{ and } \\
\1_\l \xrightarrow{\adj^i} \F_i \E_i \1_\l \la \l_i+1 \ra \xrightarrow{I X_i^{-\l_i-1}} \F_i \E_i \1_\l \la -\l_i-1 \ra \xrightarrow{\adj_i} \1_\l \ \ & \text{ if } \l_i \le 0
\end{align*}
are equal to a nonzero multiple of the identity map in $\End(\1_\l)$.
\end{Corollary}
\begin{proof}
This is an immediate consequence of Lemma \ref{LEM:bubbles1}, the fact that $\End^\ell(\1_\l) = 0$ if $\ell < 0$ and the definition of $X_i$ from section \ref{sec:Xi}.
\end{proof}

\section{Step \#6 -- The Serre relation}

In this section we fix $i,j \in I$ with $\la i,j \ra = -1$ and prove the Serre relation (Corollary \ref{COR:serre3}). The idea behind the proof is to first show that by adjunction there exist unique maps (up to rescaling) 
$$\E_i^{(2)} \E_j \oplus \E_j \E_i^{(2)} \rightarrow \E_i \E_j \E_i \rightarrow \E_i^{(2)} \E_j \oplus \E_j \E_i^{(2)}$$
whose composition is the identity (this last step is the hardest part). Separately, by using adjunction we know that $\dim \End(\E_i \E_j \E_i) \le 2$ from which the Serre relation follows. 

\begin{Lemma}\label{LEM:serre}
Suppose $i,j \in I$ with $\la i,j \ra = -1$. If $\E_j \E_i \1_\l = 0$ then $\E_j \E_i^{(2)} \1_\l = 0$. Likewise, if $\E_i \E_j \1_{\l+\alpha_i} = 0$ then $\E_i^{(2)} \E_j \1_\l = 0$. 
\end{Lemma}
\begin{proof}
We prove the first assertion (the second assertion is similar). If $\1_\l, \1_{\l+\alpha_i}, \1_{\l+2\alpha_i}$ or $\1_{\l+2\alpha_i+\alpha_j}$ are zero then we are done. So suppose they are all nonzero. 

Since $\E_j \E_i \1_\l = 0$ this means that $\1_\mu = 0$ where $\mu := \l+\alpha_i+\alpha_j$. Since $\1_{\mu+\alpha_i} \ne 0$ we must have $\la \mu+\alpha_i, \alpha_i \ra \le 0 \Leftrightarrow \l_i \le -3$ because otherwise $\1_{\mu+\alpha_i}$ is a direct summand of $\E_i \1_\mu \F_i = 0$.

Finally, if $\l_i \le -3$ then $\E_j \E_i^{(2)} \1_\l$ is a direct summand of 
$$\E_j \F_i^{(-\l_i-2)} \E_i^{(-\l_i)} \1_\l \cong \F_i^{(-\l_i-2)} \1_{s_i \cdot \mu} \E_j \E_i^{(-\l_i)} \1_\l.$$
Since, in general, $\1_\nu = 0$ if and only if $\1_{s_i \cdot \nu} = 0$ we find that the right side above is zero and hence $\E_j \E_i^{(2)} \1_\l = 0$. 
\end{proof}

\begin{Proposition}\label{PROP:Tijinonzero}
If $\E_j \E_i^{(2)} \1_\l \ne 0$ then $(T_{ji}I)(IT_{ii})(T_{ij}I) \in \End(\E_i \E_j \E_i \1_\l)$ is nonzero. Likewise, if $\E_i^{(2)} \E_j \1_\l \ne 0$ then $(IT_{ij})(T_{ii}I)(IT_{ji}) \in \End(\E_i \E_j \E_i \1_\l)$ is nonzero. Moreover, these two maps are linearly independent. 
\end{Proposition}
\begin{proof}
Suppose $(T_{ji}I)(IT_{ii})(T_{ij}I) \in \End(\E_i \E_j \E_i \1_\l)$ is zero. Then composing with $(X_iII)$ we get that 
\begin{align*}
(T_{ji}I)(IT_{ii})(T_{ij}I)(X_iII) &= (T_{ji}I)(IT_{ii})(IX_iI)(T_{ij}I) \\
&= (T_{ji}I)(X_iII)(IT_{ii})(T_{ij}I) + (T_{ji}I)(T_{ij}I) \\
&= (X_iII)(T_{ji}I)(IT_{ii})(T_{ij}I) + (T_{ji}I)(T_{ij}I)
\end{align*}
must be zero, which implies that $(T_{ji}I)(T_{ij}I) \in \End^2(\E_i \E_j \E_i \1_\l)$ is zero. 

Now, if $\la \l+\alpha_i, \alpha_i \ra \le -1$ (meaning $\l_i \le -3$) then we compose on the left with $\F_i$ and by Lemma \ref{LEM:rank} find that the $\E_j \E_i \1_\l$-rank of $(IT_{ji}I)(IT_{ij}I) \in \End^2(\F_i \E_i \E_j \E_i \1_\l)$ is $-\l_i-2 \ge 1$. Since by Lemma \ref{LEM:serre} $\E_j \E_i^{(2)} \1_\l \ne 0 \Rightarrow \E_j \E_i \1_\l \ne 0$ this means that $(T_{ji}I)(T_{ij}I)$ cannot be zero (contradiction). 

Likewise, if $\la \l+\alpha_i, \alpha_i \ra \ge 0$ (meaning $\l_i \ge -2$) then we compose with $\F_i$ on the right and by Lemma \ref{LEM:rank} find that the $\E_j \E_i$-rank of $(T_{ji}II)(T_{ij}II) \in \End^2(\E_i \E_j \E_i \1_\l \F_i)$ is at least $\l_i+3 \ge 1$ so once again $(T_{ji}I)(T_{ij}I)$ cannot be zero (contradiction).

The case of $(IT_{ij})(T_{ii}I)(IT_{ji})$ is proved similarly. 

Finally, suppose $(T_{ji}I)(IT_{ii})(T_{ij}I) = c (IT_{ij})(T_{ii}I)(IT_{ji}) $ for some $c \in \k^\times$. Composing with $(X_iII)$ as above and simplifying we get that 
$$(T_{ji}I)(T_{ij}I) = c (IT_{ij})(IT_{ji}) \in \End^2(\E_i \E_j \E_i \1_\l).$$
If $\l_i \le -3$ then we compose with $\F_i$ on the left and find that the left hand side has positive $\E_j \E_i \1_\l$-rank. However, the right hand side factors through $\F_i \E_i \E_i \E_j \1_\l$ and hence must have zero $\E_j \E_i \1_\l$-rank. 
\end{proof}

\begin{Corollary}\label{COR:serre2}
There exist $m^\l_{iji}, n^\l_{iji} \in \k^\times$ so that
\begin{equation}\label{EQ:RIII}
m^\l_{iji} (T_{ji}I)(IT_{ii})(T_{ij}I) + n^\l_{iji} (IT_{ij})(T_{ii}I)(IT_{ji}) = (III) \in \End(\E_i \E_j \E_i \1_\l).
\end{equation}
\end{Corollary}
\begin{Remark} 
This result is true without having to mod out by transient maps. This is because the key result that $\dim \End(\E_i \E_j \E_i \1_\l) \le 2$ holds without having to mod out.
\end{Remark}
\begin{proof}
Let us suppose $\E_i^{(2)} \E_j \1_\l$ and $\E_j \E_i^{(2)} \1_\l$ are both nonzero (the other cases are strictly easier). Recall that $T_{iji} = (T_{ji}I)(IT_{ii})(T_{ij}I)$ and $T'_{iji} = (IT_{ij})(T_{ii}I)(IT_{ji})$. By Corollary \ref{PROP:Tijinonzero} we know $T_{iji} I_\l$ and $T'_{iji} I_\l$ are nonzero and linearly independent. On the other hand, by Lemma \ref{LEM:3}, we have $\dim \End(\E_i \E_j \E_i \1_\l) \le 2$. Thus $T_{iji} I_\l$ and $T'_{iji} I_\l$ span this space and the identity morphism must be a linear combination of these two. 

Finally, to show that $m^\l_{iji}$ and $n^\l_{iji}$ are nonzero note that by Lemma \ref{LEM:4} we have $m^\l_{iji} T_{iji}^2 I_\l = T_{iji} I_\l$ and, in particular, $m_{iji}^\l \ne 0$. Likewise, we also get $n^\l_{iji} \ne 0$. 
\end{proof}

\begin{Lemma}\label{LEM:4}
We have $T_{iji} T'_{iji} I_\l = 0$ while $m_{iji}^\l (T_{iji})^2 I_\l = T_{iji} I_\l$ and $n_{iji}^\l (T'_{iji})^2 I_\l = T'_{iji} I_\l$.
\end{Lemma}
\begin{proof}
We have 
$$T_{iji} T'_{iji} = (T_{ji}I)(IT_{ii})[(T_{ij}I)(IT_{ij})(T_{ii}I)](IT_{ji}) = (T_{ji}I)(IT_{ii})[(IT_{ii})(T_{ij}I)(IT_{ij})](IT_{ji}) = 0$$
where we used relation (\ref{EQ:3}) to rewrite the part in the brackets. Multiplying (\ref{EQ:RIII}) by $T_{iji}$ we get $m_{iji}^\l T_{iji}^2 I_\l = T_{iji} I_\l$ and likewise for the last relation. 
\end{proof}

\begin{Corollary}\label{COR:serre3}
We have $\E_i \E_j \E_i \cong \E_i^{(2)} \E_j \oplus \E_j \E_i^{(2)}$.
\end{Corollary}
\begin{proof}
Suppose $\E_j \E_i^{(2)} I_\l$ and $\E_j \E_i^{(2)} I_\l$ are nonzero (the cases when one or both are zero is strictly easier).  First consider the composition
\begin{equation}\label{EQ:22}
\E_j \E_i^{(2)} \1_\l \xrightarrow{(T_{ji}I)(I \iota)} \E_i \E_j \E_i \1_\l \xrightarrow{(I \pi)(T_{ij}I)} \E_j \E_i^{(2)} \1_\l
\end{equation}
where $\iota: \E_i^{(2)} \rightarrow \E_i \E_i \la -1 \ra$ is the natural inclusion and $\pi: \E_i \E_i \rightarrow \E_i^{(2)} \la -1 \ra$ the natural projection. 

Claim: the composition in (\ref{EQ:22}) is nonzero. Note that $\iota \circ \pi \in \End^{-2}(\E_i \E_i)$ is equal to $T_{ii}$ up to rescaling. So it suffices to show that the following composition is nonzero
$$\E_j\E_i\E_i \1_\l \xrightarrow{(IT_{ii})} \E_j\E_i\E_i \1_\l \xrightarrow{(T_{ij}I)(T_{ji}I)} \E_j \E_i \E_i \1_\l \xrightarrow{(IT_{ii})} \E_j \E_i \E_i \1_\l$$ 
where we omit shifts for convenience. This composition appears as a factor inside $T_{iji} T_{iji} \1_\l$ so it suffices to show $T_{iji} T_{iji} I_\l \ne 0$. This follows from Lemma \ref{LEM:4} since $m_{iji}^\l T_{iji} T_{iji} I_\l = T_{iji} I_\l$ and completes the proof of the claim. 

Since (\ref{EQ:22}) is nonzero it must be some multiple of the identity (by Lemma \ref{LEM:temp}). Thus $\E_j \E_i^{(2)} \1_\l$ is a direct summand of $\E_i \E_j \E_i \1_\l$. Likewise, one can also prove that $\E_i^{(2)} \E_j \1_\l$ is a direct summand of $\E_i \E_j \E_i \1_\l$. Thus 
$$\E_i \E_j \E_i \1_\l \cong \E_j \E_i^{(2)} \1_\l \oplus \E_i^{(2)} \E_j \1_\l \oplus {\sf R} \1_\l$$
for some 1-morphism ${\sf R}$. But by Lemma \ref{LEM:3} we know $\dim \End(\E_i \E_j \E_i \1_\l) \le 2$ which means ${\sf R} \1_\l = 0$ and the result follows. 
\end{proof}

\section{Step \#7 -- The $T_{iji}$ relation}

In this section we consider $i,j \in I$ such that $\la i,j \ra = -1$. We fix a subset $S_{ij} \subset I \setminus \{i,j\}$ such that $\{\alpha_k: k \in S_{ij}\} \cup \{\alpha_i,\alpha_j\}$ give a basis of $Y_\k/R_\k$ where $R_\k \subset Y_\k$ denotes the radical. By construction this means that $\la \cdot, \cdot \ra$ is nondegenerate on the subspace of $Y_\k$ spanned by $\{\alpha_k: k \in S_{ij}\} \cup \{\alpha_i,\alpha_j\}$. 

\begin{Lemma}\label{LEM:Tij}
We have
$$T_{ji} T_{ij} I_\l = (\phi^{ij}_\l III) + (III \t^{ij}_\l) : \1_{\l+\alpha_i+\alpha_j} \E_i \E_j \1_\l \rightarrow \1_{\l+\alpha_i+\alpha_j} \E_i \E_j \1_\l \la 2 \ra $$
where $\t^{ij}_\l$ and $\phi^{ij}_\l$ satisfy $(\t^{ij}_\l,\alpha_k)_\l + (\phi^{ij}_\l,\alpha_k)_{\l+\alpha_i+\alpha_j} = 0$ for any $k \in S_{ij}$. 
\end{Lemma}
\begin{proof}
Any map $\1_{\l+\alpha_i+\alpha_j} \E_i \E_j \1_\l \rightarrow \1_{\l+\alpha_i+\alpha_j} \E_i \E_j \1_\l \la 2 \ra$ is of the form $\phi_\l^{ij} III + III \t_\l^{ij}$ for some $\t_\l^{ij} \in \End^2(\1_\l)$ and $\phi_\l^{ij} \in \End^2(\1_{\l+\alpha_i+\alpha_j})$. Now consider $k \in S_{ij}$. By nondegeneracy of $\la \cdot, \cdot \ra$ on $Y_\k/R_\k$ we can find $\gamma \in Y_\k$ so that $\la \gamma, \alpha_k \ra = 1$ while $\la \gamma, \alpha_i \ra = \la \gamma, \alpha_j \ra = 0$. 

Now consider the composition 
\begin{equation}\label{EQ:15}
\E_i \E_j \F_k \F_k \xrightarrow{(v_{ik}II)(Iv_{jk}I)} \F_k \E_i \E_j \F_k \xrightarrow{(IT_{ji}I)(IT_{ij}I)} \F_k \E_i \E_j \F_k \la 2 \ra \xrightarrow{(Iu_{kj}I)(u_{ki}II)} \E_i \E_j \F_k \F_k \la 2 \ra 
\end{equation}
On the one hand, since $k \ne i,j$ we can apply Corollary \ref{cor:homijk} to conclude that, 
$$(IT_{ij}I)(v_{ik}II)(Iv_{jk}I) \sim (v_{ik}II)(Iv_{jk}I)(T_{ij}II) \ \ \text{ and } \ \ (Iu_{kj}I)(u_{ki}II)(IT_{ji}I) \sim (T_{ji}II)(Iu_{kj}I)(u_{ki}II).$$
This means that, up to a multiple, the composition in (\ref{EQ:15}) is equal to 
$$\E_i \E_j \F_k \F_k \xrightarrow{(v_{ik}II)(Iv_{jk}I)(T_{ij}II)} \F_k \E_j \E_i \F_k \xrightarrow{(T_{ji}II)(Iu_{kj}I)(u_{ki}II)} \E_i \E_j \F_k \F_k.$$
Since $u_{kj}$ and $v_{jk}$ are inverses (up to a multiple) and likewise for $u_{ki}$ and $v_{ik}$ this composition is (up to a multiple) equal to $(T_{ji}II)(T_{ij}II) \in \End^2(\E_i \E_j \F_k \F_k)$. In particular, this means that composing (\ref{EQ:15}) on the left and right with $(IIT_{kk})$ we get zero. 

On the other hand, we can rewrite (\ref{EQ:15}) as
\begin{align}
\label{EQ:16} 
\E_i \E_j \F_k \F_k \xrightarrow{(v_{ik}II)(Iv_{jk}I)} \F_k \1_\mu \E_i \E_j \1_\l \F_k & \xrightarrow{(I \phi^{ij}_\l IIII) + (IIII \t^{ij}_\l I)} \F_k \1_\mu \E_i \E_j \1_\l \F_k \la 2 \ra \\
\nonumber & \xrightarrow{(Iu_{kj}I)(u_{ki}II)} \E_i \E_j \F_k \F_k \la 2 \ra
\end{align}
where $\mu = \l+\alpha_i+\alpha_j$. Now $(\t^{ij}_\l II) \in \End^2(\1_\l \F_k \1_{\l+\alpha_k})$ can be rewritten as $(II \rho) + a (\gamma II)$ for some $\rho \in \End^2(\1_{\l+\alpha_k})$ and where $a := (\t_\l^{ij}, \alpha_k)_\l$. Likewise, $(II \phi^{ij}_\l) \in \End^2(\1_{\mu-\alpha_k} \F_k \1_\mu)$ can we rewritten as $(\rho' II) + b(II \gamma)$ where $b := (\phi^{ij}_\l, \alpha_k)_{\l+\alpha_i+\alpha_j}$. Also, since $\la \gamma, \alpha_i \ra = \la \gamma, \alpha_j \ra = 0$ the map $b (I \gamma II) \in \End^2(\F_k \1_{\l+\alpha_i+\alpha_j} \E_i \E_j)$ is equal to $b (III \gamma) \in \End^2(\F_k \E_i \E_j \1_\l)$. Thus (\ref{EQ:16}) is equal to the composition 
$$\E_i \E_j \F_k \F_k \xrightarrow{(v_{ik}II)(Iv_{jk}I)} \F_k \E_i \E_j \F_k \xrightarrow{h} \F_k \E_i \E_j \F_k \la 2 \ra \xrightarrow{(Iu_{kj}I)(u_{ki}II)} \E_i \E_j \F_k \F_k \la 2 \ra$$
where $h$ is the map 
$$\1_{\mu-\alpha_k} \F_k \E_i \E_j \1_\l \F_k \1_{\l+\alpha_k} \xrightarrow{(\rho' IIIIII) + (IIIIII \rho) + (a+b)(IIII \gamma II)} \1_{\mu-\alpha_k} \F_k \E_i \E_j \1_\l \F_k \1_{\l+\alpha_k} \la 2 \ra.$$
Up to rescaling this is equal to 
$$\1_{\mu-\alpha_k} \E_i \E_j \F_k \1_\l \F_k \1_{\l+\alpha_k} \xrightarrow{(\rho' IIIIII) + (IIIIII \rho) + (a+b)(IIII \gamma II)} \1_{\mu-\alpha_k} \E_i \E_j \F_k \1_\l \F_k \1_{\l+\alpha_k} \la 2 \ra.$$
Composing on both sides with $IIT_{kk} \in \End^{-2}(\E_i \E_j \F_k \F_k)$ we find that we get zero if and only if $a+b = 0$. This concludes the proof. 
\end{proof}

\begin{Lemma}\label{LEM:A}
Using the notation from Corollary \ref{COR:serre2} and Lemma \ref{LEM:Tij} we have
\begin{enumerate}
\item $n_{iji}^\l \cdot (\phi^{ij}_\l, \alpha_i)_{\l+\alpha_i+\alpha_j} = 1$ if $\E_i^{(2)} \E_j \1_\l \ne 0$ and
\item $m_{jij}^{\l-\alpha_j} \cdot (\t^{ij}_\l, \alpha_j)_\l = 1$ if $\E_i \E_j^{(2)} \1_{\l-\alpha_j} \ne 0$. 
\end{enumerate}
\end{Lemma}
\begin{proof}
On the one hand we have $n_{iji}^\l (T'_{iji})^2 I_\l = T'_{iji} I_\l$ by Lemma \ref{LEM:4}. On the other hand,
$$(T'_{iji})^2 I_\l = (IT_{ij})(T_{ii}I)[(IT_{ji})(IT_{ij})](T_{ii}I)(IT_{ji}) I_\l$$
which is equal to the composition 
$$\E_i \E_j \E_i \1_\l \xrightarrow{(T_{ii}I)(IT_{ji})} \E_i \1_{\mu} \E_i \E_j \1_\l \la -1 \ra \xrightarrow{(I \phi^{ij}_\l III) + (IIII \t^{ij}_\l)} \E_i \1_{\mu} \E_i \E_j \1_\l \la 1 \ra \xrightarrow{(IT_{ij})(T_{ii}I)} \E_i \E_j \E_i \1_\l$$
where $\mu = \l+\alpha_i+\alpha_j$. Here we used Lemma \ref{LEM:Tij} to rewrite the part in the brackets. 

Now $(T_{ii}I)(IIII \t^{ij}_\l)(T_{ii}I) = 0$ whereas $(T_{ii}I)(I \phi^{ij}_\l III)(T_{ii}I) = (\phi^{ij}_\l, \alpha_i)_{\l+\alpha_i+\alpha_j} (T_{ii}I)$. Thus 
$$(n_{iji}^\l)^{-1} T'_{iji} I_\l = (\phi^{ij}_\l, \alpha_i)_{\l+\alpha_i+\alpha_j} (IT_{ij})(T_{ii}I)(IT_{ji})T'_{iji} I_\l = (\phi^{ij}_\l, \alpha_i)_{\l+\alpha_i+\alpha_j} T'_{iji} I_\l$$ 
and hence $n_{iji}^\l (\phi^{ij}_\l, \alpha_i)_{\l+\alpha_i+\alpha_j} = 1$. The second relation follows by considering $T_{jij}$ instead of $T'_{iji}$. 
\end{proof}

\begin{Lemma}\label{LEM:B}
Using the notation from Corollary \ref{COR:serre2} and Lemma \ref{LEM:Tij} we have
\begin{enumerate}
\item $m_{iji}^{\l-\alpha_i} \cdot (\phi^{ij}_\l, \alpha_i)_{\l+\alpha_i+\alpha_j} = -1$ assuming both terms are nonzero and 
\item $n_{iji}^\l \cdot (\t_\l^{ji}, \alpha_i)_\l = -1$ assuming both terms are nonzero.
\end{enumerate}
\end{Lemma}
\begin{proof}
First note that
\begin{align*}
(T_{ji}I)(IT_{ii})(T_{ij}I)(IIX_i) 
&= (T_{ji}I)[(IX_iI)(IT_{ii}) - (III)](T_{ij}I) \\
&= (X_iII)(T_{ji}I)(IT_{ii})(T_{ij}I)\1_\l - (T_{ji}I)(T_{ij}I) 
\end{align*}
and similarly
$$(IT_{ij})(T_{ii}I)(IT_{ji})(IIX_i) = (X_iII)(IT_{ij})(T_{ii}I)(IT_{ji}) - (IT_{ij})(IT_{ji}).$$
Using that $m_{iji}^\l (T_{ji}I)(IT_{ii})(T_{ij}I) I_\l + n_{iji}^\l (IT_{ij})(T_{ii}I)(IT_{ji}) I_\l = (III) I_\l$ and simplifying we arrive at the relation
\begin{equation*}
m_{iji}^\l (T_{ji}I)(T_{ij}I) I_\l + n_{iji}^\l(IT_{ij})(IT_{ji}) I_\l = (X_iII) I_\l - (IIX_i) I_\l \in \End^2(\E_i \E_j \E_i \1_\l).
\end{equation*}
Now, the left side is equal to  
\begin{equation}\label{EQ:11}
m_{iji}^\l [(\phi^{ij}_{\l+\alpha_i}IIIIII) + (IIII\t^{ij}_{\l+\alpha_i}II)] + n_{iji}^\l[(II \phi^{ji}_\l IIII)+(IIIIII \t^{ji}_\l)]
\end{equation}
as an element in $\End^2(\1_{\mu+\alpha_i} \E_i \1_{\mu} \E_j \1_{\l+\alpha_i} \E_i \1_\l)$, where $\mu = \l+\alpha_i+\alpha_j$. On the other hand, the right side is equal to 
\begin{equation}\label{EQ:12}
-(\t IIIIII) + (II \t IIII) + (IIII \t II) - (IIIIII \t)
\end{equation}
for some $\t \in Y_\k$ with $\la \t, \alpha_i \ra = 1$. Comparing (\ref{EQ:11}) and (\ref{EQ:12}) we find that 
$$m_{iji}^\l (\phi_{\l+\alpha_i}^{ij}, \alpha_i)_{\mu+\alpha_i} = - (\t, \alpha_i)_{\mu+\alpha_i} = -1 \ \ \text{ and } \ \ n_{iji}^\l(\t_\l^{ji}, \alpha_i)_\l = - (\t, \alpha_i)_{\l} = -1.$$
This concludes the proof. 
\end{proof}

\begin{Lemma}\label{LEM:C}
Using the notation from Lemma \ref{LEM:Tij} we have 
\begin{enumerate}
\item $(\t^{ij}_\l, \alpha_i)_\l = (\t^{ji}_\l, \alpha_i)_\l$ if $\E_i^{(2)} \E_j \1_{\l-\alpha_i} \ne 0$ and $\E_j \E_i^{(2)} \1_{\l-\alpha_i} \ne 0$,
\item $(\phi^{ij}_\l, \alpha_j)_{\l+\alpha_i+\alpha_j} = (\phi^{ji}_\l, \alpha_j)_{\l+\alpha_i+\alpha_j}$ if $\E_j^{(2)} \E_i \1_\l \ne 0$ and $\E_i \E_j^{(2)} \1_\l \ne 0$. 
\end{enumerate}
\end{Lemma}
\begin{proof}
Consider the following composition
\begin{equation}\label{EQ:9}
\E_j \E_i \1_\l \E_i \xrightarrow{(IT_{ii})} \E_j \E_i \1_\l \E_i \xrightarrow{(T_{ji}I)(T_{ij}I)(T_{ji}I)} \E_i \E_j \1_\l \E_i \xrightarrow{(T_{ii}I)(IT_{ji})} \E_i \E_i \E_j
\end{equation}
where we omit the grading shifts for convenience. On the one hand, this is equal to the composition 
$$\E_j \E_i \1_\l \E_i \xrightarrow{(IT_{ii})} \1_\mu \E_j \E_i \1_\l \E_i \xrightarrow{(\phi_\l^{ij} IIII) + (III \t_\l^{ij} I)} \1_\mu \E_j \E_i \1_\l \E_i \xrightarrow{T_{jii}} \E_i \E_i \E_j$$
where $\mu = \l+\alpha_i+\alpha_j$. Now $T_{jii} = T'_{jii} = (IT_{ji})(T_{ji}I)(IT_{ii})$ which means that 
$$T_{jii}(III \t_\l^{ij} I)(IT_{ii}) = (IT_{ji})(T_{ji}I) [(IT_{ii})(III \t_\l^{ij} I)(IT_{ii})] = (\t_\l^{ij}, \alpha_i)_\l T_{jii}$$
while $T_{jii}(\phi_\l^{ij} IIII)(IT_{ii}) = 0$. In particular we get that the composition in (\ref{EQ:9}) is equal to $(\t_\l^{ij}, \alpha_i)_\l T_{jii}$. 

On the other hand, we can also write the composition in (\ref{EQ:9}) as 
$$\E_j \E_i \1_\l \E_i \xrightarrow{(IT_{ii})} \1_\mu \E_j \E_i \1_\l \E_i \xrightarrow{(\phi_\l^{ji} IIII) + (III \t_\l^{ji} I)} \1_\mu \E_j \E_i \1_\l \E_i \xrightarrow{T'_{jii}} \E_i \E_i \E_j.$$
The same argument as above simplifies this composition to give $(\t_\l^{ji}, \alpha_i)_\l T_{jii}$. Thus must have $(\t_\l^{ij}, \alpha_i)_\l = (\t_\l^{ji}, \alpha_i)_\l$. The second relation follows similarly. 
\end{proof}

\begin{Proposition}\label{prop:rescaleTij}
One can rescale the maps $T_{ij}$ so that $m_{iji}^\l = t_{ij}^{-1}$ and $n_{iji}^\l = - t_{ij}^{-1}$ for all $\l$. 
\end{Proposition}
\begin{proof}
First notice that rescaling some $T_{ij} \1_\l$ does not affect the relations in (\ref{EQ:2}), (\ref{EQ:3}), (\ref{EQ:KLR2}) and (\ref{EQ:KLR3}). This is because each $T_{ij}$ occurs the same number of times on both sides of these relations. 

Choose $\l$ so that $\1_\l \ne 0$ whereas $\1_{\l-\alpha_i} = 0$. We will now rescale all $T_{ij} I_\mu$ where $\mu = \l + r \alpha_i$ for some $r \in \Z$. We assume that $\1_{\l+\alpha_j} = 0$ and proceed by increasing induction on $\l_i$. If this is not the case then instead of $\l$ we choose $s_i \cdot \l := \l - \l_i \alpha_i$ where $\1_{s_i \cdot \l} \ne 0$, $\1_{s_i \cdot \l + \alpha_i} = 0$ and $\1_{s_i \cdot \l +\alpha_i+\alpha_j} \ne 0$ and proceed by decreasing induction on $\l_i$. 

{\bf The base case.} Since $\1_{\l+\alpha_j}=0$ we have  $\E_i^{(2)} \E_j \1_\l = 0$ and hence $(IT_{ij})(T_{ii}I)(IT_{ji}) I_\l = 0$. Thus, by Corollary \ref{COR:serre2}, $m_{iji}^\l (T_{ji}I)(IT_{ii})(T_{ij}I) I_\l = (III) I_\l$. So we can rescale $T_{ij} I_{\l+\alpha_i}$ so that $m_{iji}^\l = t_{ij}^{-1}$. This proves the base case.

{\bf The induction step.} Consider now 
$$m_{iji}^\mu (T_{ji}I)(IT_{ii})(T_{ij}I) I_\mu + n_{iji}^\mu (IT_{ij})(T_{ii}I)(IT_{ji}) I_\mu = (III) I_\mu \in \End(\E_i \E_j \E_i \1_\mu)$$
where $\mu = \l + r \alpha_i$ for some $r > 0$. By induction we have $\E_j \E_i^{(2)} \1_{\mu-\alpha_i} \ne 0$ and we have rescaled $T_{ij} I_\mu$ so that $m_{iji}^{\mu-\alpha_i} = t_{ij}^{-1}$. We claim that this implies $n_{iji}^\mu = - t_{ij}^{-1}$. 

To see this note that by Lemma \ref{LEM:A} we have $(\t_\mu^{ji}, \alpha_i)_\mu = t_{ij}$. Now, if $\E_i^{(2)} \E_j \1_\mu = 0$ there is nothing to prove. If it is nonzero then by Lemma \ref{LEM:B} we have $n_{iji}^\mu \cdot (\t_\mu^{ji}, \alpha_i)_\mu = -1$ which proves the claim. 

So now we can just rescale $T_{ij} I_{\mu+\alpha_i}$ so that $m_{iji}^\mu=t_{ij}^{-1}$. This completes the induction. 
\end{proof}

At this point we know the following relations: 
\begin{align}
\label{EQ:R1} & (\t_\l^{ij}, \alpha_k)_\l = - (\phi_\l^{ij}, \alpha_k)_{\l+\alpha_i+\alpha_j} \text{ if } k \in S_{ij}, \\
\label{EQ:R2} & (\t_\l^{ij}, \alpha_i)_\l = t_{ij} = - (\phi_\l^{ij}, \alpha_i)_{\l+\alpha_i+\alpha_j} \text{ and } \\
\label{EQ:R3} & (\t_\l^{ij}, \alpha_j)_\l = - (\phi_\l^{ij}, \alpha_j)_{\l+\alpha_i+\alpha_j}.
\end{align} 
To see the last relation above notice that switching the roles of $i$ and $j$ in Lemma \ref{LEM:A} equation (i) gives $n_{iji}^\l \cdot (\phi_\l^{ji}, \alpha_j)_{\l+\alpha_i+\alpha_j} = 1$ which together with Lemma \ref{LEM:C} equation (ii) gives $n_{iji}^\l \cdot (\phi_\l^{ij}, \alpha_j)_{\l+\alpha_i+\alpha_j} = 1$. On the other hand, switching the roles of $i$ and $j$ in Lemma \ref{LEM:B} equation (ii) gives $n_{jij}^\l \cdot (\t_\l^{ij}, \alpha_j)_\l = -1$. Putting these two relations together gives $(\t_\l^{ij}, \alpha_j)_\l = - (\phi_\l^{ij}, \alpha_j)_{\l+\alpha_i+\alpha_j}$.

\begin{Corollary}\label{COR:tji}
There exists $t_{ji}^\l \in \k^\times$ such that 
\begin{align}
\label{EQ:13A} (T_{ji})(T_{ij}) I_\l &= t_{ij}(X_iI) I_\l + t_{ji}^\l(IX_j) I_\l \in \End^2(\E_i \E_j \1_\l) \\
\label{EQ:13B} (T_{ij})(T_{ji}) I_\l &= t_{ij}(IX_i) I_\l + t_{ji}^\l(X_jI) I_\l \in \End^2(\E_j \E_i \1_\l).
\end{align}
\end{Corollary}
\begin{proof}
We have
$$(T_{ji})(T_{ij}) I_\l = (\phi_\l^{ij} IIII) + (IIII \t_\l^{ij}) \in \End^2(\1_{\mu} \E_i \1_{\l+\alpha_j} \E_j \1_\l)$$
where $\mu = \l+\alpha_i+\alpha_j$. Since $(\phi_\l^{ij}, \alpha_i)_{\mu} = - t_{ij}$ we can write 
$$(\phi_\l^{ij} IIII) = - t_{ij} (\t IIII) + t_{ij} (II \t II) + (II \sigma_\l^{ij} II) = t_{ij}(I X_i III) + (II \phi_\l^{ij} II)$$
for some $\t$ satisfying $(\t,\alpha_i) = 1$ and some $\sigma \in \End^2(\1_{\l+\alpha_j})$ satisfying $(\sigma, \alpha_k)_{\l+\alpha_j} = (\t_\l^{ij}, \alpha_k)_\mu$ for all $k \in S_{ij}$. Similarly we get 
$$(IIII \t_\l^{ij}) = - t_{ji}^\l (IIII \rho) + t_{ji}^\l (II \rho II) + (II \tau II) = t_{ji}^\l (III X_j I) + (II \tau II)$$
for some $\rho$ satisfying $(\rho,\alpha_j) = 1$ and where $t_{ji}^\l := - (\t_\l^{ij}, \alpha_j)_\l$ and $\tau \in \End^2(\1_{\l+\alpha_j})$ satisfies $(\tau,\alpha_k)_{\l+\alpha_j} = (\t_\l^{ij},\alpha_k)_\l$ for all $k \in S_{ij}$. Using relation (\ref{EQ:R1}), (\ref{EQ:R2}) and (\ref{EQ:R3}) it is straightforward to check that $(\sigma+\tau,\alpha_k)_{\l+\alpha_j} = 0$ for all $k \in S_{ij}$. This completes the proof of (\ref{EQ:13A}). 

Relation (\ref{EQ:13B}) now also follows since by switching the roles of $i$ and $j$ in Lemma \ref{LEM:C} we have
$$(\t_\l^{ij}, \alpha_k)_\l = (\t_\l^{ji},\alpha_k)_\l \text{ and } (\phi^{ij}_\l,\alpha_k)_{\l+\alpha_i+\alpha_j} = (\phi^{ji}_\l,\alpha_k)_{\l+\alpha_i+\alpha_j} \ \ \text{ for } k=i,j.$$ 
\end{proof}

\begin{Proposition}\label{PROP:tji}
The value of $t_{ji}^\l$ is independent of $\l$.
\end{Proposition}
\begin{proof}
We break up the argument into three claims which together give the result. 

{\bf Claim 1:} $t_{ji}^\l = t_{ji}^{\l + \alpha_i}$. Consider the composition 
$$\E_i \E_i \E_j \1_\l \xrightarrow{(IT_{ji})(IT_{ij})} \E_i \E_i \E_j \1_\l \xrightarrow{T_{ii}I} \E_i \E_i \E_j \1_\l \xrightarrow{IT_{ij}} \E_i \E_j \E_i \1_\l$$
where we omit the grading shifts for convenience. On the one hand, this composition equals
\begin{equation}\label{EQ:14A}
(IT_{ij})(T_{ii}I)[t_{ij} (IX_iI) + t_{ji}^\l (IIX_j)] = [t_{ij}(X_iII) + t_{ji}^\l (IX_jI)](IT_{ij})(T_{ii}I) - t_{ij} (IT_{ij})
\end{equation}
while on the other hand it is equal to
\begin{align}
\nonumber (T'_{iji})(IT_{ij}) &= (T_{iji})(IT_{ij}) - t_{ij} (IT_{ij}) \\
\nonumber &= (T_{ji}I)(IT_{ii})(T_{ij}I)(IT_{ij}) - t_{ij} (IT_{ij}) \\
\nonumber &= (T_{ji}I)(T_{ij}I)(IT_{ij})(T_{ii}I) - t_{ij} (IT_{ij}) \\
\label{EQ:14B} &= t_{ij} (X_iII)(IT_{ij})(T_{ii}I) + t_{ji}^{\l+\alpha_i} (IX_jI)(IT_{ij})(T_{ii}I) - t_{ij} (IT_{ij}). 
\end{align}
Comparing (\ref{EQ:14A}) and (\ref{EQ:14B}) we get that $t_{ji}^\l = t_{ji}^{\l+\alpha_i}$. 

{\bf Claim 2:} $t_{ji}^\l = t_{ji}^{\l + \alpha_j}$. Consider the composition 
$$\E_j \E_i \E_j \1_\l \xrightarrow{(IT_{ji})(IT_{ij})} \E_j \E_i \E_j \1_\l \xrightarrow{(T_{ji}I)} \E_i \E_j \E_j \1_\l \xrightarrow{IT_{jj}} \E_i \E_j \E_j \1_\l$$
where we again have ignored the grading shifts. On the one hand this composition equals
\begin{equation}\label{EQ:19}
(IT_{jj})(T_{ji}I)[t_{ij}(IX_iI) + t_{ji}^{\l}(IIX_j)]=[(t_{ij}(X_iII)+t_{ji}^\l(IX_jI)](IT_{jj})(T_{ji}I) - t_{ji}^\l (T_{ji}I).
\end{equation}
On the other hand it is equal to 
\begin{align}
\nonumber (T_{ji}I)T'_{jij} &= - m_{jij}^\l (n_{jij}^\l)^{-1} (T_{ji}I)T_{jij} + (n_{jij}^\l)^{-1} (T_{ji}I) \\
\label{EQ:20} & = - m_{jij}^\l (n_{jij}^\l)^{-1}[t_{ij}(X_iII) + t_{ji}^{\l+\alpha_j} (IX_jI)](IT_{jj})(T_{ji}I) + (n_{jij}^\l)^{-1} (T_{ji}I) 
\end{align}
Now, by Lemma \ref{LEM:B} relation (i) we have that
$$m_{jij}^\l \cdot (\phi_\l^{ji}, \alpha_j)_{\l+\alpha_i+\alpha_j} = -1 \Rightarrow m_{jij}^\l \cdot t_{ji}^{\l+\alpha_j} = -1$$
and likewise, from Lemma \ref{LEM:B} relation (ii) we have $n_{jij}^\l \cdot t_{ji}^\l = -1$. Thus the last two terms in (\ref{EQ:19}) and (\ref{EQ:20}) are equal, leaving us with 
$$t_{ij}(X_iII)(IT_{jj})(T_{ji}I) = -m_{jij}^\l (n_{jij}^\l)^{-1} t_{ij}(X_iII)(IT_{jj})(T_{ji}I).$$
Thus we get $m_{jij}^\l = - n_{jij}^\l \Rightarrow t_{ji}^{\l+\alpha_j} = t_{ji}^\l$. 

{\bf Claim 3:} $t_{ji}^\l = t_{ji}^{\l + \alpha_k}$ when $k \ne i,j$. Consider the composition 
\begin{equation}\label{EQ:16A}
\E_k \E_i \E_j \1_\l \xrightarrow{(IT_{ji})(IT_{ij})} \E_k \E_i \E_j \1_\l \xrightarrow{(T_{ki}I)} \E_i \E_k \E_j \1_\l \xrightarrow{(IT_{kj})} \E_i \E_j \E_k \1_\l.
\end{equation}
On the one hand this composition equals 
\begin{equation}\label{EQ:17}
(IT_{kj})(T_{ki}I)[t_{ij} (IX_iI) + t_{ji}^\l (IIX_j)] = [t_{ij}(X_iII)+t_{ji}^\l(IX_jI)](IT_{kj})(T_{ki}I).
\end{equation}
On the other hand, we know that 
$$(IT_{kj})(T_{ki}I)(IT_{ji}) = c_1 (T_{ji}I)(IT_{ki})(T_{kj}I) \ \ \text{ and } \ \ (IT_{ki})(T_{kj}I)(IT_{ij}) = c_2 (T_{ij}I)(IT_{kj})(T_{ki}I)$$
for some constants $c_1,c_2 \in \k$. So the composition (\ref{EQ:16A}) equals
\begin{equation}\label{EQ:18}
c_1c_2 (T_{ji}I)(T_{ij}I)(IT_{kj})(T_{ki}I) = c_1c_2[t_{ij}(X_iII) + t_{ji}^{\l+\alpha_k}(IX_jI)](IT_{kj})(T_{ki}I).
\end{equation}
Comparing (\ref{EQ:17}) and (\ref{EQ:18}) we find that $c_1c_2=1$ and subsequently $t_{ji}^{\l+\alpha_k} = t_{ji}^\l$. 
\end{proof}

Denoting the common value of $t_{ji}^\l$ by $t_{ji}$ we can summarize the content of the results above as follows
\begin{framed}
\begin{align}
\label{EQ:B} (T_{ji}I)(IT_{ii})(T_{ij}I) &= (IT_{ij})(T_{ii}I)(IT_{ji}) + t_{ij} (III) \in \End(\E_i \E_j \E_i) \\
\label{EQ:A} (T_{ij}I)(IT_{jj})(T_{ji}I) &= (IT_{ji})(T_{jj}I)(IT_{ij}) + t_{ji} (III) \in \End(\E_j \E_i \E_j) \\
\label{EQ:C} (T_{ji})(T_{ij}) &= t_{ij} (X_i I) + t_{ji}(IX_j) \in \End^2(\E_i \E_j) \\
\label{EQ:D} (T_{ij})(T_{ji}) &= t_{ij} (IX_i) + t_{ji}(X_j I) \in \End^2(\E_j \E_i).
\end{align}
\end{framed}
Finally, if $\la j,k \ra = 0$ then $T_{kj} T_{jk} \1_\l \in \End(\E_j \E_k \1_\l)$ is some nonzero multiple of the identity. So one can rescale each $T_{jk} \1_\l$ so that 
\begin{framed}
\begin{equation}\label{EQ:E}
(T_{kj})(T_{jk}) = t_{jk} (II) \in \End(\E_j \E_k \1_\l).
\end{equation}
\end{framed}
It is easy to see that this implies $(T_{jk})(T_{kj}) = t_{jk} (II) \in \End(\E_k \E_j)$ meaning that $t_{jk} = t_{kj}$. 

\begin{Remark}
Equalities (\ref{EQ:B}), (\ref{EQ:A}) and (\ref{EQ:E}) hold even without modding out by transient maps. 
\end{Remark}

\section{Step \#8 -- The $T_{ijk}$ relation}

In this section we will show how to rescale maps so that 
$$T_{ijk} I_\l = T'_{ijk} I_\l \in \Hom(\E_i \E_j \E_k \1_\l, \E_k \E_j \E_i \1_\l \la -\ell_{ijk} \ra)$$
whenever $i,j,k \in I$ are all distinct. Assuming the Hom space above is nonzero we know by Lemma \ref{LEM:homijk} that it is one-dimensional. By Proposition \ref{PROP:homijk}, we also have that $T_{ijk} I_\l$ and $T'_{ijk} I_\l$ are nonzero which means that they must be equal to each other up to a multiple. 

It remains to find a consistent way to rescale $T$'s so that $T_{ijk} = T'_{ijk}$. To do this we fix a total ordering on $\prec$ on $I$. By decreasing induction (with respect to $\prec$) we will rescale $T_{ij}$'s so that $T_{ijk} = T'_{ijk} $ if $i \prec j \prec k$. More precisely, we can assume that we have already rescaled all $T_{ik}$ and $T_{jk}$ where $i,j \prec k$ as well as all $T_{k_1 k_2}$ with $i \prec k_1 \prec k_2$ so that $T_{ik_1k_2} = T'_{ik_1k_2}$ and $T_{jk_1k_2} = T'_{jk_1k_2}$. Then, given $T_{ij} I_\l$ one can uniquely rescale $T_{ij} I_{\l+\alpha_k}$ so that $T_{ijk} I_\l = T'_{ijk} I_\l$. 

\begin{Proposition}\label{PROP:welldefined2}
This rescaling algorithm for $T_{ij} I_\l$ is well defined. 
\end{Proposition}
\begin{proof}
The proof is analogous to that of Proposition \ref{PROP:rescale} with $T_{ii}$ replaced by $T_{ij}$. More precisely, we consider a sequence of rescalings 
\begin{equation}\label{eq:7}
T_{ij} I_\l \leadsto T_{ij} I_{\l+c_1 \alpha_{k_1}} \leadsto T_{ij} I_{\l + c_1 \alpha_{k_1} + c_2 \alpha_{k_2}} \leadsto \dots \leadsto T_{ij} I_{\l+ \sum_\ell c_\ell \alpha_{k_\ell}} = T_{ij} I_\l
\end{equation}
where $c_\ell = \pm 1$, $i,j \prec k_\ell$ for all $\ell$ and $\sum_\ell c_\ell \alpha_{k_\ell} = 0$ which we encode as $\uck = (c_1 k_1, \dots, c_m k_m)$. To show that such a sequence does not rescale $T_{ij} I_\l$ we consider the two operators $S_a$ and $D_a$ acting on such sequences as in Proposition \ref{PROP:rescale}. Then the way $\uck$ and $S_a \cdot \uck$ ends up rescaling $T_{ij} I_\l$  is the same due to Corollary \ref{COR:welldefined4} while the way $\uck$ and $D_a \cdot \uck$ rescale $T_{ij} I_\l$ is also the same (essentially by definition). The result now follows since any sequence $\underline{c \alpha}$ can be transformed into the trivial sequence by repeatedly applying these moves. 
\end{proof}

Having rescaled $T_{ij} I_\l$ as above we rescale $T_{ji} I_\l$ by the inverse. This has the effect of preserving all the other relations such as those in (\ref{EQ:C}) and (\ref{EQ:D}). So we maintain all our prior results while adding the following relation: 
\begin{equation}\label{eq:RIII3}
T_{ijk} I_\l = T'_{ijk} I_\l: \E_i \E_j \E_k \1_\l \rightarrow \E_k \E_j \E_i \1_\l \la -\ell_{ijk} \ra \ \ \text{ for all } i,j,k \in I \text{ with } i \prec j \prec k.
\end{equation}

\begin{Lemma}\label{LEM:welldefined3}
Suppose $i,j,a,b \in I$ with $i \prec j \prec a,b$ and that $\1_{\l+\eps_i \alpha_i+\eps_j \alpha_j + \eps_a \alpha_a + \eps_b \alpha_b} \ne 0$ for $\eps_i,\eps_j,\eps_a,\eps_b \in \{0,1\}$. If $T_{iab} = T'_{iab}$ and $T_{jab} = T'_{jab}$ then the two sequences
\begin{equation}\label{EQ:rescale3}
T_{ij} I_\l \leadsto T_{ij} I_{\l+\alpha_a} \leadsto T_{ij} I_{\l + \alpha_a+\alpha_b} \ \ \text{ and } \ \ T_{ij} I_\l \leadsto T_{ij} I_{\l+\alpha_b} \leadsto T_{ij} I_{\l + \alpha_a+\alpha_b}
\end{equation}
lead to the same rescaling of $T_{ij} I_{\l + \alpha_a+\alpha_b}$. 
\end{Lemma}
\begin{proof}
Consider the composition 
\begin{equation}\label{eq:quad}
(IIT_{ij})(IT_{ia}I)(IIT_{ja})(T_{ib}II)(IT_{jb}I)(IIT_{ab}) : \E_i \E_j \E_a \E_b \1_\l \rightarrow \E_b \E_a \E_j \E_i \1_\l.
\end{equation}
On the one hand one can write this as 
\begin{align}
\nonumber (IT'_{ija})(T_{ib}II)(IT_{jb}I)(IIT_{ab}) =& (IT_{ija})(T_{ib}II)(IT_{jb}I)(IIT_{ab}) \\
\nonumber =& (IT_{ja}I)(IIT_{ia})[(IT_{ij}I)(T_{ib}II)(IT_{jb}I)](IIT_{ab}) \\
\nonumber =& (IT_{ja}I)(IIT_{ia})(T'_{ijb}I)(IIT_{ab}) \\
\nonumber =& (IT_{ja}I)(IIT_{ia})(T_{ijb}I)(IIT_{ab}) \\
\label{eq:quad4} =& (IT_{ja}I)(IIT_{ia})(T_{jb}II)(IT_{ib}I)(T_{ij}II)(IIT_{ab})
\end{align}
where we rescaled $T_{ij} I_\l \leadsto T_{ij} I_{\l+\alpha_a} \leadsto T_{ij} I_{\l + \alpha_a+\alpha_b}$ in order to get the first and fourth equalities. 

On the other hand, using that $T_{iab}=T'_{iab}$ and $T_{jab}=T'_{jab}$ one can rewrite (\ref{eq:quad}) as 
$$(T_{ab}II)(IIT_{ij})(IT_{ib}I)(IIT_{jb})(T_{ia}II)(IT_{ja}I)$$
where we use that we know $T_{iab}=T'_{iab}$ and $T_{jab}=T'_{jab}$ to slide the $T_{ab}$ from the far right to the far left. Then the same sequence of equalities as above shows that this is equal to 
\begin{equation}\label{eq:quad3}
(T_{ab}II)(IT_{jb}I)(IIT_{ib})(T_{ja}II)(IT_{ia}I)(T_{ij}II)
\end{equation}
where this time we had to rescale $T_{ij} I_\l \leadsto T_{ij} I_{\l+\alpha_b} \leadsto T_{ij} I_{\l + \alpha_a+\alpha_b}$. Finally, sliding back the $(T_{ab}II)$ we find that (\ref{eq:quad3}) is equal to (\ref{eq:quad4}). Since by Lemma \ref{LEM:ijab} the composition in (\ref{eq:quad}) is nonzero it follows that the two sequences of rescalings must be the same. 
\end{proof}

\begin{Lemma}\label{LEM:ijab}
Suppose $i,j,a,b \in I$ are distinct and let $\ell_{ijab} := \la i,j+a+b \ra + \la j,a+b \ra + \la a,b \ra$. Then 
\begin{equation}\label{EQ:ijab}
(IIT_{ij})(IT_{ia}I)(IIT_{ja})(T_{ib}II)(IT_{jb}I)(IIT_{ab}) : \E_i \E_j \E_a \E_b \1_\l \rightarrow \E_b \E_a \E_j \E_i \1_\l \la -\ell_{ijab} \ra
\end{equation}
is nonzero if and only if $\1_{\l+\eps_i \alpha_i+\eps_j \alpha_j + \eps_a \alpha_a + \eps_b \alpha_b} \ne 0$ for $\eps_i,\eps_j,\eps_a,\eps_b \in \{0,1\}$. 
\end{Lemma}
\begin{proof}
This result is the analogue of Lemma \ref{LEM:iiab} which dealth with the case $i=j$. The proof is also very similar. One direction is immediate since one check that the weights $\l+\eps_i \alpha_i+\eps_j \alpha_j + \eps_a \alpha_a + \eps_b \alpha_b$ all appear somewhere in the composition (\ref{EQ:ijab}) (after possibly having to use relations such as $T_{ija} \sim T'_{ija}$). 

For the converse one can argue as in Lemma \ref{LEM:iiab} that (\ref{EQ:ijab}) is nonzero if either
\begin{equation}\label{EQ:35}
T_{ab}IT_{ij}: \E_a \E_b \1_\l \F_i \F_j \rightarrow \E_b \E_a \1_\l \F_j \F_i \ \ \text{ or } \ \ T_{ij}I T_{ab}: \F_i \F_j \1_\mu \E_a \E_b \rightarrow \F_j \F_i \1_\mu \E_b \E_a
\end{equation}
are nonzero, where $\mu := \l+\alpha_i+\alpha_j+\alpha_a+\alpha_b$ and we abuse notation by writing $T_{ij} \in \Hom^{-\la i,j \ra}(\F_i \F_j, \F_j \F_i)$ for the analogue of $T_{ij} \in \Hom^{-\la i,j \ra}(\E_i \E_j, \E_j \E_i)$. These two maps are the analogues of (\ref{EQ:28}) and (\ref{EQ:29}). 

Now, compose the left map in (\ref{EQ:35}) with $\F_a$ on the left. If $\la a,b \ra = -1$ and $\l_a \le -1$ then by Lemma \ref{LEM:rank} the $\E_b$-rank of $IT_{ab}I: \F_a \E_a \E_b \1_\l \rightarrow \F_a \E_b \E_a \1_\l \la 1 \ra$ is $-\l_a > 0$ so we reduce to showing $IIT_{ij}: \E_b \1_\l \F_i \F_j \rightarrow \E_b \1_\l \F_j \F_i$. Composing with invertible maps $v_{bj}$ and $v_{bi}$ leaves us with the composition 
$$\E_b \F_i \F_j \xrightarrow{IT_{ij}} \E_b \F_j \F_i \xrightarrow{v_{bj}I} \F_j \E_b \F_i \xrightarrow{Iv_{bi}} \F_j \F_i \E_b.$$
This is adjoint to $T_{ijb}$ and hence nonzero by Proposition \ref{PROP:homijk}. Thus $\l_a \le -1 \Rightarrow (\ref{EQ:ijab})$ is nonzero.

A similar argument except composing the right map in (\ref{EQ:35}) with $\F_a$ on the right show that
$$\l_a \ge - 1 - \la a,i+j+b \ra \Rightarrow (\ref{EQ:ijab}) \text{ is nonzero }.$$
Hence we are finished unless 
$$\la a, i+j+b \ra = -2 \text{ and } \l_a = 0 \ \ \text{ or } \ \ \la a, i+j+b \ra = -3 \text{ and } \l_a = 0 \text{ or } 1.$$
Similar arguments with $b,i$ and $j$ give us that we are also finished unless
\begin{align*}
& \la b,i+j+a \ra = -2 \text{ and } \l_b = 0 \ \ \text{ or } \ \ \la b,i+j+a \ra = -3 \text{ and } \l_b = 0 \text{ or } 1 \\
& \la i,j+a+b \ra = -2 \text{ and } \l_i = 0 \ \ \text{ or } \ \ \la i,j+a+b \ra = -3 \text{ and } \l_i = 0 \text{ or } 1 \\
& \la j,i+a+b \ra = -2 \text{ and } \l_j = 0 \ \ \text{ or } \ \ \la j,i+a+b \ra = -3 \text{ and } \l_j = 0 \text{ or } 1. 
\end{align*}
We now argue case by case that these four conditions cannot all hold. There are three cases depending on whether the Dynkin diagram containing $i,j,a,b$ is a square, contains 5 edges or contains 6 edges ({\it i.e.} is a complete graph). 
 
In the first case it follows by condition (\ref{co:vanish2}) that $\la \l, i+j+a+b \ra > 0$ since $\1_\l \ne 0$. On the other hand, $\l_i=\l_j=\l_a=\l_b=0$ which contradicts this. 

In the second case suppose (without loss of generality) that $i,j,a$ generates one of the two triangles. Then as above we must have $\la \l+b,i+j+a \ra > 0 \Rightarrow \l_i+\l_j+\l_a \ge 3$ since vertex $b$ is connected to exactly two of $i,j,a$. But this is impossible since at least one of $\l_i,\l_j,\l_a$ must be zero (and the other two either zero or one). 

In the third case $i,j,b$ form a triangle which means $\la \l+a,i+j+b \ra \ge 1 \Rightarrow \l_i+\l_j+\l_b \ge 4$. This is not possible since $\l_i,\l_j,\l_b \le 1$. 
\end{proof}

\begin{Corollary}\label{COR:welldefined4}
Suppose $i,j,a,b \in I$ with $i \prec j \prec a,b$ and that $T_{ij} I_\l$ and $T_{ij} I_{\l+\alpha_a-\alpha_b}$ are nonzero. If $T_{iab} = T'_{iab}$ and $T_{jab} = T'_{jab}$ then the two sequences
$$T_{ij} I_\l \leadsto T_{ij} I_{\l+\alpha_a} \leadsto T_{ij} I_{\l + \alpha_a - \alpha_b} \ \ \text{ and } \ \ T_{ij} I_\l \leadsto T_{ij} I_{\l-\alpha_b} \leadsto T_{ij} I_{\l + \alpha_a - \alpha_b}$$
lead to the same rescaling of $T_{ij} I_{\l+\alpha_a-\alpha_b}$.
\end{Corollary}
\begin{proof}
Since $T_{ij} I_\l \ne 0$ and $T_{ij} I_{\l+\alpha_a-\alpha_b} \ne 0$ we get that $\1_{\l+\eps_i \alpha_i+\eps_j \alpha_j} \ne 0$ and $\1_{\l+\alpha_a-\alpha_b+\eps_i \alpha_i+\eps_j \alpha_j} \ne 0$, where $\eps_i,\eps_j \in \{0,1\}$. By condition (\ref{co:new}) this implies that all $\1_{\l+\eps_i \alpha_i+\eps_j \alpha_j + \eps_a \alpha_a + \eps_b \alpha_b}$ are nonzero and the result follows from Lemma \ref{LEM:welldefined3}. 
\end{proof}

Finally, given (\ref{eq:RIII3}) it remains to check that $T_{ijk} I_\l = T'_{ijk} I_\l$ for any distinct $i,j,k \in I$ (not just when $i \prec j \prec k$). This follows from the following Lemma. 

\begin{Lemma}
For distinct $a,b,c \in I$, if $T_{abc} I_\l = T'_{abc} I_\l$ for all $\l$ then $T_{a'b'c'} I_\l = T'_{a'b'c'} I_\l$ for any permutation $a',b',c'$ of $a,b,c$. 
\end{Lemma}
\begin{proof}
We prove the case when $(a',b',c') = (b,a,c)$ (the general case follows similarly). There are various cases to consider depending on whether $a,b,c$ are connected in the Dynkin diagram. We will deal with the most difficult case when $\la a,b \ra = \la a,c \ra = \la b,c \ra = -1$ (the other cases are proved in the same way but involve simpler computations). 

By Proposition \ref{PROP:homijk} we know that $T_{bac} I_\l = s T'_{bac} I_\l$ for some $s \in \k^\times$. We need to show that in fact $s=1$. To do this consider the composition 
\begin{equation}\label{eq:9}
\E_a \E_b \E_c \1_\l \xrightarrow{T_{ab}I} \E_b \E_a \E_c \1_\l \xrightarrow{T_{ba}I} \E_a \E_b \E_c \1_\l \xrightarrow{IT_{bc}} \E_a \E_c \E_b \1_\l \xrightarrow{T_{ac}I} \E_c \E_a \E_b \1_\l.
\end{equation}
On the one hand, this is equal to 
\begin{align}
\nonumber (T_{bac})(T_{ab}I) 
=& s (T'_{bac})(T_{ab}I) = s (IT_{ba}I)(T_{bc}I)(IT_{ac})(T_{ab}I) = s (IT_{ba})(T_{abc}) \\
\nonumber =& s (IT_{ba})(T'_{abc}) = s (IT_{ba})(IT_{ab})(T_{ac}I)(IT_{bc}) \\
\label{EQ:10} =& s [t_{ab}(IX_aI) + t_{ba}(IIX_b)](T_{ac}I)(IT_{bc})
\end{align}
where we used relation (\ref{EQ:C}) to get the last equality.  

On the other hand, we can begin by first applying (\ref{EQ:C}) to (\ref{eq:9}) to obtain
\begin{equation}
\label{eq:11} (T_{ac}I)(IT_{bc})[t_{ab}(X_aII) + t_{ba}(IX_bI)] \1_\l = [t_{ab}(IX_aI) + t_{ba}(IIX_b)](T_{ac}I)(IT_{bc}) \1_\l. 
\end{equation} 
Comparing (\ref{EQ:10}) with (\ref{eq:11}) we get $s=1$ as long as 
$$[t_{ab}(IX_aI) + t_{ba}(IIX_b)](T_{ac}I)(IT_{bc}): \E_a \E_b \E_c \1_\l \rightarrow \E_c \E_a \E_b \1_\l \la 4 \ra$$ 
is nonzero. To show this note that this map is adjoint to 
\begin{equation}\label{eq:12}
t_{ab}(X_a III) + t_{ba}(IX_bII): \E_a \E_b \1_\l \F_c \rightarrow  \E_a \E_b \1_\l \F_c \la 2 \ra.
\end{equation}
Now if $\l_c \le -1$ then you can compose (\ref{eq:12}) on the right with $\E_c$. Simplifying one finds that one of the summands is 
\begin{equation}\label{eq:13}
t_{ab}(X_aII) + t_{ba}(IX_bI): \E_a \E_b \1_\l \rightarrow \E_a \E_b \1_\l. 
\end{equation}
Now, if $\l_b \ge 0$ then one can compose (\ref{eq:13}) on the right with $\F_b$ and consider the composition 
$$\E_a \1_{\l+\alpha_b} \xrightarrow{I \adj^b} \E_a \E_b \F_b \1_{\l+\alpha_b} \xrightarrow{(\tiny{\ref{eq:13})}(IX_b^{\l_b}I)} \E_a \E_b \F_b \1_{\l+\alpha_b} \xrightarrow{I \adj_b} \E_a \1_{\l+\alpha_b}$$
which, by Corollary \ref{COR:bubbles3}, is equal to (a nonzero multiple of) the identity map on $\E_a \1_{\l+\alpha_b}$. Finally, $\E_a \1_{\l+\alpha_b} \ne 0$ by Lemma \ref{LEM:nonvan} since $\1_{\l+\eps_a \alpha_a+\eps_b \alpha_b+\eps_c \alpha_c} \ne 0$ for $\eps_a,\eps_b,\eps_c \in \{0,1\}$. Thus (\ref{eq:13}) is nonzero in this case. The case $\l_b \le -2$ is similar by composing with $\F_b$ on the left instead. Thus we are left with the case $\l_b=-1$. 

A similar argument with $a$ instead of $b$ leaves us with the case $\l_a=0$. To take care of this possibility that $\l_a=0$ and $\l_b=-1$ we note that we could have switched the roles of $a,b$ in the original equation (\ref{eq:9}). Then the same argument above would take care of this possibility too. 

This completes the case when $\l_c \le -1$. One has a similar argument when $\l_c \ge 1$ by composing (\ref{eq:12}) with $\E_c$ on the left to reduce to showing that $t_{ab}(X_aII) + t_{ba}(IX_bI) \in \End(\E_a \E_b \1_{\l+\alpha_c})$ is nonzero. This is done in the same way as above. 

Finally, if $\l_c=0$ one cannot simplify as above. On the other hand, one can still argue as above to reduce to the case when $\l_a = 0,1$ and $\l_b=0,-1$. Since $\l_a+\l_b+\l_c \ge 1$ we must have that $\l_a=1$ and $\l_b=0$. This final possibility that $\l_a=1,\l_b=0$ and $\l_c=0$ can be taken care of again by reversing the roles of $a$ and $b$ in (\ref{eq:9}). This concludes the proof that the map in (\ref{eq:12}) is nonzero. 
\end{proof}

So we arrive at the following relation, which holds without having to mod out by transient maps.
\begin{framed}
\begin{equation}
T_{ijk} = T'_{ijk} \in \Hom(\E_i \E_j \E_k, \E_k \E_j \E_i \la -\ell_{ijk} \ra) \ \ \text{ for all distinct } i,j,k \in I. 
\end{equation}
\end{framed}

\section{Step \#9 -- Transients}\label{sec:transients}

At this point several of the relations above hold only modulo transient maps, namely relations (\ref{EQ:KLR1}), (\ref{EQ:KLR2}), (\ref{EQ:KLR3}), (\ref{EQ:C}) and (\ref{EQ:D}). In many cases this may be good enough since transient maps are usually negligible. For example, in applications to knot homology (section \ref{sec:knots}), the homology of any link involves a computation whose output is an endomorphism of the highest weight space where there are no transient maps. Even better, in our application to vertex operators (section \ref{sec:vertex}), there are no transient maps at all ({\it i.e.} they are all zero). 

Nevertheless, it would be nice to know whether all relations hold on the nose. In this section we prove this is the case when $\g = \sl_n$. We expect this also holds when the graph $\Gamma$ associated to $\g$ is a tree and perhaps more generally. 

To understand the general argument it is worth first considering the case $\g=\sl_2$. In this case one needs to prove the affine nilHecke relations or, equivalently, that 
\begin{equation}\label{EQ:ex}
\End(\E \1_\l \E) \ni T(I \t I) + (I \t I)T = (ITI)[(III \t)+(\t III)]+id \in \End(\1_{\l+\alpha} \E \E \1_{\l-\alpha}).
\end{equation}
We know this relation holds modulo transient maps. On the other hand, a transient map belonging to $\End^2(\1_{\l+\alpha} \E \E \1_{\l-\alpha})$ is of the form $(III \phi)$ if $\l \le 0$ and $(\phi III)$ if $\l \ge 0$. This means that, if $\l \ge 0$ (resp. $\l \le 0$) then one can redefine $\t \in \End^2(\1_{\l+\alpha})$ (resp. $\t \in \End^2(\1_{\l-\alpha})$) by adding some transient map so that (\ref{EQ:ex}) holds. Thus, starting in the middle ({\it i.e.} at weight space $\l=0$ or $\l=-1$) we can go out in both directions and redefine the $\t$'s so that (\ref{EQ:ex}) holds on the nose. 

When $\g = \sl_n$ (and $n > 2$) things are more complicated because there are further relations to check. Moreover, it is not so clear what is the analogue of the middle weight $\l=0$ or $\l=-1$. However, the general idea is the same: we will start at a ``middle weight'' (see section \ref{sec:middle}) and work our way outwards to redefine all $\t$'s by adding to them appropriate transient maps. 

More precisely, we will proceed as follows (from hereon $\g=\sl_n$). Fix an orientation of $\Gamma$ so that from each vertex there is at most one arrow leaving. Also, fix $\theta_i \in Y_\k$ so that $\la \t_i, \alpha_j \ra = \delta_{i,j}$. 
\begin{enumerate}
\item Fix a middle weight $\mu$ (see section \ref{sec:middle}). Show that having fixed $\t_i \in \End^2(\1_\mu)$ and $\t_i \in \End^2(\1_{\mu-\alpha_i})$ one can then redefine the remaining $\t_i$'s so that 
\begin{align}
\label{EQ:N1} II \t_i &= \t_i II \in \End^2(\1_{\l+\alpha_j} \E_j \1_\l) \text{ for any } i \ne j \text{ and } \\
\label{EQ:N2} \End(\E_i \1_\l \E_i) \ni T_{ii}(I \t_i I) + (I \t_i I)T_{ii} &= (IT_{ii}I)[(III \t_i)+(\t_i III)]+id \in \End(\1_{\l+\alpha_i} \E_i \E_i \1_{\l-\alpha_i}).
\end{align}
\item For each edge $i \rightarrow j$ in $\Gamma$, redefine $\theta_i \in \End^2(\1_\mu)$ and $\theta_i \in \End^2(\1_{\mu-\alpha_i})$ so that 
\begin{equation}\label{EQ:N3}
(T_{ji})(T_{ij}) = t_{ij}(X_i I) + t_{ji}(IX_j) \in \End^2(\E_i \1_\l \E_j)
\end{equation}
holds when $\l=\mu-\alpha_i$. Then use step 1 to redefine all other $\t_i$'s and show that this implies (\ref{EQ:N3}) for all $\l$. 
\item Show that 
\begin{equation}\label{EQ:N4}
(T_{ij})(T_{ji}) = t_{ij}(I X_i) + t_{ji}(X_j I) \in \End^2(\E_j \1_\l \E_i)
\end{equation}
follows as a consequence of (\ref{EQ:N1}), (\ref{EQ:N2}) and (\ref{EQ:N3}). 
\end{enumerate}

\begin{Remark}
If, following section \ref{sec:Xi}, we define $X_i \in \End^2(\1_{\l+\alpha_i} \E_i \1_\l)$ as $-(\t_i II) + (II \t_i)$ then relations (\ref{EQ:N1}) and (\ref{EQ:N2}) imply (\ref{EQ:KLR1}), (\ref{EQ:KLR2}) and (\ref{EQ:KLR3}) while (\ref{EQ:N3}) and (\ref{EQ:N4}) are the same as (\ref{EQ:C}) and (\ref{EQ:D}). 
\end{Remark}

\subsection{The middle weight}\label{sec:middle}

For each $i \in I$ we define the following {\em valid slides}
$$\l \leadsto \l+\alpha_i \text{ if } \l_i \ge -1 \ \ \text{ and } \ \ \l \leadsto \l-\alpha_i \text{ if } \l_i \le 1.$$
We say $\l \in X$ is a {\it middle weight} of $\K$ if you can reach any other nonzero weight space of $X$ from $\l$ by using a sequence of valid slides (such a sequence is called a {\em path}). 

\begin{Proposition}\label{PROP:middle}
If $\g = \sl_n$ then $\K$ has a middle weight. 
\end{Proposition}
\begin{proof}
For $\sl_n$ the cosets $X/Y$ are indexed by $\omega \in \{0, \Lambda_1, \dots, \Lambda_{n-1}\}$. We will show that these are all middle weights. To simplify notation suppose $\omega = 0$ (the other cases are exactly the same). Consider $\l = \sum_{j=1}^{n-1} a_j \alpha_j$ for some $a_j \in \Z$ with $\1_\l \ne 0$. 

{\bf Claim:} if $\l \ne 0$ then there exists an $i$ so that either $a_i, \l_i \ge 1$ or $a_i, \l_i \le -1$. 

If $a_1=0$ then this claim is reduced to $\sl_{n-1}$ and we proceed by induction. If $a_1 \ge 1$ then $\l_1 \le 0$ or otherwise we are done. Thus 
$$\l_1 = 2a_1-a_2 \le 0 \Rightarrow a_2 \ge 2a_1 \ge 2.$$
But now $a_2 \ge 2$ means $\l_2 \le 0$ or otherwise we are done. Then
$$\l_2 = 2a_2-a_1-a_3 \le 0 \Rightarrow a_3 \ge 2a_2 - a_1 \ge a_2+a_1 \ge 3$$
meaning $\l_3 \le 0$. Repeating this arguments gives the claim. The case $a_1 \le -1$ is proved similarly.  

The result now follows from the claim. More precisely, if $\l = 0$ we are done. Otherwise suppose $a_i,\l_i \ge 1$ (the case $a_i,\l_i \le -1$ is similar). Then we have a slide $\l - \alpha_i \leadsto \l$. Note that $\1_{\l-\alpha_i} \ne 0$ since $\E_i \1_{\l-\alpha_i} \F_i$ contains at least one copy of $\1_\l \ne 0$ (using that $\l_i \ge 1$). Finally, by induction on $\sum_j |a_j|$, there exists a path $0 \leadsto \dots \leadsto \l-\alpha_i$. Together with the slide $\l-\alpha_i \leadsto \l$ this gives a path $0 \leadsto \dots \leadsto \l$. 
\end{proof}

From now on we fix a middle weight $\mu$. The proof above gives us a canonical path from $\mu$ to any weight space $\l$. More precisely, we start at $\l$ and word backwards. Using the claim in the Lemma we choose the smallest $i$ so that $a_i, \l_i \ge $ or $a_i, \l_i \le -1$. Take the last slide in the path to be $\l-\alpha_i \leadsto \l$ or $\l+\alpha_i \leadsto \l$ respectively, and then repeat. We refer to this sequence of slides as the {\em canonical path} from $\mu$ to $\l$. 

\begin{Remark}\label{rem:orient}
For convenience, if $\mu = \Lambda_i$ for some $i \in I$ then we choose the orientation on $\Gamma$ so that all edges incident on $i$ are oriented away from $i$.
\end{Remark} 

\subsection{Slide equivalences}

Consider a path $\mu \leadsto \mu+c_1 \alpha_{k_1} \leadsto \dots \leadsto \mu+\sum_\ell c_\ell \alpha_{k_\ell} = \l$ where $c_\ell = \pm 1$. We will encode it as $(\uck) = (c_1 k_1, \dots, c_m k_m)$ and denote by its length by $|\uck|$ (the length equals $m$). We denote the canonical path from $\mu$ to $\l$ by $(\uck_\l)$. Note that the canonical path has the smallest possible length among all paths, namely $\sum_j |a_j|$ if $\l = \mu + \sum_j a_j \alpha_j$. 

We have two operations we can perform on a path $\uck$, namely 
\begin{align*}
S_a \cdot (c_1 k_1, \dots, c_a k_a, c_{a+1} k_{a+1}, \dots, c_m k_m) &= (c_1 k_1, \dots, c_{a+1} k_{a+1}, c_a k_a, \dots, c_m k_m) \\
D_a \cdot (c_1 k_1, \dots, c_a k_a, c_{a+1} k_{a+1}, \dots, c_m k_m) &= (c_1 k_1, \dots, c_{a-1} k_{a-1}, c_{a+2} k_{a+2}, \dots, c_m k_m) 
\end{align*}
where $S_a$ is defined as long as the right hand side is a valid path and $D_a$ is defined when $c_ak_a= -c_a k_{a+1}$. If two paths $(\uck)$ and $(\uck')$ are equivalent via a sequence of such operations we write $(\uck) \sim (\uck')$ and say they are {\em slide equivalent}. 

\begin{Lemma}\label{LEM:reduce}
Suppose $\l = \mu + \sum_j a_j \l_j$. If $a_i \le -1$ then $(\uck_\l,i) \sim (\uck')$ where $|\uck'| = |\uck_\l|-1$. Similarly, if $a_i \ge 1$ then $(\uck_\l,-i) \sim (\uck')$ where $|\uck'| = |\uck_\l|-1$.
\end{Lemma}
\begin{proof}
We prove the case $a_i \le -1$ (the case $a_i \ge 1$ is the same). Suppose $(\uck_\l,i)$ looks like 
$$\mu \leadsto \dots \leadsto \rho + \alpha_i \leadsto \rho \leadsto \dots \leadsto \l - c_m \alpha_{k_m} \leadsto \l \leadsto \l+\alpha_i$$
where the path from $\rho$ to $\l$ contains no $\alpha_i$. We will prove the result by induction on the length $\ell$ of this path. Notice that if $\ell = 1$ then we have $\dots \leadsto \rho \leadsto \rho - \alpha_i = \l \leadsto \l + \alpha_i$ and the result follows by applying $D_m$. 

If $\ell > 1$ then $k_m \ne i$. If $\la i, k_m \ra = 0$ then the path $\l - c_m \alpha_{k_m} \leadsto \l - c_m \alpha_{k_m} + \alpha_i \leadsto \l + \alpha_i$ are valid slides which means $(\uck_\l, i) = (\dots, c_m k_m, i) \sim (\dots, i, c_m k_m)$ and the result follows by induction. So we are left with the cases $k_m = i-1$ or $k_m = i+1$. 

If $k_m=i+1$ then since $\uck_\l$ is canonical it means $c_m=-1$. Now, if $\l_i \ge 0$ then 
$$\l+\alpha_{i+1} \leadsto \l+\alpha_{i+1}+\alpha_i \leadsto \l+\alpha_i$$
are valid slides and hence $(\uck_\l,i) = (\dots, - k_m, i) \sim (\dots, i, - k_m)$ and we are done by induction. Otherwise, since $\l \leadsto \l+\alpha_i$ is a valid slide, we must have $\l_i \ge -1$ which leaves us with $\l_i=-1$. But then $\l+\alpha_i \leadsto \l$ is a valid slide which means $\uck_\l = (\dots,i)$ (contradiction). 

If $k_m=i-1$ then all the $\alpha_a$ which appear between $\rho$ and $\l$ satisfy $a < i$. If $c_m = 1$ then 
$$\l = \rho + c \alpha_{i-1} + \nu$$
where $c \ge 1$ and $\la \nu, i \ra = 0$. This means $\l_i = \rho_i - c \le \rho_i - 1 \le -2$ since $\rho_i \le -1$. But this is a contradiction since $\l \leadsto \l+\alpha_i$ is a valid slide. On the other hand, if $c_m=-1$ then $(\uck_\l,i) = (\dots, -i, c_j k_j, \dots, c_m k_m, i)$ where for each $a = j, \dots, m$ either $k_a \le i-2$ or $k_a=i-1$ and $c_a = -1$. In either case one can check that $\tau \leadsto \tau + c_a \alpha_{k_a} \leadsto \tau + c_a \alpha_{k_a} - \alpha_i$ are valid slides where $\tau = \rho + \alpha_i + c_j \alpha_{k_j} + \dots + c_{a-1} \alpha_{k_{a-1}}$. This is because $\la \tau + c_a \alpha_{k_a}, i \ra \le \rho_i + 2 \le 1$. Thus we can repeatedly slide the $-i$ to the right to obtain  
$$(\uck_\l,i) \sim (\dots, c_{j-2} k_{j-2}, c_j k_j, \dots, c_m k_m, -i, i)$$ 
and we are done once again. This concludes the proof. 
\end{proof}

\begin{Proposition}\label{PROP:paths}
Any two minimal length paths between $\mu$ and $\l$ are slide equivalent. 
\end{Proposition}
\begin{proof}
Consider an arbitrary minimal length path from $\mu$ to $\l$. Let us write it as $(\uck, cj)$. It suffices to show that $(\uck, cj) \sim (\uck_\l)$. The proof is by induction on $|\uck_\l|$. Notice that by induction we can assume $(\uck) \sim (\uck_{\l-c\alpha_j})$. 

Suppose $c=-1$ and $(\uck_\l) = (\uck_{\l-\alpha_i}, i)$ for some $i$ (the other cases when  $c=1$ or $(\uck_\l) = (\uck_{\l+\alpha_i}, -i)$ are the same). Since $\l-\alpha_i \leadsto \l$ is a valid slide we have $\l_i \ge 1$. If $\l_i=1$ then by induction we have $(\uck_{\l-\alpha_i}) \sim (\uck,cj,-i)$ which means that 
$$(\uck_\l) \sim (\uck_{\l-\alpha_i},i) \sim (\uck,cj,-i,i) \sim (\uck,cj)$$
and we are done. 

If $\l_i \ge 2$ then $\l-\alpha_i+\alpha_j \leadsto \l+\alpha_j \leadsto \l$ are valid slides (the first slide can only fail to be valid if $\l_i=1$ and $\la i,j \ra = -1$). Hence we get 
$$(\uck_\l) \sim (\uck_{\l-\alpha_i},i) \sim (\uck_{\l-\alpha_i+\alpha_j},-j,i) \sim (\uck_{\l-\alpha_i+\alpha_j},i,-j) \sim (\uck,-j)$$
where the second and last equivalences follows by induction. This completes the proof.
\end{proof}

\subsection{Step (i)}

\begin{Proposition}\label{PROP:redefine}
Having fixed $\t_i \in \End^2(\1_\mu)$ and $\t_i \in \End^2(\1_{\mu-\alpha_i})$ one can redefine the remaining $\t_i \in \End^2(\1_\l)$ so that relations (\ref{EQ:N1}) and (\ref{EQ:N2}) hold.  
\end{Proposition}
\begin{proof}
We prove the case $\mu_i=0$ (the case $\mu_i=1$ is the same). Let us fix $\t_i I_\mu$ and $\t_i I_{\mu+\alpha_i}$. We redefine the remaining $\t_i I_\l$ in three steps. 

Step A. First we redefine each $\t_i I_\l$ with $\la \l, \Lambda_i \ra = 0$. To do this we use Part I of Proposition \ref{PROP:theta+transient} to redefine each $\t_i$ along the canonical path $\mu \leadsto \l$ (the condition $\la \l, \Lambda_i \ra = 0$ means that this path does not involve any slides along $\alpha_i$). 

Step B. Next we redefine each $\t_i I_\l$ with $\la \l, \Lambda_i \ra = -1$. In this case, the proof of Proposition \ref{PROP:middle} also gives a canonical path from $\mu-\alpha_i$ to $\l$ (without any slides along $\alpha_i$). Once again we redefine each $\t_i$, using Part I of Proposition \ref{PROP:theta+transient}, along this canonical path $\mu-\alpha_i \leadsto \l$. 

Step C. Finally, if $\la \l, \Lambda_i \ra > 0$ (resp. $\la \l, \Lambda_i \ra < -1$) then we redefine $\t_i$ along the canonical path from $\mu$ (resp. $\mu-\alpha_i$) to $\l$ using Parts I and II of Proposition \ref{PROP:theta+transient}. Notice that if, along the path, we wish to redefine $\t_i I_{\rho+\alpha_i}$ (resp. $\t_i I_{\rho-\alpha_i}$) then we can use Part II of \ref{PROP:theta+transient} because now $\t_i I_\rho$ and $\t_i I_{\rho-\alpha_i}$ (resp. $\t_i I_{\rho+\alpha_i}$) have already been fixed. 

It remains to show that relations (\ref{EQ:N1}) and (\ref{EQ:N2}) hold. To see (\ref{EQ:N1}) suppose $\l = \mu + \sum_j a_j \alpha_j$ with $a_i \ge 0$ (the case $a_i \le -1$ is the same). Suppose $a_j \ge 0$ (the case $a_j \le 0$ is the same). If $\l_j \ge -1$ then $(\uck_\l,j)$ is a path of minimal length and, by Proposition \ref{PROP:paths}, $(\uck_{\l+\alpha_j}) \sim (\uck_\l,j)$. By applying Lemma \ref{LEM:twopaths} repeatedly these two paths give the same redefinition of $\t_i I_{\l+\alpha_j}$. But redefining via the path $(\uck_\l,j)$ implies, by construction, that (\ref{EQ:N1}) holds. 

Similarly, if $\l_j \le -1$ then by Lemma \ref{LEM:reduce} $(\uck_\l) \sim (\uck_{\l+\alpha_j},-j)$. But redefining via the path $(\uck_{\l+\alpha_j},-j)$ implies again, by construction, that (\ref{EQ:N1}) holds. 

To see (\ref{EQ:N2}) suppose again $a_i \ge 0$ (the case $a_i \le -1$ is the same). The paths $(\uck_{\l-\alpha_i})$ and $(\uck_\l,-i)$ are minimal so if $\l_i \le 0$ then $(\uck_{\l-\alpha_i}) \sim (\uck_\l,-i)$ by Proposition \ref{PROP:paths}. By applying Lemma \ref{LEM:twopaths} repeatedly these paths give the same rescaling of $\t_i I_{\l-\alpha_i}$. But rescaling via the path $(\uck_\l,-i)$ implies, by construction, that (\ref{EQ:N2}) holds. 

Similarly, by Lemma \ref{LEM:reduce} we have $(\uck_{\l+\alpha_i}) \sim (\uck_\l,i)$ so if $\l_i \ge 0$ then rescaling $\t_i I_{\l+\alpha_i}$ via the path $(\uck_\l,i)$ implies, by construction again, that (\ref{EQ:N2}) holds.
\end{proof}

\begin{Proposition}\label{PROP:theta+transient}
There are two parts to redefining $\t_i \in \End^2(\1_\l)$. 

Part I: For $i \ne j \in I$, if $\l_j \le -1$ \(resp. $\l_j \ge 1$\) one can redefine $\t_i \in \End^2(\1_\l)$ by adding some transient map so that (\ref{EQ:N1}) holds inside $\End^2(\E_j \1_\l)$ \(resp. inside $\End^2(\1_\l \E_j)$\). 

Part II: For $i = j \in I$, if $\l_i \le -2$ \(resp. $\l_i \ge 2$\) one can redefine $\t_i \in \End^2(\1_\l)$ by adding some transient map so that (\ref{EQ:N2}) holds inside $\End(\E_i \E_i \1_\l)$ \(resp. inside $\End(\1_\l \E_i \E_i)$\).
\end{Proposition}
\begin{proof}
To prove the first claim suppose $\l_j \le -1$ (the case $\l_j \ge 1$ is the same). We know that relation (\ref{EQ:N1}) holds modulo transient maps. By Lemma \ref{LEM:EndE} it follows that 
$$(II \t_i) - (\t_i II) = (II \gamma) \in \End^2(\1_{\l+\alpha_i} \E_i \1_\l)$$ 
for some transient $\gamma \in \End^2(\1_\l)$. The result now follows if we redefine $\t_i \in \End^2(\1_\l)$ as $\t_i-\gamma$. 

To prove the second claim suppose $\l_i \le -2$ (the case $\l_i \ge 2$ is the same). From the proof of Proposition \ref{PROP:2} We know that relation (\ref{EQ:N2}) holds modulo maps of the form $(T_{ii}I)(II \gamma) \in \End^2(\E_i \E_i \1_\l)$ where $\gamma \in \End^2(\1_\l)$ is transient. Thus, if we redefine $\t_i \in \End^2(\1_\l)$ as $\t_i - \gamma$ then (\ref{EQ:N2}) will hold on the nose.
\end{proof}

\begin{Lemma}\label{LEM:twopaths}
Using Proposition \ref{PROP:theta+transient} there are two ways of redefining $\theta_i \in \End^2(\1_{\l \pm \alpha_a \pm \alpha_b})$ given $\theta_i \in \End^2(\1_\l)$, namely 
$$\theta_i I_\l \leadsto \theta_i I_{\l \pm \alpha_a} \leadsto \theta_i I_{\l \pm \alpha_a \pm \alpha_b} \ \ \text{ or }  \ \ \theta_i I_\l \leadsto \theta_i I_{\l \pm \alpha_b} \leadsto \theta_i I_{\l \pm \alpha_a \pm \alpha_b}.$$
However, these yield the same $\theta_i I_{\l \pm \alpha_a \pm \alpha_b}$. 
\end{Lemma}
\begin{proof}
We prove the case of $\theta_i I_{\l+\alpha_a+\alpha_b}$ (the other cases are the same). For simplicity we assume $\la a,b \ra = -1$ (if $a=b$ there is nothing to prove and the case $\la a,b \ra = 0$ is strictly easier). Let us denote by $\phi \in \End^2(\1_{\l+\alpha_a+\alpha_b})$ the difference between the two redefinitions. We need to show that $\phi = 0$. 

By considering $(T_{ab}I)(II \theta_i) \in \Hom^{-1}(\E_a \E_b \1_\l, \E_b \E_a \1_\l)$ we find that 
$$(\phi II)(IT_{ab}) \in \Hom^{-1}(\1_{\l+\alpha_a+\alpha_b} \E_a \E_b, \1_{\l+\alpha_a+\alpha_b} \E_b \E_a)$$
must be zero. Now we know that $\la \l+\alpha_a, \alpha_b \ra \ge -1 \Rightarrow \l_b \ge 0$ and likewise $\l_a \ge 0$. Then by Lemma \ref{LEM:rank} we know that the $\E_b$-rank of $T_{ab}II: \E_a \E_b \1_\l \F_a \rightarrow \E_b \E_a \1_\l \F_a$ is positive. Thus $(\phi II)(IT_{ab}) = 0$ implies $(\phi I) \in \End^2(\1_{\l+\alpha_a+\alpha_b} \E_b)$ is zero. But then $\1_{\l+\alpha_a+\alpha_b} \E_b \F_b$ contains $\l_b+1 \ge 1$ copies of $\1_{\l+\alpha_a+\alpha_b}$ which means that $\phi \in \End^2(\1_{\l+\alpha_a+\alpha_b})$ must also be zero. This completes the proof. 
\end{proof}

\subsection{Step (ii)}

By Remark \ref{rem:orient} we can assume $\mu_j=0$. Then if we take $\l=\mu-\alpha_i$ we have $\l_i=-2$ and $\l_j=1$. This means that a transient map in $\End^2(\E_i \1_\l \E_j)$ is of the form $I \phi I$ for some $\phi \in \End^2(\1_\l)$. Thus, by redefining $\t_i I_{\mu-\alpha_i}$, we can fix it so that (\ref{EQ:N3}) holds when $\l=\mu-\alpha_i$. 

Next, having fixed $\t_i I_\mu$ and $\t_i I_{\mu-\alpha_i}$ we can redefine the remaining $\t_i I_\l$ so that relations (\ref{EQ:N1}) and (\ref{EQ:N2}) hold (Proposition \ref{PROP:redefine}). Finally, since (\ref{EQ:N3}) holds for $\l=\mu-\alpha_i$ it holds for all $\l$ by Corollary \ref{COR:N3}. This completes Step (ii). 

\begin{Proposition}\label{PROP:N3}
Suppose (\ref{EQ:N1}) and (\ref{EQ:N2}) both hold for an edge $i \rightarrow j$ in $\Gamma$ and all $\l$. Assuming that (\ref{EQ:N3}) holds when $\l=\nu$ then it also holds for $\l=\nu+\alpha_k$ if $\la \nu+\alpha_i,\alpha_k \ra \ge -1$ and for $\l=\nu-\alpha_k$ if $\la \nu+\alpha_i,\alpha_k \ra \le 1$. 
\end{Proposition}
\begin{proof}
We assume $\la \nu+\alpha_i,\alpha_k \ra \ge -1$ (the case $\la \nu+\alpha_i,\alpha_k \ra \le 1$ is the same). There are three cases depending on whether $k=i, k=j$ or $k \ne i,j$. All three are proven in the same manner so we only prove the more complicated case $k=j$. 

Consider the composition 
$$\E_j \E_i \E_j \1_{\nu-\alpha_j} \xrightarrow{(IT_{ji})(IT_{ij})} \E_j \E_i \E_j \1_{\nu-\alpha_j} \la 2 \ra \xrightarrow{(IT_{jj})(T_{ji}I)} \E_i \E_j \E_j \1_{\nu-\alpha_j} \la 1 \ra.$$
On the one hand, using (\ref{EQ:N3}) when $\l = \nu$ this is equal to 
$$(IT_{jj})(T_{ji}I)[t_{ij}(IX_iI) + t_{ji}(IIX_j)] = [t_{ij}(X_iII) + t_{ji}(IX_jI)](IT_{jj})(T_{ji}I) - t_{ji}(T_{ji}I)$$
where we used the affine nilHecke relation to commute the $X_j$. On the other hand, this composition equals
\begin{align*}
(T'_{jji})(IT_{ij}) 
&= (T_{jji})(IT_{ij}) = (T_{ji}I)(IT_{ji})(T_{jj}I)(IT_{ij}) \\
&= (T_{ji}I)(T_{ij})(IT_{jj})(T_{ji}I) - t_{ji} (T_{ji}I) \\
&= [t_{ij}(X_iII) + t_{ji}(IX_jI) + (\gamma I)](IT_{jj})(T_{ji}I) - t_{ji}(T_{ji}I)
\end{align*}
for some transient map $\gamma \in \End^2(\E_i \E_j \1_{\nu})$. Comparing expressions we get $(\gamma I)(IT_{jj})(T_{ji}I) = 0$. It remains to show $\gamma = 0$. 

By Corollary \ref{COR:nonzero} we get $(I \gamma)(v_{jj}I)(Iv_{ij}) I_{\nu} = 0$. Now $(Iv_{ij})(Iu_{ji}) \sim id$ and $(v_{jj}I)(u_{jj}I) \sim id$ since $\nu_j \ge 0$. So composing on the right with $(Iu_{ji})(u_{jj}I)$ we get that $(I \gamma) \in \End^2(\F_j \E_i \E_j \1_{\nu})$ is zero. If we compose with $\E_j$ on the left and simplify we get several copies of $\gamma \in \End^2(\E_i \E_j \1_{\nu+\alpha_j})$ (this uses that $\la \nu+\alpha_i+\alpha_j,\alpha_j \ra \ge 1$). It follows that $\gamma = 0$ and we are done. 
\end{proof}

\begin{Corollary}\label{COR:N3}
Suppose (\ref{EQ:N1}) and (\ref{EQ:N2}) hold for an edge $i \rightarrow j$ in $\Gamma$ and all $\l$. If (\ref{EQ:N3}) holds for $\l=\mu-\alpha_i$ then it holds for all $\l$.
\end{Corollary}
\begin{proof}
Proposition \ref{PROP:N3} states that for a valid slide $\nu \leadsto \nu \pm \alpha_k$ relation (\ref{EQ:N3}) holds for $\l = \nu-\alpha_i \pm \alpha_k$ assuming it holds for $\l = \nu-\alpha_i$. The result follows since we can reach any weight via a sequence of valid slides starting from $\mu$. 
\end{proof}

\subsection{Step (iii)}

We will show that (\ref{EQ:N4}) holds when $\l = \mu-\alpha_j$. It then follows by Corollary \ref{COR:N3} that (\ref{EQ:N4}) holds for all $\l$. 

We know that $(T_{ij})(T_{ji}) = t_{ij}(IX_i) + t_{ji}(X_jI) + \phi \in \End^2(\E_j \1_{\mu-\alpha_j} \E_i)$ for some transient map $\phi$. Now consider the composition $(T_{ij})(T_{ji})(T_{ij}) \in \Hom^3(\E_i \1_{\mu-\alpha_i} \E_j, \E_j \1_{\mu-\alpha_j} \E_i)$. On the one hand, using relation (\ref{EQ:N3}), this equals
$$(T_{ij})(t_{ij}(X_iI)+t_{ji}(IX_j)) = (t_{ij}(IX_i)+t_{ji}(X_jI))(T_{ij}).$$
On the other hand, the composition equals $(t_{ij}(IX_i)+t_{ji}(X_jI)+\phi)(T_{ij})$. It follows that $\phi(T_{ij}) = 0$. 

It remains to show that $\phi(T_{ij})=0 \Rightarrow \phi=0$. We consider two cases depending on whether $\mu_j=0$ or $\mu_j=1$. If $\mu_j=0$ then $\phi$ is of the form $(I \gamma I) \in \End^2(\E_j \1_{\mu-\alpha_j} \E_i)$ where $\gamma \in \End^2(\1_{\mu-\alpha_j})$ is transient. Composing on the left with $\F_j$ we get that 
$$\F_j \E_i \1_{\mu-\alpha_i} \E_j \xrightarrow{(IT_{ij})} \F_j \E_j \1_{\mu-\alpha_j} \E_i \la 1 \ra \xrightarrow{(II \gamma I)} \F_j \E_j \1_{\mu-\alpha_j} \E_i \la 3 \ra$$
is zero. Since $\mu_j=0$ the first map induces an isomorphism between one copy of $\1_{\mu-\alpha_j} \E_i$. This means that $(\gamma I) \in \End^2(\1_{\mu-\alpha_j} \E_i)$ is zero and hence $\phi = 0$. 

Similarly, if $\mu_j=1$ then $\phi$ is of the form $(\gamma II) \in \End^2(\1_\mu \E_j \E_i)$. This time we compose with $\F_j$ on the right and find that 
$$\1_\mu \E_i \E_j \F_j \xrightarrow{(IT_{ij}I)} \1_\mu \E_j \E_i \F_j \la 1 \ra \xrightarrow{(\gamma III)}  \1_\mu \E_j \E_i \F_j \la 3 \ra$$
is zero. Since $\mu_j = 1$ the first map again induces an isomorphism between one copy of $\1_\mu \E_i$ and hence $(\gamma I) \in \End^2(\1_\mu \E_i)$ is zero. So again we find $\phi=0$ which completes the argument.

\section{An alternative definition of a $(\g,\t)$ action}\label{sec:alt}

An equivalent definition of a $(\g,\t)$ action involves replacing condition (\ref{co:theta}) with the following. 

\begin{framed}
The composition $\E_i \E_i$ decomposes as $\E_i^{(2)} \la -1 \ra \oplus \E_i^{(2)} \la 1 \ra$ for some 1-morphism $\E_i^{(2)}$. Moreover, if $\theta \in Y_\k$ where $\la \theta, \alpha_i \ra \ne 0$ (resp. $\la \theta, \alpha_i \ra = 0$) then $I \t I \in \End^2(\E_i \1_\l \E_i)$ induces a nonzero map (resp. the zero map) between the summands $\E_i^{(2)} \la 1 \ra$ on either side.
\end{framed}

\begin{Lemma}
The condition above implies condition (\ref{co:theta}) from section \ref{sec:gaction}. 
\end{Lemma}
\begin{proof}
A very similar version of this result appears as \cite[Lemma 3.6]{CLa}. We prove the case $\l_i \ge 0$ (the case $\l_i \le 0$ is the same). First, we know that 
\begin{equation}\label{eq:bubble1}
(I \t II): \E_i \1_{\l} \E_i \F_i \rightarrow \E_i \1_{\l} \E_i \F_i \la 2 \ra
\end{equation}
induces an isomorphism between the summands $\E_i^{(2)} \F_i \1_{\l} \la 1 \ra$ on either side. Since 
$$\E_i^{(2)} \F_i \1_{\l} \cong \F_i \E_i^{(2)} \1_{\l} \bigoplus_{[\l_i+1]} \E_i \1_{\l}$$
this means that the total $\E_i \1_\l$-rank of the map in (\ref{eq:bubble1}) is at least $\l_i+1$.

On the other hand, the map in (\ref{eq:bubble1}) induces a map 
$$(I \t II) \bigoplus_{[\l_i]} (I \t): \E_i \1_{\l} \F_i \E_i \bigoplus_{[\l_i]} \E_i \1_{\l} \rightarrow \E_i \1_{\l} \F_i \E_i \la 2 \ra \bigoplus_{[\l_i]} \E_i \1_{\l} \la 2 \ra.$$
Thus the total $\E_i \1_\l$-rank of the map in (\ref{eq:bubble1}) is equal to the total $\E_i \1_\l$-rank of $(I \t II) \in \End^2(\E_i \1_{\l} \F_i \E_i)$. Since $\E_i \1_{\l} \F_i \E_i \cong \F_i \E_i \E_i \1_{\l} \bigoplus_{[\l_i+2]} \E_i \1_{\l}$ where $\F_i \E_i \E_i \1_{\l}$ contains no summands $\E_i \1_\l$ we get that $(I \t I) \in \End^2(\E_i \1_{\l} \F_i)$ has total $\1_{\l+\alpha_i}$-rank at least $\l_i+1$. The result follows because by degree reasons the $\1_{\l+\alpha_i}$-rank of $(I \t I) \in \End^2(\E_i \1_{\l} \F_i)$ cannot be any larger than $\l_i+1$. 
\end{proof}

This alternative condition is similar to the condition present in the definition of a geometric categorical $\sl_n$ action from \cite{CK3}. In that geometric setup the existence of $\E_i^{(2)}$ is obvious and checking the condition above is easier than checking condition (\ref{co:theta}). On the other hand, condition (\ref{co:theta}) is easier to check in other setups such as \cite{CLi2} where even the existence of divided powers $\E_i^{(r)}$ is very difficult.

\section{Applications}\label{sec:applications}

\subsection{Categorical vertex operators}\label{sec:vertex}

In \cite{CLi1} we explained how, starting with the zig-zag algebra $A_\Gamma$ associated to the Dynkin diagram $\Gamma$ of a finite type Lie algebra $\g$, one obtains a Heisenberg algebra $\h_\Gamma$ and a 2-category $\H_\Gamma$ which categorifies it. In the process we also categorified the Fock space of $\h_\Gamma$ using a 2-category $\sF_\Gamma$. 

In the subsequent paper \cite{CLi2} we showed how to define a $(\hg,\t)$ action on the 2-category $\K_{\sF,\Gamma} := \Kom(\sF_\Gamma) \otimes_\Z \Z[Y]$ where $\Kom(\cdot)$ denotes the homotopy category and $(\cdot) \otimes_\Z \Z[Y]$ means that we have one copy of $(\cdot)$ for each element of the root lattice $Y$.

This construction categorifies the Frenkel-Kac-Segal vertex operator construction \cite{FK,Se} of the basic representation of $\hg$. One subtlety here is that $\hg$ is the affine Lie algebra of $\g$ in its {\it loop} presentation. We will show in future work that such a $(\hg,\t)$ action is equivalent to a $(\hg,\t)$ action where $\hg$ is in its Kac-Moody presentation. Moreover, in the category $\K_{\sF,\Gamma}$ all transient maps are zero so we do not need to worry about them. Thus Theorem \ref{THM:main} implies the following. 

\begin{Theorem}\label{THM:KLRvertex}
The quiver Hecke algebras associated to $\hg$ act on $\K_{\sF,\Gamma}$. 
\end{Theorem}

It seems very difficult to explicitly construct this action on $\K_{\sF,\Gamma}$. Indeed, this result is surprising because it is even difficult to show by hand that $\E_i^2 \cong \oplus_{[2]} \E_i^{(2)}$ holds in $\K_{\sF,\Gamma}$. In fact, an explicit form of $\E_i^{(k)}$ for $k > 2$ exists only conjecturally. An immediate Corollary of Theorem \ref{THM:KLRvertex} and \cite{CLa} is the following. 

\begin{Corollary}\label{COR:KLRvertex}
The 2-category $\K_{\sF,\Gamma}$ is a 2-representation of $\dot{\mathcal{U}}_Q(\hg)$ in the sense of Khovanov-Lauda. 
\end{Corollary}
\begin{Remark} Unfortunately, we do not know for what choice of $Q$ this is a 2-representation. If $\g$ is of type A then the space of such $Q$'s is parametrized by $\k^\times$. 
\end{Remark}

\subsection{The affine Grassmannian and geometric categorical $\g$ actions}\label{sec:quiver}

In \cite[section 2.2]{CK3} we introduced the idea of a geometric categorical $\g$ action. This definition was helpful because it was easier to check in geometric situations. Briefly, such an action consists of the following data.
\begin{enumerate}
\item A collection of smooth varieties $Y(\l)$ for $\l \in X$.
\item A collection of kernels
$$\sE^{(r)}_i(\l) \in D(Y(\l) \times Y(\l + r \alpha_i)) \text{ and } \sF^{(r)}_i(\l) \in D(Y(\l + r\alpha_i) \times Y(\l))$$
where $D(Y)$ denotes the derived category of coherent sheaves on $Y$. 
\item For each $Y(\l)$ a flat deformation $\tY(\l) \rightarrow Y_\k$.  
\end{enumerate}

From this data we obtain a 2-category $\K$ where the objects are $D(Y(\l))$, the 1-morphisms are kernels and 2-morphisms are maps between kernels. The extra data of the deformation $\tY(\l)$ can be used to obtain a linear map $Y_\k \rightarrow \End^2(\1_\l)$ as follows. From the standard short exact sequence 
$$0 \rightarrow T_{Y(\l)} \rightarrow T_{\tY(\l)}|_{Y(\l)} \rightarrow \O_{Y(\l)} \{2\} \otimes_\k Y_\k \rightarrow 0$$
one obtains the Kodaira-Spencer map $Y_\k \rightarrow H^1(T_{Y(\l)} \{-2\})$ (the $\{2\}$ is a grading shift corresponding to a $\C^\times$ action on $\tY(\l)$ which acts on $Y_\k$ with weight $2$). On the other hand, the Hochschild-Kostant-Rosenberg isomorphism states that 
$$\End^2(\1_\l) = HH^2(Y(\l)) \cong H^0(\wedge^2 T_{Y(\l)}) \oplus H^1(T_{Y(\l)}) \oplus H^2(\O_{Y(\l)})$$
and so we get a linear map $Y_\k \rightarrow HH^2(Y(\l)\{-2\})$. Thus, identifying the shift $\la 1 \ra$ with $[1]\{-1\}$, we obtain $Y_\k \rightarrow \End^2(\1_\l)$. 

\begin{Proposition}
A geometric categorical $\g$ action induces a $(\g,\t)$ action, assuming $\K$ satisfies conditions (\ref{co:vanish1}), (\ref{co:vanish2}) and (\ref{co:new}). 
\end{Proposition}
\begin{proof}
Conditions $(i)-(iv)$ are immediate. Most naturally we obtain a $(\g,\t)$ action in its alternative description from section \ref{sec:alt}. The alternative condition $(v)$ follows from conditions \cite[Sect. 2.2:(vi)+(x)]{CK3}. 
\end{proof}
\begin{Remark}
With hindsight of Theorem \ref{THM:main} the definition of a geometric categorical $\g$ action can be simplified. For instance, one does not need to require the existence of divide powers $\E_i^{(r)}$ for $r > 1$. 
\end{Remark}

An example of a geometric categorical $\g=\sl_n$ action categorifying $\Lambda_q^N(\C^m \otimes \C^n)$ was defined in \cite{CKL1,C}. More precisely, having fixed $N \in \N$ we take
$$Y(\l) := \{ \C[z]^m = L_0 \subset L_1 \subset \dots \subset L_n \subset \C(z)^m : z L_i \subset L_{i-1}, \dim(L_i/L_{i-1}) = k_i \}$$
where the $L_i$ are complex vector subspaces, $N = \sum_i k_i$ with $0 \le k_i \le m$, and $\l$ is determined by $\l_i = k_{i+1}-k_i$. These varieties are obtained from the affine Grassmannian of ${\rm PGL}_m$ via the convolution product. Namely, $Y(\l) = \Gr^{\Lambda_{k_1}} \tilde{\times} \dots \tilde{\times} \Gr^{\Lambda_{k_n}}$ where $\tilde{\times}$ denotes the convolution product. 

The kernels $\sE_i^{(r)}$ and $\sF_i^{(r)}$ are then defined by certain Hecke correspondence (see \cite[section 8.3]{C}). We denote the resulting 2-category $\K^n_{\Gr,m}$.

\begin{Theorem}\label{thm:new}
There exists an $(\sl_n,\t)$ action on $\K^n_{\Gr,m}$.
\end{Theorem}
\begin{proof}
The fact that there exists a geometric categorical $\sl_n$ action on $\K^n_{\Gr,m}$ was proved in \cite{CKL1} (with some of the details appearing also in \cite{CK3,C}). Since $\K^n_{\Gr,m}$ categorifies a finite dimensional representation of $\sl_n$, conditions (\ref{co:vanish1}) and (\ref{co:vanish2}) are immediate. Finally, condition (\ref{co:new}) is a an easy consequence of the particular representation being categorified (namely, $\Lambda_q^N(\C^m \otimes \C^n)$). 
\end{proof}

\begin{Corollary}\label{cor:new}
$\K^n_{\Gr,m}$ is a 2-representation of $\dot{\mathcal{U}}_Q(\sl_n)$ in the sense of Khovanov-Lauda.
\end{Corollary}

Note that in this case, since $\g=\sl_n$, this result holds without having to mod out by transient maps. Also, all choices of parameters $Q$ are equivalent so there is no ambiguity. 

\begin{Remark}
In \cite{CKL3} we constructed a geometric categorical $\g$ action on Nakajima quiver varieties which lifted Nakajima's action on K-theory \cite{N}. For a dominant weight $\Lambda$ one can define a 2-category $\K_{\rm Q}^{\Lambda}$ consisting of derived categories of coherent sheaves on Nakajima quiver varieties with highest weight $\Lambda$. Subsequently, Theorem \ref{thm:new} and Corollary \ref{cor:new} also hold if we replace $\K^n_{\Gr,m}$ with $\K_{\rm Q}^{\Lambda}$ (except that we have to mod out by transient maps if $\g \ne \sl_n$). 
\end{Remark}

\subsection{Rigidity of homological knot invariants}\label{sec:knots}

In \cite{C} we explained that given a categorification of the $U_q(\sl_{\infty})$ representation $\Lambda^{m \infty}(\C^m \otimes \C^{2 \infty}) = \lim_{N \rightarrow \infty} \Lambda_q^{mN}(\C^m \otimes \C^{2N})$ one obtains a homological knot invariant categorifying the Reshetikhin-Turaev knot invariants of type $\sl_m$. In fact, we explained that one only needs an action whose nonzero weight spaces are the same as those of $\Lambda^{m \infty}(\C^m \otimes \C^{2 \infty})$.

In principle, the homology you get will depend on the specific categorification. However, if the categorification is given by a 2-representation in the sense of Khovanov-Lauda then it can be calculated entirely from this information. Since such a 2-representation is induced by a $(\sl_\infty,\t)$ action we obtain the following.

\begin{Theorem}\label{thm:knots}
Any two $(\sl_\infty,\t)$ actions whose nonzero weight spaces are the same as those of $\Lambda^{m \infty}(\C^m \otimes \C^{2 \infty})$ yield isomorphic homological knot invariants.
\end{Theorem}

A variety of methods have been used to define homological knot invariants over the last few years. These include: derived categories of coherent sheaves \cite{CK1,CK2,C}, category $\O$ \cite{MS,Su}, matrix factorizations \cite{KR,W,Y} and foams \cite{MSV,LQR,QR}. Since all these invariants fit within the framework described above they define equivalent homologies. Thus Theorem \ref{THM:main} also implies a certain rigidity for homological knot invariants of Reshetikhin-Turaev type A. 

\appendix

\section{Spaces of morphisms}

In this section we collect a series of calculations of the dimension of spaces of maps between various 1-morpshisms. All these computations are basically performed in the same way, namely by repeatedly applying adjunction and simplifying until one ends up with $\End^i(\1_\l)$ which we have assumed to be zero if $i < 0$ and one-dimensional if $i=0$. All the proofs are independent of other results in the main body of this paper, meaning that they only use the definition of a $(\g,\t)$ action from section \ref{sec:gaction}. 

\subsection{Spaces not involving divided powers}

\begin{Lemma}\label{LEM:homEs} 
We have $\dim \End^d(\E_i \1_\l) \le \begin{cases} 1 & \text{ if } d = 0 \\ 0 & \text{ if } d < 0. \end{cases}$
\end{Lemma}
\begin{proof}
Suppose $\l_i \le 0$ (the case $\l_i \ge 0$ is the same). In this case we have 
\begin{align*}
\Hom(\E_i \1_\l, \E_i \1_\l \la d \ra) 
&\cong \Hom(\F_i \E_i \1_\l, \1_\l \la d+\l_i+1 \ra) \\
&\cong \Hom(\E_i \F_i \1_\l \bigoplus_{[-\l_i]} \1_\l, \1_\l \la d+\l_i+1 \ra).
\end{align*}
On the one hand, if $\l_i < 0$ then 
$$\dim \Hom(\1_\l, \bigoplus_{[-\l_i]} \1_\l \la d+\l_i+1 \ra) = \sum_{r=0}^{-\l_i-1} \dim \Hom(\1_\l, \1_\l \la d - 2r \ra) \le \begin{cases} 1 & \text{ if } d=0 < 0 \\ 0 & \text{ if } d < 0 \ \end{cases}$$
while on the other hand
$$\dim \Hom(\E_i \F_i \1_\l, \1_\l \la d+\l_i+1 \ra) = \dim \Hom(\E_i \1_{\l-\alpha_i}, \E_i \1_{\l-\alpha_i} \la d+2\l_i \ra) = 0.$$
where the last equality follows by induction on $\l_i$ (this is where we use the condition that $\1_{\l-r\alpha_i}=0$ for $r \gg 0$). The result follows. In the special case $\l_i=0$ the first term vanishes while the second is $\End^d(\E_i \1_{\l-\alpha_i})$ and the result follows again by induction. 
\end{proof}

\begin{Lemma}\label{LEM:homEEs}
We have $\dim \End^{-2}(\E_i \E_i \1_\mu) \le 1$ with equality if $\E_i \E_i \1_\mu \ne 0$.
\end{Lemma}
\begin{proof}
By applying the commutation relation twice one finds that 
\begin{align*}
\E_i \E_i \F_i \1_\nu \cong \F_i \E_i \E_i \1_\nu \bigoplus_{[2][\nu_i+1]} \E_i \1_\nu \ \ & \text{ if } \ \ \nu_i \ge -1 \\
\F_i \E_i \E_i \1_\nu \cong \E_i \E_i \F_i \1_\nu \bigoplus_{[2][-\nu_i-1]} \E_i \1_\nu \ \ & \text{ if } \ \ \nu_i \le -1. \end{align*}
Thus, if $\mu_i \ge -2$ then we get
\begin{align*}
&\End^{-2}(\E_i \E_i \1_\mu) \\
\cong &\Hom(\E_i \E_i \1_{\mu+\alpha_i}, \E_i \E_i \F_i \1_{\mu+\alpha_i} \la -\mu_i-3 \ra) \\
\cong &\Hom(\E_i \E_i \1_{\mu+\alpha_i}, \F_i \E_i \E_i \1_{\mu+\alpha_i} \la -\mu_i-3 \ra \bigoplus_{[2][\mu_i+3]} \E_i \E_i \1_{\mu+\alpha_i}\la -\mu_i-3 \ra) \\
\cong &\Hom(\E_i \E_i \1_{\mu+\alpha_i}, \E_i \E_i \1_{\mu+\alpha_i} \la -2\mu_i-8 \ra) \oplus \Hom(\E_i \1_{\mu+\alpha_i}, \bigoplus_{[2][\mu_i+3]} \E_i \1_{\mu+\alpha_i} \la -\mu_i-3 \ra).
\end{align*}
The left term vanishes by induction while, using Lemma \ref{LEM:homEs}, the only surviving term in the right hand sum is $\Hom(\E_i \1_{\mu+\alpha_i}, \E_i \1_{\mu+\alpha_i}) \cong \k$. Thus $\End^{-2}(\E_i \E_i \1_\mu) \cong \k$ (as long as $\E_i \1_{\mu+\alpha_i} \ne 0$). The case $\mu_i \le -2$ is similar.
\end{proof}

\begin{Lemma}\label{LEM:homEEEs}
We have $\dim \End^{-6}(\E_i \E_i \E_i \1_\mu) \le 1$ with equality if $\E_i \E_i \E_i \1_\mu \ne 0$. 
\end{Lemma}
\begin{proof}
By applying the commutation relation three times one finds that 
\begin{align*}
\E_i \E_i \E_i \F_i \1_\nu \cong \F_i \E_i \E_i \E_i \1_\nu \bigoplus_{[3][\nu_i+2]} \E_i \E_i \1_\nu \ \ & \text{ if } \ \ \nu_i \ge -2 \\
\F_i \E_i \E_i \E_i \1_\nu \cong \E_i \E_i \E_i \F_i \1_\nu \bigoplus_{[3][-\nu_i-2]} \E_i \E_i \1_\nu \ \ & \text{ if } \ \ \nu_i \le -2. 
\end{align*}
Thus, if $\mu_i \ge -3$ then we get 
\begin{align*}
&\End^{-6}(\E_i \E_i \E_i \1_\mu) \\
\cong &\Hom(\E_i \E_i \1_{\mu+\alpha_i}, \E_i \E_i \E_i \F_i \1_{\mu+\alpha_i} \la -\mu_i-7 \ra) \\
\cong &\Hom(\E_i \E_i \1_{\mu+\alpha_i}, \F_i \E_i \E_i \E_i \1_{\mu+\alpha_i} \la -\mu_i-7 \ra \bigoplus_{[3][\mu_i+4]} \E_i \E_i \1_{\mu+\alpha_i}\la -\mu_i-7 \ra) \\
\cong &\Hom(\E_i \E_i \E_i \1_{\mu+\alpha_i}, \E_i \E_i \E_i \1_{\mu+\alpha_i} \la -2\mu_i-14 \ra) \oplus \Hom(\E_i \E_i \1_{\mu+\alpha_i}, \bigoplus_{[3][\mu_i+4]} \E_i \E_i \1_{\mu+\alpha_i} \la -\mu_i-7 \ra).
\end{align*}
The left term vanishes by induction while, using Lemma \ref{LEM:homEEs}, the only surviving term in the right hand sum is $\Hom(\E_i \E_i \1_{\mu+\alpha_i}, \E_i \E_i \1_{\mu+\alpha_i} \la -2 \ra)$. Thus, by Lemma \ref{LEM:homEEs}, $\dim \End^{-6}(\E_i \E_i \E_i \1_\mu) \le 1$ and equality holds if $\E_i \E_i \1_{\mu+\alpha_i} \ne 0$. The case $\mu_i \le -3$ is similar. 
\end{proof}

\begin{Lemma}\label{LEM:Ts1} 
Suppose $i,j \in I$ with $\la i,j \ra = -1$. Then 
\begin{align}
\label{eq:hom1} \dim \Hom(\E_i \E_j \1_\l, \E_j \E_i \1_\l \la d \ra) &\le \begin{cases} 1 & \text{ if } d = 1 \\ 0 & \text{ if } d < 1 \end{cases} \\
\label{eq:hom2} \dim \Hom(\E_i \E_j \1_\l, \E_i \E_j \1_\l \la d \ra) &\le \begin{cases} 1 & \text{ if } d = 0 \\ 0 & \text{ if } d < 0. \end{cases}
\end{align}
In (\ref{eq:hom1}) equality holds when $d=1$ if $\E_i \E_j \1_\l$ and $\E_j \E_i \1_\l$ are both nonzero. Likewise in (\ref{eq:hom2}) equality holds when $d=0$ if $\E_i \E_j \1_\l$ is nonzero. 
\end{Lemma}
\begin{proof}
Suppose $\mu_i+\mu_j \le 0$ (the case $\mu_i+\mu_j \ge 0$ is the same). This means that either $\mu_i \le 0$ or $\mu_j \le 0$. We assume $\mu_j \le 0$ as the other case is similar. 

The proof is by induction on $\mu_i+\mu_j$. The base case follows from condition (\ref{co:vanish1}) where we take $\alpha = \alpha_i + \alpha_j$. This is the only time we use this case of condition (\ref{co:vanish1}) in this paper. 

To obtain the induction step we first note that
\begin{align}
\nonumber \Hom(\E_i \E_j \1_\mu, \E_j \E_i \1_\mu \la d \ra) \cong & \Hom(\F_j \E_i \E_j \1_\mu \la - \mu_j \ra, \E_i \1_\mu \la d \ra) \\
\label{eq:Y} \cong & \Hom(\E_i \E_j \F_j \1_\mu, \E_i \1_\mu \la \mu_j+d \ra) \bigoplus_{r=0}^{-\mu_j-1} \Hom(\E_i \1_\mu, \E_i \1_\mu \la - 1 - 2r + d \ra).
\end{align}
Now, if $d \le 1$ then by Lemma \ref{LEM:homEs} every term in the right hand sum is zero unless $r=0$, $d=1$ and $\E_i \1_\mu \ne 0$ (and $\mu_i < 0$) in which case it is one-dimensional. Meanwhile, the left hand term is equal to 
$$\Hom(\E_i \E_j \1_{\mu-\alpha_j}, \E_i (\1_{\mu-\alpha_j} \F_j)_L \la \mu_j + d \ra) \cong \Hom(\E_i \E_j \1_{\mu-\alpha_j}, \E_i \E_j \1_{\mu-\alpha_j} \la 2 \mu_j + d - 1 \ra).$$
By (\ref{eq:hom2}) this is zero unless $d=1, \mu_j=0$ and $\E_i \E_j \1_{\mu-\alpha_j} \ne 0$ in which case it is one-dimensional. Thus (\ref{eq:hom1}) when $\l = \mu$ follows from (\ref{eq:hom2}) when $\l=\mu-\alpha_j$. 

Now, if we study (\ref{eq:hom2}) we have 
\begin{align*}
\Hom(\E_i \E_j \1_\mu, \E_i \E_j \1_\mu \la d \ra) 
&\cong \Hom(\1_\mu, \F_j \F_i \E_i \E_j \1_\mu \la \mu_i+\mu_j+1+d \ra).
\end{align*}
If $\mu_i \le 1$ then we can simplify $\F_j \F_i \E_i \F_j \1_\mu$ directly (this is the easier case). We assume the more difficult situation that $\mu_i \ge 1$ in which case we have 
\begin{align*}
\F_j \E_i \F_i \E_j \1_\mu 
&\cong \F_j \F_i \E_i \E_j \1_\mu \bigoplus_{[\mu_i-1]} \F_j \E_j \1_\mu \\
&\cong \F_j \F_i \E_i \E_i \1_\mu \bigoplus_{[\mu_i-1]} \E_j \F_j \1_\mu \bigoplus_{[\mu_i-1][-\mu_j]} \1_\mu
\end{align*}
while the left hand side of the equation above equals 
\begin{align*}
\E_i \F_j \E_j \F_i \1_\mu 
&\cong \E_i \E_j \F_j \F_i \1_\mu \bigoplus_{[-\mu_j-1]} \E_i \F_i \1_\mu \\
&\cong \E_i \E_j \F_j \F_i \1_\mu \bigoplus_{[-\mu_j-1]} \F_i \E_i \1_\mu \bigoplus_{[\mu_i][-\mu_j-1]} \1_\mu.
\end{align*}
Now, if $a,b \in \N$ with $b \ge a$ then $[a][b-1] = [a-1][b]+[b-a]$. Thus, taking $a=\mu_i$, $b=-\mu_j$ and using that morphisms have a unique decomposition we get that 
\begin{equation}\label{eq:hom3}
\F_j \F_i \E_i \E_j \1_\mu \bigoplus_{[\mu_i-1]} \E_j \F_j \1_\mu \cong \E_i \E_j \F_j \F_i \1_\mu \bigoplus_{[-\mu_j-1]} \F_i \E_i \1_\mu \bigoplus_{[-\mu_i-\mu_j]} \1_\mu.
\end{equation}
Now
$$\Hom(\1_\mu, \E_j \F_j \1_\mu \la s \ra) \cong \Hom((\1_{\mu-\alpha_j} \F_j)_L, \E_j \1_{\mu-\alpha_j} \la s \ra) \cong \Hom(\1_\mu \E_j, \1_\mu \E_j \la s + \mu_j-1 \ra)$$
so by Lemma \ref{LEM:homEs} this vanishes if $s \le -\mu_j$. In particular, this means that 
$$\Hom(\1_\mu, \bigoplus_{[\mu_i-1]} \E_j \F_j \1_\mu \la \mu_i+\mu_j+1+d \ra)=0 \ \ \text{ if } \ \ d \le 0.$$
Similarly, one finds that $\Hom(\1_\mu, \bigoplus_{[-\mu_j-1]} \F_i \E_i \1_\mu \la \mu_i+\mu_j+1+d \ra) = 0$ if $d \le 0$. Finally, 
\begin{align}
\Hom(\1_\mu, \E_i \E_j \F_j \F_i \1_\mu \la \mu_i+\mu_j+1+d \ra) 
\nonumber & \cong \Hom((\1_{\mu-\alpha_i} \F_i)_R (\1_{\nu} \F_j)_R, \E_i \E_j \1_\nu \la \mu_i+\mu_j+1+d \ra) \\
\label{eq:hom5} & \cong \Hom(\E_i \E_j \1_\nu, \E_i \E_j \1_\nu \la 2(\mu_i + \mu_j)+d \ra)
\end{align}
where $\nu = \mu-\alpha_i-\alpha_j$. Now, $\la \nu, \alpha_i+\alpha_j \ra = \la \mu, \alpha_i+\alpha_j \ra - 2$ so by induction this is zero if $d \le 0$ (unless $\mu_i+\mu_j=0$ in which case it is one-dimensional). Thus we get that 
\begin{align*}
\Hom(\E_i \E_j \1_\mu, \E_i \E_j \1_\mu \la d \ra) 
& \cong \Hom(\1_\mu, \F_j \F_i \E_i \E_i \1_\mu \la \mu_i+\mu_j+1+d \ra) \\
& \cong \Hom(\1_\mu, \bigoplus_{[-\mu_i-\mu_j]} \1_\mu \la \mu_i+\mu_j+1+d \ra) \\
& \cong \bigoplus_{r=0}^{-\mu_i-\mu_j-1} \Hom(\1_\mu, \1_\mu \la d - 2r \ra). 
\end{align*}
Each term in the last direct sum above is zero if $d < 0$ or $r > 0$ and one-dimensional if $d = 0 = r$ (unless $d = \mu_i+\mu_j=0$ in which case the sum vanishes but then (\ref{eq:hom5}) contributes $1$ to the dimension). In conclusion, equation (\ref{eq:hom2}) when $\l = \mu$ follows from (\ref{eq:hom1}) when $\l=\mu-\alpha_i-\alpha_j$. This completes the induction. 
\end{proof}

\begin{Lemma}\label{LEM:Ts2} 
Suppose $i,j \in I$ with $\la i,j \ra = 0$. Then 
\begin{align}
\label{eq:hom'1} \dim \Hom(\E_i \E_j \1_\l, \E_j \E_i \1_\l \la d \ra) &\le \begin{cases} 1 & \text{ if } d = 0 \\ 0 & \text{ if } d < 0 \end{cases} \\
\label{eq:hom'2} \dim \Hom(\E_i \E_j \1_\l, \E_i \E_j \1_\l \la d \ra) &\le \begin{cases} 1 & \text{ if } d = 0 \\ 0 & \text{ if } d < 0. \end{cases}
\end{align}
In (\ref{eq:hom'1}) equality holds when $d=0$ if $\E_i \E_j \1_\l$ and $\E_j \E_i \1_\l$ are both nonzero. Likewise in (\ref{eq:hom'2}) equality holds when $d=0$ if $\E_i \E_j \1_\l$ is nonzero. 
\end{Lemma}
\begin{proof}
The argument here is a much simpler version of that in the proof of Lemma \ref{LEM:Ts1}. The main difference here is that $\E_i \E_j \cong \E_j \E_i$ so one only needs to mimic the first part of the proof of Lemma \ref{LEM:Ts1}. We leave the details to the reader. 
\end{proof}

\begin{Lemma}\label{LEM:EF-FEhom} If $i \ne j \in I$ then
\begin{align*}
\dim \Hom(\F_j \E_i \1_\l, \E_i \F_j \1_\l \la d \ra) = \dim \Hom(\E_i \F_j \1_\l, \F_j \E_i \1_\l \la d \ra) &= \begin{cases} 1 & \text{ if } d = 0 \text{ and } \F_j \E_i \1_\l \ne 0 \\ 0 & \text{ if } d < 0 \end{cases} \\ 
\dim \Hom(\F_i \E_i \1_\l, \1_\l \la d \ra) = \dim \Hom(\1_\l, \F_i \E_i \1_\l \la d \ra) &= \begin{cases} 1 & \text{ if } d = \l_i+1 \text{ and } \E_i \1_\l \ne 0 \\ 0 & \text{ if } d < \l_i+1 \end{cases} \\
\dim \Hom(\E_i \F_i \1_\l, \1_\l \la d \ra) = \dim \Hom(\1_\l, \E_i \F_i \1_\l \la d \ra) &= \begin{cases} 1 & \text{ if } d = -\l_i+1 \text{ and } \F_i \1_\l \ne 0 \\ 0 & \text{ if } d < -\l_i+1. \end{cases} 
\end{align*}
\end{Lemma}
\begin{proof}
There are three cases to consider in proving the first equality. If $i \ne j$ with $\la i,j \ra = -1$ then 
$$\Hom(\F_j \E_i \1_\l, \E_i \F_j \1_\l \la d \ra) \cong \Hom(\E_i (\1_{\l-\alpha_j} \F_j)_R , (\1_{\l+\alpha_i-\alpha_j} \F_j)_R \E_i \la d \ra) \cong \Hom(\E_i \E_j, \E_j \E_i \la d+1 \ra)$$
and the result follows from Lemma \ref{LEM:Ts1}. Likewise, if $i \ne j$ with $\la i,j \ra = 0$ the result follows from Lemma \ref{LEM:Ts2}. Finally, if $i=j$ then $\Hom(\F_i \E_i \1_\l, \E_i \F_i \1_\l \la d \ra) \cong \Hom(\E_i \E_i \1_{\l-\alpha_i}, \E_i \E_i \1_{\l-\alpha_i} \la d-2 \ra)$ and the result follows from Lemma \ref{LEM:homEs}. The second equality follows similarly.

The second and third pairs of equalities follow directly by adjunction together with Lemma \ref{LEM:homEs}. 
\end{proof}

\begin{Lemma}\label{LEM:nonvan}
For any $i,j \in I$ we have: 
\begin{enumerate}
\item \label{eq:nonvan1} $\E_i \1_\l \ne 0$ if and only if $\1_\l$ and $\1_{\l+\alpha_i}$ are both nonzero.
\item \label{eq:nonvan2} $\E_j \E_i \1_\l \ne 0$ if and only if $\1_\l, \1_{\l+\alpha_i}$ and $\1_{\l+\alpha_i+\alpha_j}$ are all nonzero. 
\item \label{eq:nonvan3} $\E_i \1_\l \F_j \ne 0$ if and only if $\1_\l, \1_{\l+\alpha_i}, \1_{\l+\alpha_j}$ and $\1_{\l+\alpha_i+\alpha_j}$ are all nonzero. 
\item \label{eq:nonvan4} $\E_i \1_\l \E_i \F_j \ne 0$ if and only if $\1_{\l+r\alpha_i}$ and $\1_{\l+\alpha_j+r\alpha_i}$ are nonzero for $-1 \le r \le 1$. 
\end{enumerate}
\end{Lemma}
\begin{proof}
Clearly if $\1_\l=0$ or $\1_{\l+\alpha_i} = 0$ then $\E_i \1_\l = 0$. Conversely, suppose $\1_\l$ and $\1_{\l+\alpha_i}$ are nonzero. If $\l_i \ge -1$ then 
$$\E_i \1_\l \F_i \cong \F_i \E_i \1_{\l+\alpha_i} \bigoplus_{[\l_i+2]} \1_{\l+\alpha_i}$$
which is nonzero since $\1_{\l+\alpha_i} \ne 0$. Thus $\E_i \1_\l \ne 0$. Similarly, if $\l_i \le -1$ then $\F_i \E_i \1_\l \ne 0$ so $\E_i \1_\l \ne 0$. 

This proves (i). The argument for (ii) is the same. Namely, if $\l_i \ge -1$ then 
$$\E_j \E_i \1_\l \F_i \cong \E_j \F_i \E_i \1_{\l+\alpha_i} \bigoplus_{[\l_i+2]} \E_j \1_{\l+\alpha_i}$$
which is nonzero since $\1_{\l+\alpha_i}$ and $\1_{\l+\alpha_i+\alpha_j}$ are nonzero (and similarly if $\l_i \le -1$). 

We now prove (iii). Since $\E_i \F_j \cong \F_j \E_i$ it is clear that $\1_\l, \1_{\l+\alpha_i}, \1_{\l+\alpha_j}, \1_{\l+\alpha_i+\alpha_j}$ are all nonzero if $\E_i \1_\l \F_j \ne 0$. Conversely, suppose first that $\l_i \le -1$. Then we compose with $\F_i$ to get 
$$\F_i \E_i \1_\l \F_j \cong \E_i \F_i \1_\l \F_j \bigoplus_{[-\l_i]} \1_\l \F_j$$
which is nonzero since $\1_\l$ and $\1_{\l+\alpha_j}$ are nonzero. If $\l_i \ge 0$ then we precompose with $\F_i$ to get
$$\E_i \1_\l \F_j \F_i \cong \F_j \E_i \F_i \1_{\l+\alpha_j+\alpha_i} \cong \F_j \E_i \F_i \1_{\l+\alpha_j} \bigoplus_{[\l_i+2+\la i,j \ra]} \F_j \1_{\l+\alpha_i+\alpha_j}$$
which is nonzero since $\1_{\l+\alpha_i}$ and $\1_{\l+\alpha_i+\alpha_j}$ are nonzero (here we use that $\l_j+2+\la i,j \ra \ge 1$). 

The proof of (iv) is similar to that of (iii). The main difference is that we compose with $\F_i \F_i$ if $\l_i \le -2$ and precompose with $\F_i \F_i$ if $\l_i \ge -1$. 
\end{proof}

\subsection{Spaces involving divided powers}

In this section the proofs are still independent of the other results in this paper with one exception, we assume that we know $\E_i^2 \cong \E_i^{(2)}[1] \oplus \E_i^{(2)} [-1]$. 

\begin{Lemma}\label{LEM:homE2s}
We have $\dim \End^d(\E_i^{(2)} \1_\l) \le \begin{cases} 1 & \text{ if } d = 0 \\ 0 & \text{ if } d < 0. \end{cases}$
\end{Lemma}
\begin{proof}
The proof is analogous to the one computing $\End^d(\E_i \1_\l)$ from Lemma \ref{LEM:homEs}. 
\end{proof}

\begin{Lemma}\label{LEM:temp}
Suppose $i,j \in I$ with $\la i,j \ra = -1$. Then 
\begin{align}
\label{eq:3} \dim \Hom(\E_i \E_j^{(2)} \1_\l, \E_j^{(2)} \E_i \1_\l \la d \ra) & \le \begin{cases} 1 & \text{ if } d = 2 \\ 0 & \text{ if } d < 2 \end{cases} \\
\label{eq:4} \dim \Hom(\E_j^{(2)} \E_i \1_\l, \E_j^{(2)} \E_i \1_\l \la d \ra) & \le \begin{cases} 1 & \text{ if } d = 0 \\ 0 & \text{ if } d < 0. \end{cases}
\end{align}
In (\ref{eq:3}) equality holds when $d=2$ if and only if $\E_i \E_j^{(2)} \1_\l$ and $\E_j^{(2)} \E_i \1_\l$ are both nonzero. Likewise in (\ref{eq:4}) equality holds when $d=0$ if and only if $\E_j^{(2)} \E_i \1_\l$ is nonzero.
\end{Lemma}
\begin{proof}
The proof is similar to that of Lemma \ref{LEM:Ts1}. For example, suppose $\mu_i + \mu_j \le 0$ with $\mu_i \le 0$ (the case $\mu_i \ge 0$ is similar). Then by adjunction we have 
\begin{align*}
& \Hom(\E_i \E_j^{(2)} \1_\mu, \E_j^{(2)} \E_i \1_\mu \la d \ra) \\
\cong & \Hom(\E_j^{(2)} \1_\mu, (\E_i \1_{\l+2\alpha_j})_R \E_j^{(2)} \E_i \1_\mu \la d \ra) \\
\cong & \Hom(\E_j^{(2)} \1_\mu, \E_j^{(2)} \F_i \E_i \1_\mu \la d+\mu_i-1 \ra) \\
\cong & \Hom(\E_j^{(2)} \1_\mu, \E_j^{(2)} \E_i \F_i \1_\mu) \bigoplus_{r=0}^{-\mu_i-1} \Hom(\E_j^{(2)} \1_\mu, \E_j^{(2)} \1_\mu \la d-2r-2 \ra) \\
\cong & \Hom(\E_j^{(2)} \E_i \1_{\mu-\alpha_i} \la -\mu_i+1 \ra, \E_j^{(2)} \E_i \1_{\mu-\alpha_i}) \bigoplus_{r=0}^{-\mu_i-1} \Hom(\E_j^{(2)} \1_\mu, \E_j^{(2)} \1_\mu \la d-2r-2 \ra).
\end{align*}
The terms in the direct sum on the right side are all zero by Lemma \ref{LEM:homEs}, unless $r=0$ and $d=2$ in which case the term is one dimensional. The left hand term is zero by (\ref{eq:4}) with $\l = \mu-\alpha_i$. Thus (\ref{eq:3}) when $\l=\mu$ follows from (\ref{eq:4}) when $\l = \mu-\alpha_i$. Likewise one can show that (\ref{eq:4}) follows from (\ref{eq:3}) as in the proof of Lemma \ref{LEM:Ts1}. 
\end{proof}

\begin{Lemma}\label{LEM:2A}
Suppose $i,j \in I$ with $\la i,j \ra = -1$. Then assuming $\E_i^{(2)} \E_j \1_\l, \E_j \E_i^{(2)} \1_\l$ and $\E_i \E_j \E_i \1_\l$ are nonzero, the spaces
\begin{align*}
\Hom^d(\E_i^{(2)} \E_j \1_\l, \E_i \E_j \E_i \1_\l) \ \ \text{ and } \ \ \Hom^d(\E_j \E_i^{(2)} \1_\l, \E_i \E_j \E_i \1_\l) \\
\Hom^d(\E_i \E_j \E_i \1_\l, \E_i^{(2)} \E_j \1_\l) \ \ \text{ and } \ \ \Hom^d(\E_i \E_j \E_i \1_\l, \E_j \E_i^{(2)} \1_\l)
\end{align*}
are all zero if $d < 0$ and one-dimensional if $d=0$. 
\end{Lemma}
\begin{proof}
This follows by the usual adjunction formalism. We will always assume $d \le 0$. First, if $\l_i \le 0$ then we have 
\begin{align*}
& \Hom^d(\E_i^{(2)} \E_j \1_\l, \E_i \E_j \E_i \1_\l) \\
\cong & \Hom^d((\E_i \1_{\l+\alpha_j+\alpha_i})_L \E_i^{(2)} \E_j, \E_j \E_i \1_\l) \\
\cong & \Hom^d(\F_i \E_i^{(2)} \E_j \1_\l \la -\l_i-2 \ra, \E_j \E_i \1_\l) \\
\cong & \Hom^d(\E_i^{(2)} \E_j \F_i \1_\l, \E_j \E_i \1_\l \la \l_i + 2 \ra) \bigoplus_{[-\l_i]} \Hom^d(\E_i \E_j \1_\l, \E_j \E_i \1_\l \la \l_i+2 \ra) \\
\cong & \Hom^d(\E_i^{(2)} \E_j, \E_j \E_i \E_i \la 2 \l_i + 1 \ra) \bigoplus_{r=0}^{-\l_i-1} \Hom^d(\E_i \E_j \1_\l, \E_j \E_i \1_\l \la 1-2r \ra).
\end{align*}
By Lemma \ref{LEM:Ts1} all the terms in the sum on the right are zero except when $d=0$ and $r=0$ where the space is one-dimensional. Moreover, by Lemma \ref{LEM:temp}, the left term is zero if $\l_i < 0$ or $d<0$. The case $\l_i=0$ is special since the sum on the right disappears. But now the left hand term is equal to 
$$\Hom^d(\E_i^{(2)} \E_j \1_\l, \E_j \E_i^{(2)} \1_\l \la 2 \ra \oplus \E_j \E_i^{(2)} \1_\l) \cong \begin{cases} \k & \text{ if } d = 0 \\ 0 & \text{ if } d < 0 \end{cases}.$$
Thus, in both cases we get $\dim \Hom^d(\E_i^{(2)} \E_j \1_\l, \E_i \E_j \E_i \1_\l) = \begin{cases} 1 & \text{ if } d = 0 \\ 0 & \text{ if } d < 0. \end{cases}$ 

If $\l_i \ge 0$ the argument is similar except that the first step is 
$$\Hom^d(\E_i^{(2)} \E_j \1_\l, \E_i \E_j \E_i \1_\l) \cong \Hom^d(\E_i^{(2)} \E_j (\E_i \1_\l)_R, \E_i \E_j \1_{\l+\alpha_i}).$$
The other three $\Hom$-space calculations are the same. 
\end{proof}

\begin{Lemma}\label{LEM:3}
If $i,j \in I$ with $\la i,j \ra = -1$ then  
$$\dim \End(\E_i \E_j \E_i \1_\l) \le \dim \End(\E_i^{(2)} \E_j \1_\l) + \dim \End(\E_j \E_i^{(2)} \1_\l).$$ 
\end{Lemma}
\begin{proof}
{\bf Part 1.} Let us first assume $\E_i^{(2)} \E_j \1_\l$ and $\E_j \E_i^{(2)} \1_\l$ are both nonzero. By Lemma \ref{LEM:temp} we need to show that $\dim \End(\E_i \E_j \E_i \1_\l) \le 2$. 

Case 1: $\l_i < -1$. By adjunction we have  
\begin{align}
\nonumber \End(\E_i \E_j \E_i \1_\l) 
\cong & \Hom((\E_i \1_{\l+\alpha_i+\alpha_j})_L \E_i \E_j \E_i, \E_j \E_i \1_\l) \\
\nonumber \cong & \Hom(\F_i \E_i \1_{\l+\alpha_i+\alpha_j} \E_j \E_i \la -\l_i-2 \ra, \E_j \E_i \1_\l) \\
\label{eq:temp16} \cong & \Hom(\E_i \E_j \F_i \E_i \1_\l, \E_j \E_i \1_\l \la \l_i+2 \ra) \bigoplus_{[-\l_i-1]} \Hom(\E_j \E_i \1_\l, \E_j \E_i \1_\l \la \l_i+2 \ra).
\end{align}
Then the right hand sum is equal to $\bigoplus_{r=0}^{-\l_i-2} \Hom(\E_j \E_i \1_\l, \E_j \E_i \1_\l \la -2r \ra)$. By Lemma \ref{LEM:Ts1}, all these terms are zero except when $r=0$ where it is at most one-dimensional. 

On the other hand, the left hand term in (\ref{eq:temp16}) equals 
\begin{align}
\nonumber & \Hom(\E_i \E_j \E_i (\1_{\l-\alpha_i} \F_i), \E_j \E_i \1_\l \la \l_i+2 \ra) \bigoplus_{[-\l_i]} \Hom(\E_i \E_j \1_\l, \E_j \E_i \1_\l \la \l_i+2 \ra) \\
\label{eq:temp17} \cong& \Hom(\E_i \E_j \E_i \1_{\l - \alpha_i}, \E_j \E_i \E_i \1_{\l-\alpha_i} \la 2 \l_i+1 \ra) \bigoplus_{r=0}^{-\l_i-1} \Hom(\E_i \E_j \1_\l, \E_j \E_i \1_\l \la 1 - 2r \ra).
\end{align}
Then by Lemma \ref{LEM:2A} the left hand term is zero while, by Lemma \ref{LEM:Ts1}, all the terms in the right hand sum are zero except when $r=0$ when it is at most one-dimensional. Thus $\dim \End(\E_i \E_j \E_i \1_\l) \le 2$.

Case 2: $\l_i=-1$. This is a special case of the argument above. The right hand sum in (\ref{eq:temp16}) disappears but now the left hand term in (\ref{eq:temp17}) is
$$\dim \Hom(\E_i \E_j \E_i \1_\l, \E_j \E_i^{(2)} \1_\l \la -2 \ra \oplus \E_j \E_i^{(2)} \1_\l) \le 1 $$
where we use Lemmas \ref{LEM:temp} and \ref{LEM:2A} to get this isomorphism. So once again $\dim \End(\E_i \E_j \E_i \1_\l) \le 2$. 

Case 3: $\l_i > -1$. The argument is the same as in Case 1 except that the first step is
$$\End(\E_i\E_j\E_i \1_\l) \cong \Hom(\E_i \E_j, \E_i \E_j \E_i (\E_i \1_\l)_L \1_{\l+\alpha_i}).$$

{\bf Part 2.} It remains to consider the situation when $\E_i^{(2)} \E_j \1_\l$ or $\E_j \E_i^{(2)} \1_\l$ are zero. Suppose $\E_i^{(2)} \E_j \1_\l = 0$ and $\E_j \E_i^{(2)} \1_\l \ne 0$ with $\l_i < -1$ (the other cases are similar). Then the second term in (\ref{eq:temp16}) still contributes at most one to the dimension. On the other hand, the second term in (\ref{eq:temp17}) is now entirely zero because $\E_i \E_j \1_\l = 0$. To see that $\E_i \E_j \1_\l = 0$ we use that $\E_i^{(2)} \E_j \1_\l = 0$ which means
$$0 = \F_i \E_i \E_i \E_j \1_\l \cong \E_i \F_i \E_i \E_j \1_\l \bigoplus_{[-\l_i-1]} \E_i \E_j \1_\l.$$
Thus $\dim \End(\E_i \E_j \E_i \1_\l) \le 1$ which is what we wanted to show.
\end{proof}

\begin{Lemma}\label{LEM:homijk}
Let $i,j,k \in I$ be distinct. Then 
$$\dim \Hom(\E_i \E_j \E_k \1_\l, \E_k \E_j \E_i \1_\l \la d \ra) \le  \begin{cases} 1 & \text{ if } d = - \ell_{ijk} \\ 0 & \text{ if } d < - \ell_{ijk} \end{cases}.$$
If $d = -\ell_{ijk}$ then equality holds if and only if $\1_{\l+\eps_i \alpha_i + \eps_j \alpha_j + \eps_k \alpha_k} \ne 0$ for $\eps_i, \eps_j, \eps_k \in \{0,1\}$. 
\end{Lemma}
\begin{proof}
This computation depends on whether $i,j,k \in I$ are joined by an edge. The most difficult case is when they are all joined by an edge, meaning $\la i,j \ra = \la i,k \ra = \la j,k \ra = -1$ (this is because in this case $\E_i,\E_j$ and $\E_k$ do not commute among each other). We will only deal with this case as the general case is essentially the same (and in fact a bit easier). 

By condition (\ref{co:vanish2}) we have $\ell := \l_i+\l_j+\l_k > 0$. There are several cases to consider depending on whether $\l_i,\l_j,\l_k$ are positive or negative. 

{\bf Case 1:} $(\l_i,\l_j,\l_k) = (+,+,-)$ (meaning that $\l_i,\l_j \ge 0$ and $\l_k \le 0$). First, we have 
\begin{align*}
\Hom(\E_i \E_j \E_k \1_\l, \E_k \E_j \E_i \1_\l \la 3 \ra)
&\cong \Hom(\1_\mu, \E_k \E_j \E_i \1_\l (\E_k \1_\l)_L (\E_j \1_{\l+\alpha_k})_L (\E_i \1_{\alpha_j+\alpha_k})_L \la 3 \ra) \\
&\cong \Hom(\1_\mu, \E_k \E_j \E_i \F_k \F_j \F_i \1_\mu \la -\ell + 3 \ra).
\end{align*}
where $\mu := \l+\alpha_i+\alpha_j+\alpha_k$. 
Next, we can simplify the right hand term as follows 
\begin{align*}
\E_k \E_j \E_i \F_k \F_j \F_i \1_\mu
&\cong \E_k \F_k \E_j \F_j \E_i \F_i \1_\mu \\
&\cong \E_k \F_k \F_j \F_i \E_j \E_i \1_\mu \bigoplus_{[\l_i]} \E_k \F_k \F_j \E_j \1_\mu \bigoplus_{[\l_j]} \E_k \F_k \F_i \E_i \1_\mu \bigoplus_{[\l_i][\l_j]} \E_k \F_k \1_\mu.
\end{align*}
Now 
\begin{align*}
\Hom(\1_\mu, \E_k \F_k \F_j \E_j \1_\mu \la -\ell+3 \ra) 
&\cong \Hom((\E_k \1_{\mu-\alpha_k})_L (\E_j \1_\mu)_R, \F_k \F_j \1_{\mu+\alpha_j} \la -\ell+3 \ra) \\
&\cong \Hom(\F_k \F_j \1_{\mu+\alpha_j}, \F_k \F_j \1_{\mu+\alpha_j} \la -\l_i - 2\l_j + 1 \ra)
\end{align*}
so by Lemmas \ref{LEM:nonvan} and \ref{LEM:Ts1} we have 
\begin{align}
\nonumber \dim \Hom(\1_\mu, \bigoplus_{[\l_i]} \E_k \F_k \F_j \E_j \1_\mu \la -\ell+3 \ra) =& \sum_{d=0}^{\l_i-1} \dim \Hom(\F_k \F_j \1_{\mu+\alpha_j}, \F_k \F_j \1_{\mu+\alpha_j} \la -2\l_j - 2d \ra)  \\
\label{eq:P1} =& \begin{cases} 1 & \text{ if } \l_j = 0 \text{ and } \1_{\mu+\alpha_j}, \1_\mu, \1_{\mu-\alpha_k} \text{ are nonzero }  \\ 0 & \text{ if } \l_j > 0 \text{ or } \l_i = 0. \end{cases}
\end{align}
Notice that if $\l_j = 0$ then $\mu_j=0$ so $\1_{\mu+\alpha_j} \ne 0 \Leftrightarrow \1_{\mu-\alpha_j} \ne 0$. Likewise one can show that 
\begin{equation}\label{eq:P2}
\dim \Hom(\1_\mu, \bigoplus_{[\l_j]} \E_k \F_k \F_i \E_i \1_\mu \la -\ell+3 \ra) =  \begin{cases} 1 & \text{ if } \l_i = 0 \text{ and } \1_{\mu-\alpha_i}, \1_\mu, \1_{\mu-\alpha_k} \text{ are nonzero } \\ 0 & \text{ if } \l_i > 0 \text{ or } \l_j = 0. \end{cases}
\end{equation}
and 
\begin{equation}\label{eq:P3}
\dim \Hom(\1_\mu, \bigoplus_{[\l_i][\l_j]} \E_k \F_k \1_\mu \la -\ell+3 \ra) = \begin{cases} 1 & \text{ if } \l_i, \l_j > 0 \text{ and } \1_\mu, \1_{\mu-\alpha_k} \text{ are nonzero } \\ 0 & \text{ if } \l_i=0 \text{ or } \l_j=0 \end{cases}.
\end{equation}
Finally, to compute $\Hom(\1_\mu, \E_k \F_k \F_j \F_i \E_j \E_i \1_\mu \la -\ell+3 \ra)$ we use that 
$$\F_k \E_k \F_j \F_i \E_j \E_i \1_\mu \cong \E_k \F_k \F_j \F_i \E_j \E_i \1_\mu \bigoplus_{[-\l_k]} \F_j \F_i \E_j \E_i \1_\mu.$$
Looking at the left hand side we have  
\begin{align*}
\Hom(\1_\mu, \F_k \E_k \F_j \F_i \E_j \E_i \1_\mu \la -\ell+3 \ra) 
&\cong \Hom(\1_\mu, \F_k \F_j \F_i \E_k \E_j \E_i \1_\mu \la -\ell+3 \ra) \\
&\cong \Hom(\E_i \E_j \E_k \1_\mu, \E_k \E_j \E_i \1_\mu \la -2 \ell + 3 \ra)
\end{align*}
which, by induction, is zero since (by condition (\ref{co:vanish2})) $\mu_i + \mu_j + \mu_k = \ell > 0$ and $\1_{\l+r(\alpha_i+\alpha_j+\alpha_k)} = 0$ for $r \gg 0$. This means that $\Hom(\1_\mu, \E_k \F_k \F_j \F_i \E_j \E_i \1_\mu \la -\ell+3 \ra) = 0$. 

Finally, (\ref{eq:P1}), (\ref{eq:P2}) and (\ref{eq:P3}) together end up contributing a total of $1$ to the dimension (note that there are several subcases to consider here depending on whether or not $\l_i = 0$ and $\l_j=0$). Thus $\dim \Hom(\E_i \E_j \E_k \1_\l, \E_k \E_j \E_i \1_\l \la 3 \ra) \le 1$ with equality holding if $\1_\mu, \1_{\mu-\alpha_i}, \1_{\mu-\alpha_j}, \1_{\mu-\alpha_k}$ are all nonzero. This is precisely what we needed to prove. 

{\bf Case 2:} $(\l_i,\l_j,\l_k) = (-,+,-)$. Arguing as above we have 
\begin{align*}
\Hom(\E_i \E_j \E_k \1_\l, \E_k \E_j \E_i \1_\l \la 3 \ra) 
&\cong \Hom(\1_\mu, \E_k \E_j \E_i \F_k \F_j \F_i \1_\mu \la -\ell + 3 \ra) \\
&\cong \Hom(\1_\mu, \E_k \F_k \F_j \E_j \E_i \F_i \1_\mu \la -\ell+3 \ra \bigoplus_{[\mu_j]} \E_k \F_k \E_i \F_i \1_\mu \la -\ell+3 \ra).
\end{align*}
On the one hand, by adjunction and Lemma \ref{LEM:Ts1} it follows that 
\begin{equation*}
\dim \Hom(\1_\mu, \bigoplus_{[\mu_j]} \E_k \F_k \E_i \F_i \1_\mu \la -\ell+3 \ra) = \begin{cases} 1 & \text{ if } \l_j > 0 \text{ and } \1_\mu, \1_{\mu-\alpha_i}, \1_{\mu-\alpha_i-\alpha_k} \text{ are nonzero } \\ 0 & \text{ if } \l_j = 0. \end{cases}
\end{equation*}
On the other hand, $\E_k \F_k \F_j \E_j \E_i \F_i \1_\mu$ is a direct summand of $\F_k \E_k \F_j \E_j \F_i \E_i \1_\mu$ and 
$$\Hom(\1_\mu, \F_k \E_k \F_j \E_j \F_i \E_i \1_\mu \la -\ell+3 \ra) \cong \Hom(\E_i \E_j \E_k \1_\mu, \E_k \E_j \E_i \1_\mu \la -2\ell+3 \ra).$$
By induction this is zero, which means that $\Hom(\1_\mu, \E_k \F_k \F_j \E_j \E_i \F_i \1_\mu \la -\ell+3 \ra)=0$. It follows that $\dim \Hom(\E_i \E_j \E_k \1_\l, \E_k \E_j \E_i \1_\l \la 3 \ra) \le 1$ with equality holding if $\1_\mu, \1_{\mu-\alpha_i}, \1_{\mu-\alpha_i-\alpha_k}$ are nonzero. This is again precisely what we wanted to prove. 

{\bf Case 3:} $(\l_i,\l_j,\l_k) = (+,-,+)$. First we have 
$$\E_k \E_j \E_i \F_k \F_j \F_i \1_\mu \cong \F_k \E_k \E_j \F_j \F_i \E_i \1_\mu \bigoplus_{[\l_k]} \E_j \F_j \F_i \E_i \1_\mu \bigoplus_{[\l_i]} \F_k \E_k \E_j \F_j \1_\mu \bigoplus_{[\l_i][\l_k]} \E_j \F_j \1_\mu.$$
Then, arguing as in case 1, we get 
\begin{align*}
\nonumber
\dim \Hom(\1_\mu, \F_k \E_k \E_j \F_j \F_i \E_i \1_\mu \la -\ell+3 \ra &= 0 \\
\dim \Hom(\1_\mu, \bigoplus_{[\l_k]} \E_j \F_j \F_i \E_i \1_\mu \la -\ell+3 \ra) &= \begin{cases} 1 & \text{ if } \l_i = 0 \text{ and } \1_{\mu-\alpha_i}, \1_\mu, \1_{\mu-\alpha_k} \text{ are nonzero } \\ 0 & \text{ if } \l_i > 0 \text{ or } \l_k = 0 \end{cases} \\
\dim \Hom(\1_\mu, \bigoplus_{[\l_i]} \F_k \E_k \E_j \F_j \1_\mu \la -\ell+3 \ra) &= \begin{cases} 1 & \text{ if } \l_k = 0 \text{ and } \1_{\mu-\alpha_j}, \1_\mu, \1_{\mu-\alpha_k} \text{ are nonzero } \\ 0 & \text{ if } \l_k > 0 \text{ or } \l_i = 0 \end{cases} \\
\dim \Hom(\1_\mu, \bigoplus_{[\l_i][\l_k]} \E_j \F_j \1_\mu \la -\ell+3 \ra) &= \begin{cases} 1 & \text{ if } \l_i, \l_k > 0 \text{ and } \1_\mu, \1_{\mu-\alpha_j} \text{ are nonzero } \\ 0 & \text{ if } \l_i = 0 \text{ or } \l_k = 0. \end{cases} 
\end{align*}
Thus we get that 
$$\dim \Hom(\E_i\E_j\E_k \1_\l, \E_k \E_j \E_i \1_\l \la 3 \ra) = \dim \Hom(\1_\mu, \E_k \E_j \E_i \F_k \F_j \F_i \1_\mu \la -\ell+3 \ra) \le 1$$
with equality holding if $\1_\mu, \1_{\mu-\alpha_i}, \1_{\mu-\alpha_j}, \1_{\mu-\alpha_k}$ are nonzero. 

{\bf Other cases.} There are four other cases, namely $(\l_i,\l_j,\l_k)$ equal to $(-,+,+)$, $(+,-,-)$, $(-,-,+)$ and $(+,+,+)$. The first is a consequence of Case 1 by symmetry while the others follow using the same arguments as above. 

{\bf The converse.} Suppose $\Hom(\E_i \E_j \E_k \1_\l, \E_k \E_j \E_i \1_\l) \ne 0$. Then $\E_i \E_j \E_k \1_\l \ne 0$ which means $\1_\l, \1_{\l+\alpha_k}, \1_{\l+\alpha_j+\alpha_k}, \1_{\l+\alpha_i+\alpha_j+\alpha_k}$ are all nonzero. Similarly, $\E_k \E_j \E_i \1_\l \ne 0$ means that $\1_{\l+\alpha_i}, \1_{\l+\alpha_i+\alpha_j}$ are also nonzero. It remains to show that $\1_{\l+\alpha_j}$ and $\1_{\l+\alpha_i+\alpha_k}$ are nonzero. This follows from condition (\ref{co:new}). 
\end{proof}

\begin{Corollary}\label{cor:homijk}
Let $i,j,k \in I$ be distinct. Then 
$$\Hom(\F_k \E_i \E_j \1_\l, \E_j \E_i \F_k \1_\l \la d \ra) \le \begin{cases} 1 & \text{ if } d = -\la i,j \ra \\ 0 & \text{ if } d < - \la i,j \ra \end{cases}$$
and when $d = - \la i,j \ra$ equality holds if and only if $\1_{\l+\eps_i \alpha_i + \eps_j \alpha_j - \eps_k \alpha_k}$ for $\eps_i,\eps_j,\eps_k \in \{0,1\}$. 
\end{Corollary}
\begin{proof}
By adjunction
\begin{align*}
&\Hom(\F_k \E_i \E_j \1_\l, \E_j \E_i \F_k \1_\l \la d \ra)  \\
\cong& \Hom(\E_i \E_j (\1_{\l-\alpha_k} \F_k)_R, (\1_{\l+\alpha_i+\alpha_j-\alpha_k} \F_k)_R \E_j \E_i) \la d \ra) \\
\cong& \Hom(\E_i \E_j \E_k \1_{\l-\alpha_k} \la -\l_k + 2 - 1 \ra, \E_k \E_j \E_i \1_{\l-\alpha_k} \la d - \l_k - \la i,k \ra - \la j,k \ra + 2 - 1 \ra).
\end{align*}
The result now follows from Lemma \ref{LEM:homijk}. 
\end{proof}


\begin{thebibliography}{E-G-S}

\bibitem[C]{C}
S. Cautis, Clasp technology to knot homology via the affine Grassmannian; \textsf{arXiv:1207.2074}

\bibitem[CK1]{CK1}
S. Cautis and J. Kamnitzer, Knot homology via derived categories of coherent sheaves I, $\sl_2$ case, \textit{Duke Math. J.} \textbf{142} (2008), no. 3, 511--588; \textsf{math.AG/0701194}

\bibitem[CK2]{CK2}
S. Cautis and J. Kamnitzer, Knot homology via derived categories of coherent sheaves II, $\sl_m$ case, \textit{Inventiones Math.} \textbf{174} (2008), no. 1, 165--232; \textsf{arXiv:0710.3216}

\bibitem[CK3]{CK3}
S. Cautis and J. Kamnitzer, Braiding via geometric Lie algebra actions, \textit{Compositio Math.} \textbf{148} (2012), no. 2,  464--506; \textsf{arXiv:1001.0619}

\bibitem[CKL1]{CKL1}
S. Cautis, J. Kamnitzer and A. Licata, Categorical geometric skew Howe duality, \textit{Inventiones Math.} \textbf{180} (2010), no. 1, 111--159; {\sf arXiv:0902.1795}

\bibitem[CKL2]{CKL2}
S. Cautis, J. Kamnitzer and A. Licata, Coherent sheaves and categorical $\sl_2$ actions, \textit{Duke Math. J.} \textbf{154} (2010), no. 1, 135--179; {\sf arXiv:0902.1796}

\bibitem[CKL3]{CKL3}
S. Cautis, J. Kamnitzer and A. Licata, Coherent sheaves on quiver varieties and categorification, \textit{Math. Annalen} \textbf{357} (2013), no. 3, 805--854; {\sf arXiv:1104.0352}

\bibitem[CLa]{CLa}
S. Cautis and A. Lauda, Implicit structure in 2-representations of quantum groups, \textit{Selecta Math.} (to appear); {\sf arXiv:1111.1431}

\bibitem[CLi1]{CLi1}
S. Cautis and A. Licata, Heisenberg categorification and Hilbert schemes, \textit{Duke Math. J.} \textbf{161} (2012), no. 13, 2469--2547; {\sf arXiv:1009.5147}

\bibitem[CLi2]{CLi2}
S. Cautis and A. Licata, Vertex operators and 2-representations of quantum affine algebras; {\sf arXiv:1112.6189}

\bibitem[CR]{CR}
J. Chuang and R. Rouquier, Derived equivalences for symmetric groups and $\sl_2$-categorification, \textit{Ann. of Math.} \textbf{167} (2008), no. 1, 245--298; \textsf{math.RT/0407205} 

\bibitem[FK]{FK}
I. Frenkel and V. Kac, Basic representations of affine Lie algebras and dual resonance models, \textit{Invent. math.} \textbf{62} (1980), 23--66.

\bibitem[KR]{KR}
M. Khovanov and L. Rozansky, Matrix factorizations and link homology, \textit{Fund. Math.} \textbf{199} (2008), no. 1, 1--91. 

\bibitem[KL1]{KL1}
M. Khovanov and A. Lauda, A diagrammatic approach to categorification of quantum groups I, \textit{Represent. Theory} \textbf{13} (2009), 309--347; \textsf{arXiv:0803.4121}

\bibitem[KL2]{KL2}
M. Khovanov and A. Lauda, A diagrammatic approach to categorification of quantum groups II, \textit{Trans. Amer. Math. Soc.} \textbf{363} (2011), 2685--2700; \textsf{arXiv:0804.2080}

\bibitem[KL3]{KL3}
M. Khovanov and A. Lauda, A diagrammatic approach to categorification of quantum groups III, \textit{Quantum Topology} \textbf{1}, Issue 1 (2010), 1--92; \textsf{arXiv:0807.3250}

\bibitem[LQR]{LQR}
A. Lauda, H. Queffelec and D. Rose, Khovanov homology is a skew Howe 2-representation of categorified quantum $\sl_m$; \textsf{arXiv:1212.6076}

\bibitem[MSV]{MSV}
M. Mackaay, M. Stosic and P. Vaz, $\sl_N$-link homology ($N \ge 4$) using foams and the Kapustin-Li formula, \textit{Geom. Topol.}, \textbf{13} (2009), no. 2, 1075--1128; \textsf{arXiv:0708.2228}

\bibitem[MS]{MS}
V. Mazorchuk and C. Stroppel, A combinatorial approach to functorial quantum $\sl_k$ knot invariants, \textit{Amer. Jour. Math.} \textbf{131} (2009), no. 6, 1679--1713; \textsf{arXiv:0709.1971}

\bibitem[N]{N}
H. Nakajima, Quiver varieties and finite-dimensional representations of quantum affine algebras. \textit{J. Amer. Math. Soc.} \textbf{14} (2001), no.1, 145--238.

\bibitem[QR]{QR}
D. Rose and H. Queffelec, The $\sl_n$ foam 2-category: a combinatorial formulation of Khovanov-Rozansky homology via categorical skew Howe duality; \textsf{arXiv:1405.5920} 

\bibitem[Ri]{Ri} 
C. M. Ringel, Tame Algebras and Integral Quadratic Forms, \textit{Lecture Notes in Mathematics}, \textbf{1099} (1984), Springer Berlin. 

\bibitem[R]{R}
R. Rouquier, 2-Kac-Moody Algebras; \textsf{arXiv:0812.5023} 

\bibitem[Se]{Se}
G. Segal, Unitary representations of some infinite dimensional groups, \textit{Comm. Math. Phys.} \textbf{80} (1981), 301--342.

\bibitem[Su]{Su}
J. Sussan, Category $\O$ and $\sl_k$ link invariants; \textsf{arXiv:0701045}

\bibitem[W]{W}
H. Wu, A colored $\sl_N$-homology for links in $S^3$; \textsf{arXiv:0907.0695} 

\bibitem[Y]{Y}
Y. Yonezawa, Quantum $(\sl_n, \Lambda V_n)$ link invariant and matrix factorizations, \textit{Nagoya Math. J.} \textbf{204} (2011), 69--123; \textsf{arXiv:0906.0220}

\end{thebibliography}
\end{document}